\documentclass[letterpaper,leqno,12pt]{amsart}
\usepackage{microtype}
\usepackage{amssymb}
\usepackage{amsfonts}
\usepackage{geometry}
\usepackage{hyperref}

\newtheorem{theorem}{Theorem}
\theoremstyle{plain}

\newtheorem{corollary}[theorem]{Corollary}

\newtheorem{lemma}[theorem]{Lemma}

\newtheorem{problem}{Problem}
\newtheorem{proposition}[theorem]{Proposition}
\newtheorem{remark}{Remark}

\numberwithin{equation}{section}
\numberwithin{theorem}{section}  
\input{tcilatex}
\begin{document}
\title[The Klainerman-Machedon Conjecture for $\beta \in \left( 0,1\right) $]%
{Correlation structures, Many-body Scattering Processes and the Derivation
of the Gross-Pitaevskii Hierarchy}
\author{Xuwen Chen}
\address{Department of Mathematics, Brown University, 151 Thayer Street,
Providence, RI 02912}
\email{chenxuwen@math.brown.edu}
\urladdr{http://www.math.brown.edu/\symbol{126}chenxuwen/}
\author{Justin Holmer}
\address{Department of Mathematics, Brown University, 151 Thayer Street,
Providence, RI 02912}
\email{holmer@math.brown.edu}
\urladdr{http://www.math.brown.edu/\symbol{126}holmer/}

\begin{abstract}
We consider the dynamics of $N$ bosons in three dimensions. We assume the
pair interaction is given by $N^{3\beta -1}V(N^{\beta }\cdot )$ . By
studying an associated many-body wave operator, we introduce a BBGKY
hierarchy which takes into account all of the interparticle singular
correlation structures developed by the many-body evolution from the
beginning. Assuming energy conditions on the $N$-body wave function, for $%
\beta \in \left( 0,1\right] $, we derive the Gross-Pitaevskii hierarchy with 
$2$-body interaction. In particular, we establish that, in the $N\rightarrow
\infty $ limit, all $k$-body scattering processes vanishes if $k\geqslant 3$
and thus provide a direct answer to a question raised by Erd\"{o}s, Schlein,
and Yau in \cite{E-S-Y2}. Moreover, this new BBGKY hierarchy shares the
limit points with the ordinary BBGKY hierarchy strongly for $\beta \in
\left( 0,1\right) $ and weakly for $\beta =1$. Since this new BBGKY
hierarchy converts the problem from a two-body estimate to a weaker three-%
body estimate for which we have the estimates to achieve $\beta <1$, it then
allows us to prove that all limit points of the ordinary BBGKY hierarchy
satisfy the space-time bound conjectured by Klainerman and Machedon in \cite%
{KlainermanAndMachedon} for $\beta \in \left( 0,1\right) $.
\end{abstract}

\maketitle
\tableofcontents

\section{Introduction}

A Bose-Einstein condensate (BEC), is a peculiar gaseous state in which
particles of integer spin (bosons) occupy a macroscopic quantum state.
Though the existence of a BEC was first predicted theoretically by Einstein
for non-interacting particles in 1925, it was not verified experimentally
until the Nobel prize winning first observation of Bose-Einstein condensate
(BEC) for interacting atoms in low temperature in 1995 \cite{Anderson, Davis}
using laser cooling techniques. Since then, this new state of matter has
attracted a lot of attention in physics and mathematics as it can be used to
explore fundamental questions in quantum mechanics, such as the emergence of
interference, decoherence, superfluidity and quantized vortices.
Investigating various condensates has become one of the most active areas of
contemporary research.

As in the study of any time-dependent interacting $N$-body system, the main
difficulty in the theory of BEC is that the governing PDE is impossible to
solve or simulate when $N$ is large. For BEC, the time-evolution of a $N$
boson system without trapping in $\mathbb{R}^{3}$ is governed by the
many-body Schr\"{o}dinger equation 
\begin{equation}
i\partial _{t}\psi _{N}=H_{N}\psi _{N}  \label{eqn:N-body Schrodinger}
\end{equation}%
where the $N$-body Hamiltonian is given by 
\begin{equation}
H_{N}=-\sum_{j=1}^{N}\triangle _{x_{j}}+\sum_{1\leqslant i<j\leqslant
N}N^{3\beta -1}V(N^{\beta }(x_{i}-x_{j}))\text{ with }\beta \in \left( 0,1%
\right] \text{.}  \label{eqn:N-body Hamiltonian}
\end{equation}%
Here, $(x_{1},...,x_{N})\in \mathbb{R}^{3N}$ is the position vector of $N$
particles in $\mathbb{R}^{3}$, we choose $\left\Vert \psi _{N}(0)\right\Vert
_{L^{2}(\mathbb{R}^{3N})}=1$, and we assume the interparticle interaction is
given by $N^{3\beta -1}V(N^{\beta }\cdot )$. On the one hand, 
\begin{equation}
V_{N}(\cdot )=N^{3\beta }V(N^{\beta }\cdot )  \label{def:V_N}
\end{equation}%
is an approximation of the Dirac $\delta $-function as $N\rightarrow \infty $
and hence matches the Gross-Pitaevskii description that the many-body effect
should be modeled by an on-site strong self interaction.\footnote{%
From here on out, we consider the $\beta >0$ case solely. For $\beta =0$
(Hartree dynamics), see \cite%
{Frolich,E-Y1,KnowlesAndPickl,RodnianskiAndSchlein,MichelangeliSchlein,GMM1,GMM2,Chen2ndOrder,Ammari2,Ammari1,LChen}
.} On the other hand, if we denote by $\func{scat}(W)$ the 3D scattering
length of the potential $W$, then we have 
\begin{equation*}
N\func{scat}(N^{-1}V_{N}(\cdot ))\sim 1
\end{equation*}%
which is the Gross-Pitaevskii scaling condition introduced by Lieb,
Seiringer and Yngvason in \cite{Lieb2}. In the current experiments, we have $%
N\sim 10^{4}$ which already makes equation (\ref{eqn:N-body Schrodinger})
unrealistic to solve. In fact, according to the references in \cite{Lieb2},
the largest system one could simulate at the moment has $N\sim 10^{2}$.
Hence, it is necessary to find reductions or approximations.

It is widely believed that the mean-field approximation / limit of equation
( \ref{eqn:N-body Schrodinger}) is given by the cubic nonlinear Schr\"{o}%
dinger equation (NLS) 
\begin{equation}
i\partial _{t}\phi =-\triangle \phi +c\left\vert \phi \right\vert ^{2}\phi ,
\label{eqn:cubic NLS}
\end{equation}%
where the coupling constant $c$ is exactly given by $8\pi N\func{scat}%
(N^{-1}V_{N}(\cdot ))$. That is, if we define the $k$-particle marginal
densities associated with $\psi _{N}$ by 
\begin{equation}
\gamma _{N\,}^{(k)}(t,\mathbf{x}_{k};\mathbf{x}_{k}^{\prime })=\int \psi
_{N}(t,\mathbf{x}_{k},\mathbf{x}_{N-k})\overline{\psi _{N}}(t,\mathbf{x}%
_{k}^{\prime },\mathbf{x}_{N-k})d\mathbf{x}_{N-k},\quad \mathbf{x}_{k},%
\mathbf{x}_{k}^{\prime }\in \mathbb{R}^{3k},  \label{E:marginal}
\end{equation}%
and assume 
\begin{equation*}
\gamma _{N\,}^{(k)}(0,\mathbf{x}_{k},\mathbf{x}_{k}^{\prime })\sim
\dprod\limits_{j=1}^{k}\phi _{0}(x_{j})\bar{\phi}_{0}(x_{j}^{\prime })\text{
as }N\rightarrow \infty
\end{equation*}%
where $\mathbf{x}_{k}=(x_{1},...,x_{j})\in \mathbb{R}^{3k}$, then we have
the propagation of chaos, namely, 
\begin{equation}
\gamma _{N\,}^{(k)}(t,\mathbf{x}_{k},\mathbf{x}_{k}^{\prime })\sim
\dprod\limits_{j=1}^{k}\phi (t,x_{j})\bar{\phi}(t,x_{j}^{\prime })\text{ as }%
N\rightarrow \infty  \label{relation:propogation of chaos}
\end{equation}%
and $\phi (t,x_{j})$ is given by (\ref{eqn:cubic NLS}) subject to the
initial $\phi (0,x_{j})=\phi _{0}(x_{j})$. Naturally, to prove (\ref%
{relation:propogation of chaos}), one studies the $N\rightarrow \infty $
limit of the Bogoliubov--Born--Green--Kirkwood--Yvon (BBGKY) hierarchy of
the many-body system (\ref{eqn:N-body Schrodinger}) satisfied by $\left\{
\gamma _{N\,}^{(k)}\right\} $: 
\begin{equation}
i\partial _{t}\gamma _{N}^{(k)}+\left[ \triangle _{\mathbf{x}_{k}},\gamma
_{N}^{(k)}\right] =\begin{aligned}[t] &\frac{1}{N}\sum_{1\leqslant
i<j\leqslant k}\left[ V_{N}\left( x_{i}-x_{j}\right) ,\gamma
_{N}^{(k)}\right] \\ &+\frac{N-k}{N}
\sum_{j=1}^{k}\limfunc{Tr}\nolimits_{k+1}\left[ V_{N}\left(
x_{j}-x_{k+1}\right) ,\gamma _{N}^{(k+1)}\right] \end{aligned}
\label{hierarchy:fake BBGKY hierarchy in operator form}
\end{equation}%
if we do not distinguish $\gamma _{N}^{(k)}$ as a kernel and the operator it
defines. Here the operator $V_{N}\left( x\right) $ represents multiplication
by the function $V_{N}\left( x\right) $ and $\limfunc{Tr}\nolimits_{k+1}$
means taking the $k+1$ trace, for example, 
\begin{equation*}
\limfunc{Tr}\nolimits_{k+1}V_{N}\left( x_{j}-x_{k+1}\right) \gamma
_{N}^{(k+1)}=\int V_{N}\left( x_{j}-x_{k+1}\right) \gamma _{N}^{(k+1)}(t,%
\mathbf{x}_{k},x_{k+1};\mathbf{x}_{k}^{\prime },x_{k+1})dx_{k+1}\text{.}
\end{equation*}%
Such an approach for deriving mean-field type equations by studying the
limit of the BBGKY hierarchy was proposed by Kac in the classical setting
and demonstrated by Landford's work on the Boltzmann equation. In the
current quantum setting, it was suggested by Spohn \cite{Spohn} and has been
proven to be successful by Erd\"{o}s, Schlein, and Yau in their fundamental
papers \cite{E-S-Y1,E-S-Y2,E-S-Y5, E-S-Y3} which have inspired many works by
many authors \cite%
{KlainermanAndMachedon,Kirpatrick,TChenAndNP,ChenAnisotropic,TChenAndNPSpace-Time,Chen3DDerivation,SchleinNew,C-H3Dto2D, C-H2/3,GM1,C-HFocusing,Sohinger3,C-HFocusingII}
.

This paper, like the aforementioned work, is inspired by the work of Erd\"{o}%
s, Schlein, and Yau. The first main part of this paper deals with a problem
raised on \cite[p.516]{E-S-Y2}. To motivate and state the problem, we first
notice the formal limit of hierarchy (\ref{hierarchy:fake BBGKY hierarchy in
operator form}): 
\begin{equation}
i\partial _{t}\gamma ^{(k)}+\left[ \triangle _{\mathbf{x}_{k}},\gamma ^{(k)}%
\right] =b_{0}\sum_{j=1}^{k}\limfunc{Tr}\nolimits_{k+1}\left[ \delta
(x_{j}-x_{k+1}),\gamma ^{(k+1)}\right]  \label{hierarchy:GP with b_0}
\end{equation}%
where 
\begin{equation*}
b_{0}=\int_{\mathbb{R}^{3}}V(x)dx.
\end{equation*}%
We make such an observation because $V_{N}(\cdot )\rightarrow \left( \int_{%
\mathbb{R}^{3}}V(x)dx\right) \delta (\cdot )$. If we plug 
\begin{equation}
\gamma ^{(k)}(t,\mathbf{x}_{k},\mathbf{x}_{k}^{\prime
})=\dprod\limits_{j=1}^{k}\phi (t,x_{j})\bar{\phi}(t,x_{j}^{\prime })
\label{eqn:product form soln}
\end{equation}%
into (\ref{hierarchy:GP with b_0}) and assume $\phi $ solves (\ref{eqn:cubic
NLS}), then (\ref{eqn:product form soln}) is a solution to (\ref%
{hierarchy:GP with b_0}) if and only if the coupling constant $c$ in (\ref%
{eqn:cubic NLS}) equals to $b_{0}$. Since 
\begin{equation*}
8\pi \lim_{N\rightarrow \infty }N\func{scat}(N^{-1}V_{N}(\cdot ))=b_{0}\text{
for }\beta \in \left( 0,1\right) ,
\end{equation*}%
the formal limit (\ref{hierarchy:GP with b_0}) checks the prediction. It
also has been proven in \cite{E-S-Y2} for $\beta \in (0,1/2)$. However, this
formal limit does not meet the prediction when $\beta =1$ because 
\begin{equation*}
8\pi N\func{scat}(N^{-1}V_{N}(\cdot ))=8\pi \func{scat}(V)\equiv 8\pi a_{0}%
\text{ for }\beta =1
\end{equation*}%
which is usually a number smaller than $b_{0}$. In \cite{E-S-Y1,E-S-Y5,
E-S-Y3}, Erd\"{o}s, Schlein and Yau have established rigorously that the
real limit of the BBGKY hierarchy (\ref{hierarchy:fake BBGKY hierarchy in
operator form}) associated with (\ref{eqn:N-body Schrodinger}) matches the
prediction and is given by 
\begin{equation}
i\partial _{t}\gamma ^{(k)}+\left[ \triangle _{\mathbf{x}_{k}},\gamma ^{(k)}%
\right] =8\pi a_{0}\sum_{j=1}^{k}\limfunc{Tr}\nolimits_{k+1}\left[ \delta
(x_{j}-x_{k+1}),\gamma ^{(k+1)}\right] .  \label{hierarchy:GP with a_0}
\end{equation}%
The reasoning given is that one has to take into account the correlation
between the particles. To be specific, as in \cite%
{Lieb2,E-S-Y1,E-S-Y2,E-S-Y3}, let $w_{0}$ be the solution to 
\begin{eqnarray*}
(-\triangle +\frac{1}{2}N^{\beta -1}V)w_{0}\left( x\right) &=&\frac{1}{2}V,
\\
\lim_{\left\vert x\right\vert \rightarrow \infty }w_{0}(x) &=&0.
\end{eqnarray*}%
We scale $w_{0}$ by 
\begin{equation*}
w_{N}(x)=N^{\beta -1}w_{0}\left( N^{\beta }x\right)
\end{equation*}%
so that $w_{N}$ is the solution to 
\begin{eqnarray}
(-\triangle +\tfrac{1}{2N}V_{N})(1-w_{N}\left( x\right) ) &=&0,
\label{eqn:zero scattering} \\
\lim_{\left\vert x\right\vert \rightarrow \infty }w_{N}(x) &=&0.  \notag
\end{eqnarray}%
The papers \cite{E-S-Y1,E-S-Y5, E-S-Y3} then suggest that, instead of
considering the limit of hierarchy (\ref{hierarchy:fake BBGKY hierarchy in
operator form}) directly, one should investigate the limit of the following
hierarchy 
\begin{eqnarray}
&&i\partial _{t}\gamma _{N}^{(k)}+\triangle _{\mathbf{x}_{k}}\gamma
_{N}^{(k)}-\triangle _{\mathbf{x}_{k}^{\prime }}\gamma _{N}^{(k)}
\label{BBGKY:ESY REAL} \\
&=&\frac{1}{N}\sum_{1\leqslant i<j\leqslant k}\left( \tilde{V}%
_{N}(x_{i}-x_{j})\gamma _{N,i,j}^{(k)}-\tilde{V}_{N}(x_{i}^{\prime
}-x_{j}^{\prime })\gamma _{N,i^{\prime },j^{\prime }}^{(k)}\right)  \notag \\
&&+\frac{N-k}{N}\sum_{j=1}^{k}\limfunc{Tr}\nolimits_{k+1}\left( \tilde{V}%
_{N}(x_{j}-x_{k+1})\gamma _{N,j,k+1}^{(k+1)}-\tilde{V}_{N}(x_{j}^{\prime
}-x_{k+1}^{\prime })\gamma _{N,j^{\prime },\left( k+1\right) ^{\prime
}}^{(k+1)}\right) ,  \notag
\end{eqnarray}%
which has the singular correlations between particles built in. Here 
\begin{equation*}
\tilde{V}_{N}(\cdot )=V_{N}(\cdot )(1-w_{N}(\cdot )),
\end{equation*}%
and 
\begin{equation*}
\gamma _{N,i,j}^{(k)}=\frac{\gamma _{N}^{(k)}}{(1-w_{N}(x_{i}-x_{j}))}\text{%
. }
\end{equation*}%
As $N\rightarrow \infty $, one formally has 
\begin{equation*}
\gamma _{N}^{(k)}\sim \gamma _{N,i,j}^{(k)}\text{,}
\end{equation*}%
and 
\begin{equation*}
\tilde{V}_{N}(\cdot )\rightarrow 8\pi a_{0}\delta (\cdot )\text{,}
\end{equation*}%
hence one obtains (\ref{hierarchy:GP with a_0}) as the limit of the
many-body dynamic (\ref{eqn:N-body Schrodinger}).

One immediate question to this delicate limiting process is: aside from
physical motivation, is there a more mathematical explanation for why (\ref%
{hierarchy:GP with b_0}) is not the limit of (\ref{eqn:N-body Schrodinger})
when $\beta =1$? An answer is that the "usual" energy condition: 
\begin{equation}
\sup_{t}\left( \limfunc{Tr}S^{(k+1)}\gamma _{N}^{(k+1)}+\frac{1}{N}\limfunc{
Tr}S_{1}S_{1^{\prime }}S^{(k)}\gamma _{N}^{(k)}\right) \leqslant C^{k}\text{
for }k\geqslant 0,  \label{energy condition:"usual"}
\end{equation}
where 
\begin{equation*}
S_{j}=(1-\triangle _{x_{j}})^{\frac{1}{2}}\text{ and }S^{(k)}=\dprod
\limits_{j=1}^{k}S_{j}S_{j^{\prime }},
\end{equation*}
first proved in \cite{E-E-S-Y1,E-S-Y2} for $\beta \in \left( 0,\frac{3}{5}
\right) $ and later in \cite%
{Kirpatrick,TChenAndNP,ChenAnisotropic,Chen3DDerivation,C-HFocusing,C-HFocusingII}
, is not true when $\beta =1$. This can be proved by contradiction: assume
that (\ref{energy condition:"usual"}) does hold when $\beta =1$, then with a
simple argument in \cite{Kirpatrick} which is first hinted in \cite{E-S-Y2}
and used in \cite%
{Kirpatrick,TChenAndNP,ChenAnisotropic,Chen3DDerivation,C-HFocusing,C-HFocusingII}
, one easily proves that hierarchy (\ref{hierarchy:fake BBGKY hierarchy in
operator form}) converges to the wrong limit (\ref{hierarchy:GP with b_0})
and reaches a contradiction.

Another immediate but much deeper question is that, if the singular
correlation structure between particles is so crucial, then why would one
only take a pair into account at a time? For example, when considering the
term 
\begin{equation*}
V_{N}(x_{1}-x_{2})\gamma _{N}^{(k)}
\end{equation*}
why would one only put in the singular correlation structure between
particles $x_{1}$ and $x_{2}$ and why not put in the singular correlation
structure between particles $x_{1}$ and $x_{3}$ or $x_{2}$ and $x_{3}$? That
is, why not consider a term like 
\begin{equation*}
\left[ \tilde{V}_{N}(x_{1}-x_{2})(1-w_{N}(x_{1}-x_{3}))\right] \left[ \frac{
\gamma _{N,1,2}^{(k)}}{(1-w_{N}(x_{1}-x_{3}))}\right] ?
\end{equation*}
The above expression corresponds to a three-body interaction. Basically, the
question is: why can this case be dropped? This is actually a problem raised
on \cite[p.516]{E-S-Y2}.

\begin{problem}[{\protect\cite[p.516]{E-S-Y2}}]
\label{Problem:ESY}One should rigorously establish the fact that all
three-body scattering processes are negligible in the limit.
\end{problem}

In the first main part of this paper, we provide a direct answer to Problem %
\ref{Problem:ESY}. We take into account all of the interparticle singular
correlation structures developed by the many-body evolution from the
beginning.\footnote{%
In the Fock space version of the problem, there is another way to insert all
of the correlation structures using the metaplectic representation /
Bogoliubov transform. See \cite{SchleinNew}.} We rigorously establish that,
in the $N\rightarrow \infty $ limit, all $k$-body scattering processes
vanishes if $k\geqslant 3$. To be specific, we have the following theorem.

\begin{theorem}[Main Theorem I]
\label{THM:Convergence}Define 
\begin{equation}
\alpha _{N}^{(k)}(t,\mathbf{x}_{k},\mathbf{x}_{k}^{\prime })\overset{\mathrm{%
\ def}}{=}\left( G_{N}^{(k)}\left( \mathbf{x}_{k}\right) \right) ^{-1}\left(
G_{N}^{(k)}\left( \mathbf{x}_{k}^{\prime }\right) \right) ^{-1}\gamma
_{N}^{(k)}(t,\mathbf{x}_{k},\mathbf{x}_{k}^{\prime }),  \label{def:alpha(k)}
\end{equation}
where 
\begin{equation}
G_{N}^{(k)}\left( \mathbf{x}_{k}\right) \overset{\mathrm{def}}{=}
\prod_{1\leq i<j\leq k}(1-w_{N}(x_{i}-x_{j})).  \label{def:G_N(k)}
\end{equation}
Suppose $\beta \in \left( 0,1\right] $. Assume the energy bound\footnote{%
We remind the readers that the "usual" energy condition (\ref{energy
condition:"usual"}) is not true when $\beta =1$. The energy conditions (\ref%
{bound:energy estimate}) and (\ref{bound:k-energy estimate}) we impose on
Theorems \ref{THM:Convergence} and \ref{THM:KM Bound} have been proven for $%
k=0,1$ or with spatial cut-offs for general $k$ in \cite{E-S-Y3,E-S-Y5}.}: 
\begin{equation}
\sup_{t}\left( \limfunc{Tr}S^{(3)}\alpha _{N}^{(3)}+\frac{1}{N}\limfunc{Tr}
S_{1}S_{1^{\prime }}S^{(3)}\alpha _{N}^{(3)}\right) \leqslant C\text{.}
\label{bound:energy estimate}
\end{equation}
Moreover, denote $\mathcal{L}_{k}^{2}$ the space of Hilbert-Schmidt
operators on $L^{2}(\mathbb{R}^{3k})$. Then every limit point $\Gamma
=\left\{ \gamma ^{(k)}\right\} _{k=1}^{\infty }$ of $\left\{ \Gamma
_{N}(t)=\left\{ \alpha _{N}^{(k)}\right\} _{k=1}^{N}\right\} $ in $%
\bigoplus_{k\geqslant 1}C\left( \left[ 0,T\right] ,\mathcal{L}%
_{k}^{2}\right) $ with respect to the product topology $\tau _{prod}$
(defined in Appendix \ref{appendix:ESYTopology}), if there is any, satisfies
the cubic Gross-Pitaevskii hierarchy: 
\begin{equation}
i\partial _{t}\gamma ^{(k)}=\sum_{j=1}^{k}\left[ -\triangle _{x_{j}},\gamma
^{(k)}\right] +c_{0}\sum_{j=1}^{k}\limfunc{Tr}\nolimits_{k+1}\left[ \delta
(x_{j}-x_{k+1}),\gamma ^{(k+1)}\right] ,  \label{hierarchy:GP}
\end{equation}
where the coupling constant $c_{0}$ is given by 
\begin{equation}
c_{0}= 
\begin{cases}
\int_{\mathbb{R}^{3}}V(x)dx & \text{if }\beta \in \left( 0,1\right) , \\ 
8\pi a_{0} & \text{if }\beta =1.%
\end{cases}
\label{coupling constant:c_0}
\end{equation}
\end{theorem}

An important feature of $\alpha _{N}^{(k)}$ is that, considered as bounded
operators, $\alpha _{N}^{(k)}\ $and $\gamma _{N}^{(k)}$\ share the same $%
N\rightarrow \infty $ limit for $\beta \in \left( 0,1\right) $, if there is
any.\footnote{%
The same thing is weakly true for $\beta =1$ but we omit the proof at the
moment since Theorem \ref{THM:KM Bound} applies only to $\beta <1$.} We will
prove this simple fact in Lemma \ref{Lemma:SameLimit}, \S \ref{Sec:New BBGKY}%
. Hence, Theorem \ref{THM:Convergence} and its proof give us a better
understanding of the limiting process and allow us to solve an open problem,
raised by Klainerman and Machedon in 2008, for $\beta \in \left( 0,1\right) $
in the second main part of this paper. After reading Theorem \ref%
{THM:Convergence}, an alert reader can easily tell that one needs to prove a
uniqueness theorem of solutions to hierarchy (\ref{hierarchy:GP}) before
concluding that equation (\ref{eqn:cubic NLS}) is the mean-field limit to
the $N$-body dynamic (\ref{eqn:N-body Schrodinger}). In the second main part
of this paper, we solve an open problem about an a-priori bound on the limit
points which leads to uniqueness of (\ref{hierarchy:GP}), conjectured by
Klainerman and Machedon \cite{KlainermanAndMachedon} in 2008 for $\beta \in
\left( 0,1\right) $. Though this conjecture was not stated explicitly in 
\cite{KlainermanAndMachedon}, as we will explain after stating Theorem \ref%
{THM:KM Bound}, this Klainerman-Machedon a-priori bound is necessary to
implement Klainerman-Machedon's powerful and flexible approach in the most
involved part of proving the cubic nonlinear Schr\"{o}dinger equation (NLS)
as the $N\rightarrow \infty $ limit of quantum $N$-body dynamics.
Kirkpatrick-Schlein-Staffilani \cite{Kirpatrick} completely solved the $%
\mathbb{T}^{2}$ version of the conjecture with a trace theorem and were the
first to successfully implement such an approach. However, the $\mathbb{R}%
^{3}$ version of the conjecture as stated inside Theorem \ref{THM:KM Bound},
was fully open until recently. T. Chen and Pavlovi\'{c} \cite%
{TChenAndNPSpace-Time} have been able to prove the conjecture for $\beta \in
\left( 0,1/4\right) $. In \cite{Chen3DDerivation}, X.C simplified and
extended the result to the range of $\beta \in $ $\left( 0,2/7\right] .$
X.C. and J.H. \cite{C-H2/3}\ then extended the $\beta \in (0,2/7]$ result by
X.C. to $\beta \in (0,2/3)$. In the second main part of this paper, we prove
it for $\beta \in \left( 0,1\right) $. In particular, away from the $\beta
=1 $ case, the conjecture is now resolved. To be specific, we prove the
following theorem.

\begin{theorem}[Main Theorem II]
\label{THM:KM Bound}Define 
\begin{equation*}
R^{(k)}=\dprod\limits_{j=1}^{k}\left\vert \nabla _{x_{j}}\right\vert
\left\vert \nabla _{x_{j}^{\prime }}\right\vert ,
\end{equation*}
and 
\begin{equation*}
B_{j,k+1}\gamma ^{(k+1)}=\limfunc{Tr}\nolimits_{k+1}\left[ \delta
(x_{j}-x_{k+1}),\gamma ^{(k+1)}\right] ,
\end{equation*}
Suppose $\beta \in \left( 0,1\right) $. Assume the energy bound: 
\begin{equation}
\sup_{t}\left( \limfunc{Tr}S^{(k+1)}\alpha _{N}^{(k+1)}+\frac{1}{N}\limfunc{
Tr}S_{1}S_{1^{\prime }}S^{(k)}\alpha _{N}^{(k)}\right) \leqslant C_{0}^{k+1} 
\text{ for }k\geqslant 0\text{,}  \label{bound:k-energy estimate}
\end{equation}
then every limit point $\Gamma =\left\{ \gamma ^{(k)}\right\} _{k=1}^{\infty
}$ of $\left\{ \Gamma _{N}(t)=\left\{ \alpha _{N}^{(k)}\right\}
_{k=1}^{N}\right\} $ obtained in Theorem \ref{THM:Convergence} (and hence of 
$\left\{ \left\{ \gamma _{N}^{(k)}\right\} _{k=1}^{N}\right\} _{N=1}^{\infty
}$ because they have the same limit), satisfies the space-time bound
conjectured by Klainerman-Machedon \cite{KlainermanAndMachedon} in 2008: 
\begin{equation}
\int_{0}^{T}\left\Vert R^{(k)}B_{j,k+1}\gamma ^{(k+1)}(t,\cdot ,\cdot
)\right\Vert _{L_{\mathbf{x,x}^{\prime }}^{2}}dt\leqslant C^{k}\text{.}
\label{bound:target KM bound}
\end{equation}
In particular, there is only one limit point due to the Klainerman-Machedon
uniqueness theorem \cite[Theorem 1.1]{KlainermanAndMachedon}.
\end{theorem}

In 2007, Erd\"{o}s, Schlein, and Yau obtained the first uniqueness theorem
of solutions \cite[Theorem 9.1]{E-S-Y2} to hierarchy (\ref{hierarchy:GP}).
The proof is surprisingly delicate -- it spans 63 pages and uses complicated
Feynman diagram techniques. The main difficulty is that hierarchy (\ref%
{hierarchy:GP}) is a system of infinitely coupled equations. Briefly, \cite[
Theorem 9.1]{E-S-Y2} is the following:

\begin{theorem}[{Erd\"{o}s-Schlein-Yau uniqueness {\protect\cite[Theorem 9.1]%
{E-S-Y2}}}]
There is at most one nonnegative symmetric operator sequence $\left\{ \gamma
^{(k)}\right\} _{k=1}^{\infty }$ that solves hierarchy (\ref{hierarchy:GP})
subject to the energy condition 
\begin{equation}
\sup_{t\in \lbrack 0,T]}\limfunc{Tr}S^{(k)}\gamma ^{(k)}\leqslant C^{k}.
\label{bound:ESYCondition}
\end{equation}
\end{theorem}

In \cite{KlainermanAndMachedon}, based on their null form paper \cite%
{KlainermanMachedonNullForm}, Klainerman and Machedon gave a different
uniqueness theorem of hierarchy (\ref{hierarchy:GP}) in a space different
from that used in \cite[Theorem 9.1]{E-S-Y2}. The proof is shorter (13
pages) than the proof of \cite[Theorem 9.1]{E-S-Y2}. Briefly, \cite[Theorem
1.1]{KlainermanAndMachedon} is the following:

\begin{theorem}[{Klainerman-Machedon uniqueness {\protect\cite[Theorem 1.1]%
{KlainermanAndMachedon}}}]
There is at most one symmetric operator sequence $\left\{ \gamma
^{(k)}\right\} _{k=1}^{\infty }$ that solves hierarchy (\ref{hierarchy:GP})
subject to the space-time bound (\ref{bound:target KM bound}).
\end{theorem}

When propagation of chaos (\ref{relation:propogation of chaos}) happens,
condition (\ref{bound:ESYCondition}) is actually 
\begin{equation}
\sup_{t\in \lbrack 0,T]}\left\Vert \left\langle \nabla _{x}\right\rangle
\phi \right\Vert _{L^{2}}\leqslant C,  \label{estimate:ESYforNLS}
\end{equation}%
while condition (\ref{bound:target KM bound}) means 
\begin{equation}
\int_{0}^{T}\left\Vert \left\vert \nabla _{x}\right\vert \left( \left\vert
\phi \right\vert ^{2}\phi \right) \right\Vert _{L^{2}}dt\leqslant C.
\label{estimate:KMforNLS}
\end{equation}%
When $\phi $ satisfies NLS (\ref{eqn:cubic NLS}), both are known. Due to the
Strichartz estimate \cite{Keel-Tao}, (\ref{estimate:ESYforNLS}) implies (\ref%
{estimate:KMforNLS}), that is, condition (\ref{bound:target KM bound}) seems
to be a bit weaker than condition (\ref{bound:ESYCondition}). The proof of 
\cite[Theorem 1.1]{KlainermanAndMachedon} (13 pages) is also considerably
shorter than the proof of \cite[Theorem 9.1]{E-S-Y2} (63 pages). It is then
natural to wonder whether \cite[Theorem 1.1]{KlainermanAndMachedon} provides
a simple proof of uniqueness. To answer such a question it is necessary to
know whether the limit points in Theorem \ref{THM:Convergence} satisfy
condition (\ref{bound:target KM bound}).

Away from curiosity, there are realistic reasons to study the
Klainerman-Machedon bound (\ref{bound:target KM bound}). In the NLS
literature, uniqueness subject to condition (\ref{estimate:ESYforNLS}) is
called unconditional uniqueness while uniqueness subject to condition (\ref%
{estimate:KMforNLS}) is called conditional uniqueness. While the conditional
uniqueness theorems usually come for free with the uniqueness conditions
verified naturally in NLS theory because they are parts of the existence
argument, the unconditional uniqueness theorems usually do not yield any
information of existence. Recently, using a version of the quantum de
Finetti theorem from \cite{Lewin}, T. Chen, Hainzl, Pavlovi\'{c}, and
Seiringer \cite{ChHaPaSe13}\ provided an alternative 33 pages proof to {\cite%
[Theorem 9.1]{E-S-Y2}} and confirmed that it is an unconditional uniqueness
result in the sense of NLS theory.\footnote{%
See also \cite{Sohinger3,C-PUniqueness,HoTaXi14}.} Therefore, the general
existence theory of the Gross-Pitaevskii hierarchy (\ref{hierarchy:GP})
subject to general initial datum has to require that the limits of the BBGKY
hierarchy (\ref{hierarchy:fake BBGKY hierarchy in operator form}) lie in the
space in which the space-time bound (\ref{bound:target KM bound}) holds. See 
\cite%
{TChenAndNPGPWellPosedness1,TChenAndNPGPWellPosedness2,TChenAndNPSpace-Time,TCKT}
.

Moreover, while \cite[Theorem 9.1]{E-S-Y2} is a powerful theorem, it is very
difficult to adapt such an argument to various other interesting and
colorful settings: a different spatial dimension, a three-body interaction
instead of a pair interaction, or the Hermite operator instead of the
Laplacian. The last situation mentioned is physically important. On the one
hand, all the known experiments of BEC use harmonic trapping to stabilize
the condensate \cite{Anderson, Davis,Cornish, Ketterle, Stamper}. On the
other hand, different trapping strength produces quantum behaviors which do
not exist in the Boltzmann limit of classical particles nor in the quantum
case when the trapping is missing and have been experimentally observed \cite%
{Kettle3Dto2DExperiment, FrenchExperiment, Philips, NatureExperiment,
Another2DExperiment}. The Klainerman-Machedon approach applies easily in
these meaningful situations (\cite%
{Kirpatrick,TChenAndNP,ChenAnisotropic,Chen3DDerivation,C-H3Dto2D,Sohinger,C-HFocusing,C-HFocusingII}
). Thus proving the Klainerman-Machedon bound (\ref{bound:target KM bound})
actually helps to advance the study of quantum many-body dynamic and the
mean-field approximation in the sense that it provides a flexible and
powerful tool in 3D.

\subsection{Organization of the Paper\label{Sec:org of paper}}

We will first compute the BBGKY hierarchy satisfied by $\left\{ \alpha
_{N}^{(k)}\right\} $, defined in (\ref{def:alpha(k)}), in \S \ref{Sec:New
BBGKY}. Due to the definition of $\left\{ \alpha _{N}^{(k)}\right\} $, the
BBGKY hierarchy of $\left\{ \alpha _{N}^{(k)}\right\} $, written as (\ref%
{Hierarchy: Real BBGKY at k level}), takes into account all of the singular
correlation structures developed by the many-body evolution from the
beginning. The differences between hierarchy (\ref{Hierarchy: Real BBGKY at
k level}) and hierarchy (\ref{hierarchy:fake BBGKY hierarchy in operator
form}) are obvious: hierarchy (\ref{Hierarchy: Real BBGKY at k level}) for $%
\alpha _{N}^{(k)}$ has $k-$body interactions where $k=2,...,k$, but most
importantly, for the purpose of Theorems \ref{THM:Convergence} and \ref%
{THM:KM Bound}, hierarchy (\ref{Hierarchy: Real BBGKY at k level}) does not
have $2$-body interactions not under an integral sign. We will call the key
new terms the \emph{potential terms}, which consist of three-body
interactions, and the \emph{$k$-body interaction terms}, which consist of $k$%
-body interaction for all $k\geqslant 3$.

With the BBGKY hierarchy satisfied by $\left\{ \alpha _{N}^{(k)}\right\} $
computed in \S \ref{Sec:New BBGKY}, we prove Theorem \ref{THM:Convergence}
in \S \ref{sec:pf of thm1} as a "warm up" first and then establish Theorem %
\ref{THM:KM Bound} in \S \ref{sec:pf of thm2}. The gut of the proof of
Theorems \ref{THM:Convergence} and \ref{THM:KM Bound} is the careful
application of the 3D and 6D retarded endpoint Strichartz estimates \cite%
{Keel-Tao} and the Littlewood-Paley theory.

One of the effects of considering the singular interparticle correlation
structures developed by the many-body evolution is to replace the potential 
\begin{equation}
N^{-1}V_{N}(x_{i}-x_{j})=N^{3\beta -1}V(N^{\beta }(x_{i}-x_{j}))
\label{E:pot1}
\end{equation}%
with the new potential 
\begin{equation}
(\nabla _{x_{\ell }}G_{N,i,\ell })(\nabla _{x_{\ell }}G_{N,j,\ell })\,,\quad
i\neq j,i\neq \ell ,j\neq \ell  \label{E:pot2}
\end{equation}%
(among other terms). (\ref{E:pot2}) could be considered as a three-body
interaction, since it is only nontrivial if all three $x_{i}$, $x_{j}$, and $%
x_{\ell }$ are within $\sim N^{-\beta }$. One might wonder why a three-body
interaction is better then a two-body interaction because a three-body
interaction is more complicated. For the purposes of estimates, the original
potential \eqref{E:pot1} has the behavior 
\begin{equation}
N^{-1}V_{N}(x_{i}-x_{j})\sim N^{3\beta -1}\langle N^{\beta
}(x_{i}-x_{j})\rangle ^{-100}  \label{E:pot1b}
\end{equation}%
For the new potential, we have effectively 
\begin{equation}
(\nabla _{x_{\ell }}G_{N,i,\ell })(\nabla _{x_{\ell }}G_{N,j,\ell })\sim
N^{4\beta -2}\langle N^{\beta }(x_{i}-x_{\ell })\rangle ^{-2}\langle
N^{\beta }(x_{j}-x_{\ell })\rangle ^{-2}  \label{E:pot2b}
\end{equation}%
Note that if $\beta =1$ and $i=j$, then \eqref{E:pot2b} and \eqref{E:pot1b}
are effectively the same, and there is no gain in going from \eqref{E:pot1}
to \eqref{E:pot2}. However, $i\neq j$ in \eqref{E:pot2} and hence %
\eqref{E:pot2}, a three-body interaction, actually offers more localization
than \eqref{E:pot1}, a two-body interaction. It is then natural to use the
6D endpoint Strichartz estimate when one wants to estimate a term like%
\begin{equation*}
\left\Vert \int_{0}^{t_{k}}U^{(k)}(t_{k}-t_{k+1})\left[ (\nabla _{x_{\ell
}}G_{N,i,\ell })(\nabla _{x_{\ell }}G_{N,j,\ell })\alpha _{N}^{(k)}\left(
t_{k+1}\right) \right] dt_{k+1}\right\Vert _{L_{T}^{\infty }L_{\mathbf{x,x}%
^{\prime }}^{2}}.
\end{equation*}%
Here $U^{(k)}(t_{k})=e^{it_{k}\triangle _{\mathbf{x}_{j}}}e^{-it_{k}%
\triangle _{\mathbf{x}_{j}^{\prime }}}$.

Using the Littlewood-Paley theory or frequency localization effectively
gains one derivative in the analysis. That is, we avoid a $N^{\beta }$ in
the estimates. Heuristically speaking, it sort of averages the best and the
worst estimates. Here, the "best" means no derivatives hits $V_{N}$ and the
"worst" means that two derivatives hit $V_{N}$. For example, say one would
like to look at%
\begin{equation}
\left\Vert P_{M}^{1}P_{M}^{2}\left\vert \nabla _{x_{1}}\right\vert
\left\vert \nabla _{x_{2}}\right\vert V_{N}(x_{1}-x_{2})\right\Vert _{L^{2}}.
\label{E:example for LP}
\end{equation}%
Here, $P_{M}^{i}$ is the Littlewood-Paley projection onto frequencies $\sim
M $, acting on functions of $x_{i}\in \mathbb{R}^{3}$. There are two ways to
look at (\ref{E:example for LP}), namely%
\begin{equation*}
M^{2}\left\Vert V_{N}(x_{1}-x_{2})\right\Vert _{L^{2}}
\end{equation*}%
and%
\begin{equation*}
N^{2\beta }\left\Vert P_{M}^{1}P_{M}^{2}\left( V^{\prime \prime }\right)
_{N}(x_{1}-x_{2})\right\Vert _{L^{2}}\text{.}
\end{equation*}%
Then depending on the sizes of $N^{\beta }$ and $M$, one is better than the
other. As we will see in the proof of Theorem \ref{THM:KM Bound} in \S \ref%
{sec:pf of thm2}, such a consideration will effectively avoid a $N^{\beta }$
in the estimates.

\subsection{Acknowledgements}

J.H. was supported in part by NSF grant DMS-1200455.

\section{The BBGKY Hierarchy with Singular Correlation Structure\label%
{Sec:New BBGKY}}

Recall (\ref{def:G_N(k)})%
\begin{equation*}
G_{N}^{(k)}(x_{1},\cdots ,x_{N})=\prod_{1\leq i<j\leq
k}(1-w_{N}(x_{i}-x_{j})),
\end{equation*}%
where $w_{N}$ is defined via (\ref{eqn:zero scattering}). We decompose $%
G_{N}^{(k)}$ as follows: 
\begin{equation*}
G_{N}^{(k)}=\prod_{1\leq i<j\leq k}G_{N,i,j}\,,\qquad
G_{N,i,j}=1-w_{N}(x_{i}-x_{j})
\end{equation*}%
and define the multiplication operator $Y_{N}^{(k)}$ by 
\begin{equation}
((Y_{N}^{(k)})^{-1}\psi _{N})(x_{1},\cdots ,x_{N})\overset{\mathrm{def}}{=}%
G_{N}^{(k)}(x_{1},\cdots ,x_{k})\psi _{N}(x_{1},\cdots ,x_{N}).
\label{E:tilde-Y}
\end{equation}

An immediate property of $Y_{N}^{(k)}$ is the following.

\begin{lemma}
\label{Lemma:SameLimit}Let $\alpha _{N}^{(k)}$ be defined as in (\ref%
{def:alpha(k)}). For $\beta \in \left( 0,1\right) $, $\forall f\in
L^{2}\left( \mathbb{R}^{3k}\right) $,%
\begin{equation*}
\lim_{N\rightarrow \infty }\left\Vert \alpha _{N}^{(k)}\left( f\right)
-\gamma _{N}^{(k)}\left( f\right) \right\Vert _{L^{2}}\leqslant C\left\Vert
f\right\Vert \lim_{N\rightarrow \infty }\left\Vert Y_{N}^{(k)}-1\right\Vert
_{op}=0.
\end{equation*}%
Here $\alpha _{N}^{(k)}\left( f\right) $ and $\gamma _{N}^{(k)}\left(
f\right) $ means the operators $\alpha _{N}^{(k)}$ and $\gamma _{N}^{(k)}$
act on $f$, and $\left\Vert \cdot \right\Vert _{op}$ means the operator norm.
\end{lemma}

\begin{proof}
We have%
\begin{eqnarray*}
\left\Vert \alpha _{N}^{(k)}f-\gamma _{N}^{(k)}f\right\Vert _{L^{2}}
&=&\left\Vert Y_{N}^{(k)}\gamma _{N}^{(k)}Y_{N}^{(k)}f-\gamma
_{N}^{(k)}f\right\Vert _{L^{2}} \\
&\leqslant &\left\Vert Y_{N}^{(k)}\gamma _{N}^{(k)}Y_{N}^{(k)}f-\gamma
_{N}^{(k)}Y_{N}^{(k)}f\right\Vert _{L^{2}}+\left\Vert \gamma
_{N}^{(k)}Y_{N}^{(k)}f-\gamma _{N}^{(k)}f\right\Vert _{L^{2}} \\
&\leqslant &\left\Vert Y_{N}^{(k)}-1\right\Vert _{op}\left\Vert \gamma
_{N}^{(k)}\right\Vert _{op}\left\Vert Y_{N}^{(k)}\right\Vert _{op}\left\Vert
f\right\Vert _{L^{2}}+\left\Vert \gamma _{N}^{(k)}\right\Vert
_{op}\left\Vert Y_{N}^{(k)}-1\right\Vert _{op}\left\Vert f\right\Vert
_{L^{2}}\text{.}
\end{eqnarray*}%
Notice that the Hilbert-Schmidt norm of $\gamma _{N}^{(k)}$ is uniformly
bounded by $1$ because we assume $\left\Vert \psi _{N}(0)\right\Vert _{L^{2}(%
\mathbb{R}^{3N})}=1$. Moreover, since $\beta <1$, we have%
\begin{equation*}
\lim_{N\rightarrow \infty }\left\Vert Y_{N}^{(k)}-1\right\Vert _{op}=0\text{.%
}
\end{equation*}%
In fact, consider%
\begin{equation*}
\int \left\vert \frac{f(x_{1},x_{2})}{1-\omega _{N}(x_{1}-x_{2})}%
-f(x_{1},x_{2})\right\vert ^{2}d\mathbf{x}_{2}=\int \left\vert \frac{\omega
_{N}(x_{1}-x_{2})}{1-\omega _{N}(x_{1}-x_{2})}\right\vert ^{2}\left\vert
f(x_{1},x_{2})\right\vert ^{2}d\mathbf{x}_{2}
\end{equation*}%
where%
\begin{eqnarray*}
\frac{\omega _{N}(x_{1}-x_{2})}{1-\omega _{N}(x_{1}-x_{2})} &=&\frac{%
N^{\beta -1}\omega _{0}(N^{\beta }\left( x_{1}-x_{2}\right) )}{1-N^{\beta
-1}\omega _{0}(N^{\beta }\left( x_{1}-x_{2}\right) )} \\
&\leqslant &CN^{\beta -1}\rightarrow 0\text{ as }N\rightarrow \infty
\end{eqnarray*}%
because $\omega _{0}\in L^{\infty }\left( \mathbb{R}^{3}\right) $ and $\beta
<1$. So we conclude that%
\begin{equation*}
\lim_{N\rightarrow \infty }\left\Vert \alpha _{N}^{(k)}\left( f\right)
-\gamma _{N}^{(k)}\left( f\right) \right\Vert _{L^{2}}\leqslant C\left\Vert
f\right\Vert \lim_{N\rightarrow \infty }\left\Vert Y_{N}^{(k)}-1\right\Vert
_{op}=0.
\end{equation*}
\end{proof}

To compute the BBGKY hierarchy of $\left\{ \alpha _{N}^{(k)}\right\} $, we
first give the following lemma.

\begin{lemma}
We have 
\begin{equation}
H_{N}^{(k)}(Y_{N}^{(k)})^{-1}=(Y_{N}^{(k)})^{-1}(H_{N,0}^{(k)}+A_{N}^{(k)}+E_{N}^{(k)}),
\label{E:approx-wo}
\end{equation}%
where $H_{N,0}^{(k)}$ is the ordinary Laplacian%
\begin{equation*}
H_{N,0}^{(k)}=-\sum_{j=1}^{k}\Delta _{x_{j}},
\end{equation*}%
$A_{N}^{(k)}$ is the zeroth order operator of multiplication by 
\begin{equation*}
-\sum_{\substack{ 1\leq i,j,\ell \leq k  \\ i,j,\ell \text{ distinct}}}\frac{%
\nabla _{x_{\ell }}G_{N,\ell ,i}\cdot \nabla _{x_{\ell }}G_{N,\ell ,j}}{%
G_{N,\ell ,i}\;G_{N,\ell ,j}},
\end{equation*}%
and $E_{N}^{(k)}$ is the first order operator 
\begin{equation*}
2\sum_{\substack{ 1\leq j,\ell \leq k  \\ j\neq \ell }}\frac{\nabla
_{x_{\ell }}G_{N,j,\ell }}{G_{N,j,\ell }}\cdot \nabla _{x_{\ell }}.
\end{equation*}
\end{lemma}

Before proceeding to the proof, let us note that the terms $A_N^{(k)}$ and $%
E_N^{(k)}$ should be thought of as ``error terms''. Indeed, $A_N^{(k)}$
involves only three-body interaction -- it is only nontrivial if $x_i$, $x_j$%
, and $x_\ell$ are within $\sim N^{-\beta}$ of each other.

\begin{proof}
We start with%
\begin{equation}
(-\Delta _{x_{\ell }})\log G_{N}^{(k)}=-\frac{\Delta _{x_{\ell }}G_{N}^{(k)}%
}{G_{N}^{(k)}}+\frac{|\nabla _{x_{\ell }}G_{N}^{(k)}|^{2}}{(G_{N}^{(k)})^{2}}
\label{E:GNderiv1a}
\end{equation}%
Using that 
\begin{equation*}
\frac{\nabla _{x_{\ell }}G_{N}^{(k)}}{G_{N}^{(k)}}=\sum_{1\leq i<j\leq k}%
\frac{\nabla _{x_{\ell }}G_{N,i,j}}{G_{N,i,j}}=\sum_{\substack{ 1\leq j\leq
k  \\ j\neq \ell }}\frac{\nabla _{x_{\ell }}G_{N,\ell ,j}}{G_{N,\ell ,j}},
\end{equation*}%
we can rewrite \eqref{E:GNderiv1a} as 
\begin{equation}
-\frac{\Delta _{x_{\ell }}G_{N}^{(k)}}{G_{N}^{(k)}}=(-\Delta _{x_{\ell
}})\log G_{N}^{(k)}-\left\vert \sum_{\substack{ 1\leq j\leq k  \\ j\neq \ell 
}}\frac{\nabla _{x_{\ell }}G_{N,\ell ,j}}{G_{N,\ell ,j}}\right\vert ^{2}.
\label{E:GNderiv1}
\end{equation}%
On the other hand, we have 
\begin{equation*}
\log G_{N}^{(k)}=\sum_{1\leq i<j\leq k}\log G_{N,i,j},
\end{equation*}%
and hence \eqref{E:GNderiv1a} also reads%
\begin{align*}
(-\Delta _{x_{\ell }})\log G_{N}^{(k)}& =\sum_{1\leq i<j\leq k}\left( -\frac{%
\Delta _{x_{\ell }}G_{N,i,j}}{G_{N,i,j}}+\left\vert \frac{\nabla _{x_{\ell
}}G_{N,i,j}}{G_{N,i,j}}\right\vert ^{2}\right) \\
& =\sum_{\substack{ 1\leq j\leq k  \\ j\neq \ell }}\frac{-\Delta _{x_{\ell
}}G_{N,\ell ,j}}{G_{N,\ell ,j}}+\sum_{\substack{ 1\leq j\leq k  \\ j\neq
\ell }}\left\vert \frac{\nabla _{x_{\ell }}G_{N,\ell ,j}}{G_{N,\ell ,j}}%
\right\vert ^{2}.
\end{align*}%
Plugging this into \eqref{E:GNderiv1} and expanding the square in %
\eqref{E:GNderiv1}, 
\begin{equation*}
-\frac{\Delta _{x_{\ell }}G_{N}^{(k)}}{G_{N}^{(k)}}=\sum_{\substack{ 1\leq
j\leq k  \\ j\neq \ell }}\frac{-\Delta _{x_{\ell }}G_{N,\ell ,j}}{G_{N,\ell
,j}}-2\sum_{\substack{ 1\leq i<j\leq k  \\ i\neq \ell ,\;j\neq \ell }}\frac{%
\nabla _{x_{\ell }}G_{N,\ell ,i}\cdot \nabla _{x_{\ell }}G_{N,\ell ,j}}{%
G_{N,\ell ,i}\;G_{N,\ell ,j}}
\end{equation*}%
We infer from (\ref{eqn:zero scattering}) that $-\Delta G_{N}=-\frac{1}{2}%
N^{-1}V_{N}G_{N}$, so 
\begin{equation*}
-\frac{\Delta _{x_{\ell }}G_{N}^{(k)}}{G_{N}^{(k)}}=-\frac{1}{2N}\sum 
_{\substack{ 1\leq j\leq k  \\ j\neq \ell }}V_{N,\ell ,j}-2\sum_{\substack{ %
1\leq i<j\leq k  \\ i\neq \ell ,\;j\neq \ell }}\frac{\nabla _{x_{\ell
}}G_{N,\ell ,i}\cdot \nabla _{x_{\ell }}G_{N,\ell ,j}}{G_{N,\ell
,i}\;G_{N,\ell ,j}}
\end{equation*}%
Now summing in $\ell $, $1\leq \ell \leq k$, we obtain 
\begin{equation*}
H_{N}^{(k)}G_{N}^{(k)}=-2G_{N}^{(k)}\sum_{\substack{ 1\leq i<j\leq k  \\ %
1\leq \ell \leq k  \\ i\neq \ell ,\;j\neq \ell }}\frac{\nabla _{x_{\ell
}}G_{N,\ell ,i}\cdot \nabla _{x_{\ell }}G_{N,\ell ,j}}{G_{N,\ell
,i}\;G_{N,\ell ,j}},
\end{equation*}%
Here $H_{N}^{(k)}G_{N}^{(k)}$ is considered as $H_{N}^{(k)}$ applied to the
function $G_{N}^{(k)}$. Note that the sum on the right side is perhaps more
intuitively written as 
\begin{equation*}
H_{N}^{(k)}G_{N}^{(k)}=-G_{N}^{(k)}\sum_{\substack{ 1\leq i,j,\ell \leq k 
\\ i,j,\ell \text{ distinct}}}\frac{\nabla _{x_{\ell }}G_{N,\ell ,i}\cdot
\nabla _{x_{\ell }}G_{N,\ell ,j}}{G_{N,\ell ,i}\;G_{N,\ell ,j}}
\end{equation*}%
which implies \eqref{E:approx-wo}.
\end{proof}

With the above Lemma, we compute the BBGKY hierarchy of $\left\{ \alpha
_{N}^{(k)}\right\} $. Applying $Y_{N}^{(k)}$ to the left of the operator
equation \eqref{E:approx-wo}, we obtain 
\begin{equation}
Y_{N}^{(k)}H_{N}^{(k)}(Y_{N}^{(k)})^{-1}=(H_{N,0}^{(k)}+A_{N}^{(k)}+E_{N}^{(k)})
\label{E:approx-wo-2}
\end{equation}%
Thus $Y_{N}^{(k)}$ could be regarded as an approximation to the wave
operator relating $H_{N}^{(k)}$ to $H_{N,0}^{(k)}$, although a more precise
statement is that $Y_{N}^{(k)}$ is an exact wave operator relating $%
H_{N}^{(k)}$ to an approximation of $H_{N,0}^{(k)}$, namely the operator $%
H_{N,0}^{(k)}+A_{N}^{(k)}+E_{N}^{(k)}$. Since $%
H_{N,0}^{(k)}+A_{N}^{(k)}+E_{N}^{(k)}$ is not self-adjoint, the wave
operator $Y_{N}^{(k)}$ is not unitary.

We now work out the BBGKY hierarchy of $\left\{ \alpha _{N}^{(k)}\right\} $.
We will need to compute $Y_{N}^{(k)}[H_{N}^{(k)},\gamma
_{N}^{(k)}]Y_{N}^{(k)}$. To this end, we use the operator property: given
two operators $Y_{1}$, $Y_{2}$, let $\alpha =Y_{1}\gamma Y_{2}^{-1}$, then 
\begin{equation*}
Y_{1}[H,\gamma ]Y_{2}^{-1}=(Y_{1}HY_{1}^{-1})\alpha -\alpha
(Y_{2}HY_{2}^{-1}).
\end{equation*}%
In the above, taking $Y_{1}=Y_{N}^{(k)}$ and $Y_{2}=(Y_{N}^{(k)})^{-1}$, and
applying \eqref{E:approx-wo-2} give 
\begin{equation}
Y_{N}^{(k)}[H_{N}^{(k)},\gamma
_{N}^{(k)}]Y_{N}^{(k)}=(H_{N,0}^{(k)}+A_{N}^{(k)}+E_{N}^{(k)})\alpha -\alpha
(H_{N,0}^{(k)}+A_{N}^{(k)}+(E_{N}^{(k)})^{\ast })  \label{E:conjugate}
\end{equation}%
Moreover, let us introduce the operator $W_{N}^{(k)}$ which acts on any
kernel $K(\mathbf{x}_{k},\mathbf{x}_{k}^{\prime })$ by 
\begin{eqnarray*}
W_{N}^{(k)}K(\mathbf{x}_{k},\mathbf{x}_{k}^{\prime })
&=&[Y_{N}^{(k)}Y_{N}^{(k^{\prime })}K](\mathbf{x}_{k},\mathbf{x}_{k}^{\prime
}) \\
&=&\frac{1}{G_{N}(\mathbf{x}_{k})G_{N}(\mathbf{x}_{k}^{\prime })}K(\mathbf{x}%
_{k},\mathbf{x}_{k}^{\prime }).
\end{eqnarray*}%
With the above notation, the BBGKY hierarchy of equations for the operators $%
\left\{ \alpha _{N}^{(k)}=Y_{N}^{(k)}\gamma _{N}^{(k)}Y_{N}^{(k)}\right\} $
or the corresponding kernels $\left\{ \alpha _{N}^{(k)}(\mathbf{x}_{k},%
\mathbf{x}_{k}^{\prime })=\frac{\gamma _{N}^{(k)}(\mathbf{x}_{k},\mathbf{x}%
_{k}^{\prime })}{G_{N}(\mathbf{x}_{k})G_{N}(\mathbf{x}_{k}^{\prime })}%
\right\} $ (using that $Y_{N}^{(k)}$ is equal to its transpose) is given by 
\begin{eqnarray}
i\partial _{t}\alpha _{N}^{(k)} &=&\left( H_{N,0}^{(k)}-H_{N,0}^{(k^{\prime
})}\right) \alpha _{N}^{(k)}+\left( A_{N}^{(k)}-A_{N}^{(k^{\prime })}\right)
\alpha _{N}^{(k)}+\left( E_{N}^{(k)}-E_{N}^{(k^{\prime })}\right) \alpha
_{N}^{(k)}  \label{E:BBGKY-with-approx-wo} \\
&&+\frac{N-k}{N}\sum_{l=1}^{k}W_{N}^{(k)}B_{N,l,k+1}(W_{N}^{\left(
k+1\right) })^{-1}\alpha _{N}^{(k+1)}  \notag
\end{eqnarray}%
where%
\begin{eqnarray}
&&(W_{N}^{(k)}B_{N,l,k+1}(W_{N}^{\left( k+1\right) })^{-1}\alpha
_{N}^{(k+1)})(\mathbf{x}_{k};\mathbf{x}_{k}^{\prime })  \label{E:mod-int-1}
\\
&=&\int_{\mathbb{R}^{3}}\left( V_{N}(x_{l}-x_{k+1})-V_{N}(x_{l}^{\prime
}-x_{k+1})\right) \dprod\limits_{j=1}^{k}G_{N,j,k+1}G_{N,j^{\prime
},k+1}\alpha _{N}^{(k+1)}(...)dx_{k+1}  \notag
\end{eqnarray}%
where $(\cdots )$ is $(x_{1},\ldots ,x_{k},x_{k+1};x_{1}^{\prime },\cdots
,x_{k}^{\prime },x_{k+1})$.

We will decompose the terms in \eqref{E:mod-int-1} to properly set up the
Duhamel-Born series. Let 
\begin{equation*}
L_{N,\ell ,k+1}+1\overset{\mathrm{def}}{=}G_{N,\ell
,k+1}^{-1}\prod_{j=1}^{k}G_{N,j,k+1}G_{N,j^{\prime },k+1}=G_{N,\ell ^{\prime
},k+1}\prod_{\substack{ 1\leq j\leq k  \\ j\neq \ell }}G_{N,j,k+1}G_{N,j^{%
\prime },k+1},
\end{equation*}%
\begin{equation*}
L_{N,\ell ^{\prime },k+1}+1\overset{\mathrm{def}}{=}G_{N,\ell ^{\prime
},k+1}^{-1}\prod_{j=1}^{k}G_{N,j,k+1}G_{N,j^{\prime },k+1}=G_{N,\ell
,k+1}\prod_{\substack{ 1\leq j\leq k  \\ j\neq \ell }}G_{N,j,k+1}G_{N,j^{%
\prime },k+1}.
\end{equation*}%
Here $L$ stands for localization. Also let 
\begin{equation*}
\tilde{V}_{N}(x)=V_{N}\left( x\right) \left( 1-w_{N}\left( x\right) \right)
\end{equation*}%
so that 
\begin{equation*}
\tilde{V}_{N}(x_{l}-x_{k+1})=V_{N}(x_{l}-x_{k+1})G_{N,\ell ,k+1}
\end{equation*}%
\begin{equation*}
\tilde{V}_{N}(x_{l}^{\prime }-x_{k+1})=V_{N}(x_{l}^{\prime
}-x_{k+1})G_{N,\ell ^{\prime },k+1}
\end{equation*}%
Then%
\begin{eqnarray}
&&(W_{N}^{(k)}B_{N,l,k+1}(W_{N}^{\left( k+1\right) })^{-1}\alpha
_{N}^{(k+1)})(\mathbf{x}_{k};\mathbf{x}_{k}^{\prime })  \label{E:mod-int-2}
\\
&=&\int_{\mathbb{R}^{3}}\tilde{V}_{N}(x_{l}-x_{k+1})\left(
L_{N,l,k+1}+1\right) \alpha _{N}^{(k+1)}(...)dx_{k+1}  \notag \\
&&-\int_{\mathbb{R}^{3}}\tilde{V}_{N}(x_{l}^{\prime }-x_{k+1})\left(
L_{N,l^{\prime },k+1}+1\right) \alpha _{N}^{(k+1)}(...)dx_{k+1}  \notag
\end{eqnarray}%
Separate "the $k$-body part" and "the 2-body part": 
\begin{equation}  \label{E:B-many-decomp}
\tilde{B}_{N,\mathrm{many}}^{(k+1)}\alpha _{N}^{(k+1)} =\sum_{l=1}^{k}\tilde{%
B}_{N,\mathrm{many},l,k+1}\alpha _{N}^{(k+1)},
\end{equation}
where 
\begin{equation}  \label{def:B,loc}
B_{N,\mathrm{many},l,k+1}\alpha_N^{(k+1)} \overset{\mathrm{def}}{=} %
\begin{aligned}[t] &\frac{N-k}{N}
\int_{\mathbb{R}^{3}}(\tilde{V}_{N}(x_{l}-x_{k+1})L_{N,l,k+1} \\ & \qquad
\qquad -\tilde{V}_{N}(x_{l}^{\prime}-x_{k+1})L_{N,l^{\prime },k+1})\alpha
_{N}^{(k+1)}(\cdots )\,dx_{k+1} \end{aligned}
\end{equation}
and 
\begin{eqnarray}
&&\tilde{B}_{N}^{(k+1)}\alpha _{N}^{(k+1)}  \label{def:real B} \\
&\equiv &\frac{N-k}{N}\sum_{l=1}^{k}\int_{\mathbb{R}^{3}}(\tilde{V}%
_{N}(x_{l}-x_{k+1})-\tilde{V}_{N}(x_{l}^{\prime }-x_{k+1}))\alpha
_{N}^{(k+1)}(\cdots )\,dx_{k+1}  \notag \\
&\equiv &\sum_{l=1,}^{k}\tilde{B}_{N,l,k+1}\alpha _{N}^{(k+1)},  \notag
\end{eqnarray}%
so that 
\begin{equation*}
\frac{N-k}{N}\sum_{\ell =1}^{k}W_{N}^{(k)}B_{N,\ell
,k+1}(W_{N}^{(k+1)})^{-1}=\tilde{B}_{N,many}^{(k+1)}+\tilde{B}_{N}^{(k+1)}
\end{equation*}%
The operator $\tilde{B}_{N,many}$ will give rise to the $k$\emph{-body
interaction part} and $\tilde{B}_{N}^{(k)}$ will give rise to the \emph{%
interaction part} in the Duhamel-Born series below.

Finally, introduce the operator 
\begin{equation*}
\tilde{V}_{N}^{(k)}\alpha _{N}^{(k)}=(A_{N}^{(k)}-A_{N}^{(k)^{\prime
}})\alpha _{N}^{(k)}+(E_{N}^{(k)}-E_{N}^{(k)^{\prime }})\alpha _{N}^{(k)}
\end{equation*}%
which will give rise to the \emph{potential part} in the Duhamel-Born series
below.

From \eqref{E:BBGKY-with-approx-wo}, 
\begin{eqnarray}
\alpha _{N}^{(k)}(t_{k}) &=&U^{(k)}(t_{k})\alpha
_{N,0}^{(k)}-i\int_{0}^{t_{k}}U^{(k)}(t_{k}-t_{k+1})\tilde{V}%
_{N}^{(k)}\alpha _{N}^{(k)}\left( t_{k+1}\right) dt_{k+1}
\label{Hierarchy: Real BBGKY at k level} \\
&&-i\int_{0}^{t_{k}}U^{(k)}(t_{k}-t_{k+1})\tilde{B}_{N,many}^{(k+1)}\alpha
_{N}^{(k+1)}\left( t_{k+1}\right) dt_{k+1}  \notag \\
&&-i\int_{0}^{t_{k}}U^{(k)}(t_{k}-t_{k+1})\tilde{B}_{N}^{(k+1)}\alpha
_{N}^{(k+1)}\left( t_{k+1}\right) dt_{k+1}  \notag \\
&\equiv &FP^{k,0}+PP^{k,0}+KIP^{k,0}+IP^{k,0},  \notag
\end{eqnarray}%
Here, $U^{(k)}(t_{k})=e^{it_{k}\triangle _{\mathbf{x}_{j}}}e^{-it_{k}%
\triangle _{\mathbf{x}_{j}^{\prime }}}$, $FP^{k,0}$ stands for the free part
of $\alpha _{N}^{(k)}$ with coupling level $0$, $PP^{k,0}$ stands for the
potential part of $\alpha _{N}^{(k)}$ with coupling level $0$, $KIP^{k,0}$
stands for the $k$-body interaction part of $\alpha _{N}^{(k)}$ with
coupling level $0$, and $IP^{k,0}$ stands for the $2$-body interaction part
of $\alpha _{N}^{(k)}$ with coupling level $0$. We will use this notation
for the rest of the paper.

\begin{remark}
In the case $\beta =1$, $\tilde{B}_{N,l,k+1}$ is where $8\pi a_{0}$ shows
up. In fact 
\begin{equation*}
G_{N,\ell ,k+1}V_{N,\ell ,k+1}\rightarrow 8\pi a_{0}\delta (x_{\ell
}-x_{k+1})
\end{equation*}%
as shown in \cite{E-S-Y3}.
\end{remark}

\section{Proof of Theorem \protect\ref{THM:Convergence}\label{sec:pf of thm1}%
}

We prove Theorem \ref{THM:Convergence} as a warm up to the proof of Theorem %
\ref{THM:KM Bound}. Here "warm up" means that we do not need to iterate (\ref%
{Hierarchy: Real BBGKY at k level}) many times to get a good enough decay in
time for the interaction part and do not need to use the Littlewood-Paley
theory or the $X_{0,b}$ spaces.

To prove Theorem \ref{THM:Convergence}, we prove that hierarchy (\ref%
{Hierarchy: Real BBGKY at k level}) converges to hierarchy (\ref%
{hierarchy:GP}) which written in the integral form is 
\begin{equation}
\gamma ^{(k)}(t_{k})=U^{(k)}(t_{k})\gamma
_{0}^{(k)}-ic_{0}\int_{0}^{t_{k}}U^{(k)}(t_{k}-t_{k+1})\limfunc{Tr}%
\nolimits_{k+1}\left[ \delta (x_{j}-x_{k+1}),\gamma ^{(k+1)}\left(
t_{k+1}\right) \right] dt_{k+1}.  \label{hierarchy:GP integral}
\end{equation}%
It has been proven in \cite%
{AGT,E-E-S-Y1,E-S-Y1,E-S-Y3,E-S-Y2,E-S-Y5,Kirpatrick,TChenAndNP,C-H3Dto2D}
that, provided that the energy bound (\ref{bound:energy estimate}) holds,
the 1st term and the last term on the right handside of (\ref{Hierarchy:
Real BBGKY at k level}) do converge to the right hand side of (\ref%
{hierarchy:GP integral}) weak*-ly in $L_{T}^{\infty }\mathcal{L}^{1}$. In
particular, it is proved that, as trace class operators%
\begin{eqnarray*}
&&\int_{0}^{t_{k}}U^{(k)}(t_{k}-t_{k+1})\tilde{B}_{N,j,k+1}\alpha
_{N}^{(k+1)}\left( t_{k+1}\right) dt_{k+1} \\
&\rightharpoonup &\left( \lim_{N\rightarrow \infty }\int \tilde{V}%
_{N}(x)dx\right) \int_{0}^{t_{k}}U^{(k)}(t_{k}-t_{k+1})B_{j,k+1}\gamma
^{(k+1)}\left( t_{k+1}\right) dt_{k+1}\text{ weak*}
\end{eqnarray*}%
where $\lim_{N\rightarrow \infty }\int \tilde{V}_{N}(x)dx$ is exactly the $%
c_{0}$ defined in (\ref{coupling constant:c_0}). So we only need to prove
the following two estimates:%
\begin{eqnarray}
\left\Vert PP^{k,0}\right\Vert _{L_{T}^{\infty }\mathcal{L}^{2}}
&\rightarrow &0\text{ as }N\rightarrow \infty
\label{estimate:PP_0 goes to zero} \\
\left\Vert \mathnormal{KIP}^{k,0}\right\Vert _{L_{T}^{\infty }\mathcal{L}%
^{2}} &\rightarrow &0\text{ as }N\rightarrow \infty .
\label{estimate:KIP_0 goes to zero}
\end{eqnarray}%
where 
\begin{eqnarray*}
PP^{k,0}(t_{k}) &=&-i\int_{0}^{t_{k}}U^{(k)}(t_{k}-t_{k+1})\tilde{V}%
_{N}^{(k)}\alpha _{N}^{(k)}\left( t_{k+1}\right) dt_{k+1}, \\
\mathnormal{KIP}^{k,0}(t_{k}) &=&-i\int_{0}^{t_{k}}U^{(k)}(t_{k}-t_{k+1})%
\tilde{B}_{N,many}^{(k+1)}\alpha _{N}^{(k+1)}\left( t_{k+1}\right) dt_{k+1}.
\end{eqnarray*}%
Before delving into the proof, we remark that condition (\ref{bound:k-energy
estimate}) implies that%
\begin{equation*}
\sup_{t}\left( \left\Vert S^{(k+1)}\alpha _{N}^{(k+1)}\right\Vert _{L_{%
\mathbf{x}_{k},\mathbf{x}_{k}^{\prime }}^{2}}+\frac{1}{\sqrt{N}}\left\Vert
S_{1}S^{(k)}\alpha _{N}^{(k)}\right\Vert _{L_{\mathbf{x}_{k},\mathbf{x}%
_{k}^{\prime }}^{2}}\right) \leqslant C_{0}^{k+1}
\end{equation*}%
In fact, consider the second term for $k=2$:%
\begin{eqnarray*}
&&\frac{1}{\sqrt{N}}\left\Vert S_{1}^{2}S_{2}S_{1^{\prime }}S_{2^{\prime
}}\alpha _{N}^{(2)}\right\Vert _{L_{\mathbf{x}_{1},\mathbf{x}_{1}^{\prime
}}^{2}} \\
&=&\frac{1}{\sqrt{N}}\left( \int \left\vert \int S_{1}^{2}S_{2}\left( \frac{%
\psi _{N}(\mathbf{x}_{2},\mathbf{x}_{N-1})}{G_{N}^{(2)}(\mathbf{x}_{2})}%
\right) S_{1^{\prime }}S_{2^{\prime }}\frac{\psi _{N}(\mathbf{x}_{2}^{\prime
},\mathbf{x}_{N-2})}{G_{N}^{(k)}(\mathbf{x}_{2}^{\prime })}d\mathbf{x}%
_{N-2}\right\vert ^{2}d\mathbf{x}_{2}\right) ^{\frac{1}{2}}
\end{eqnarray*}%
Cauchy-Schwarz in $d\mathbf{x}_{N-2}$, 
\begin{eqnarray*}
&\leqslant &\frac{1}{\sqrt{N}}\left( \int \left\vert S_{1}^{2}S_{2}\left( 
\frac{\psi _{N}(\mathbf{x}_{2},\mathbf{x}_{N-1})}{G_{N}^{(2)}(\mathbf{x}_{2})%
}\right) \right\vert ^{2}d\mathbf{x}_{N}\right) ^{\frac{1}{2}}\left( \int
\left\vert S_{1}S_{2}\frac{\psi _{N}(\mathbf{x}_{2},\mathbf{x}_{N-2})}{%
G_{N}^{(k)}(\mathbf{x}_{2})}\right\vert ^{2}d\mathbf{x}_{N}\right) ^{\frac{1%
}{2}} \\
&=&\frac{1}{\sqrt{N}}\left( \limfunc{Tr}S_{1}S_{1^{\prime }}S^{(2)}\alpha
_{N}^{(2)}\right) ^{\frac{1}{2}}\left( \limfunc{Tr}S^{(2)}\alpha
_{N}^{(2)}\right) ^{\frac{1}{2}} \\
&\leqslant &C_{0}^{3}
\end{eqnarray*}%
by condition (\ref{bound:k-energy estimate}) with $k=2$.

\subsection{Estimate for the Potential Term}

Recall%
\begin{equation*}
\tilde{V}_{N}^{(k)}\alpha _{N}^{(k)}=(A_{N}^{(k)}-A_{N}^{(k)^{\prime
}})\alpha _{N}^{(k)}+(E_{N}^{(k)}-E_{N}^{(k)^{\prime }})\alpha _{N}^{(k)},
\end{equation*}%
where 
\begin{equation*}
A_{N}^{(k)}\alpha _{N}^{(k)}=-\sum_{\substack{ 1\leq i,j,\ell \leq k  \\ %
i,j,\ell \text{ distinct}}}\frac{\nabla _{x_{\ell }}G_{N,\ell ,i}\cdot
\nabla _{x_{\ell }}G_{N,\ell ,j}}{G_{N,\ell ,i}\;G_{N,\ell ,j}}\alpha
_{N}^{(k)},
\end{equation*}%
and%
\begin{equation*}
E_{N}^{(k)}\alpha _{N}^{(k)}=2\sum_{\substack{ 1\leq j,\ell \leq k  \\ j\neq
\ell }}\frac{\nabla _{x_{\ell }}G_{N,j,\ell }}{G_{N,j,\ell }}\cdot \nabla
_{x_{\ell }}\alpha _{N}^{(k)}.
\end{equation*}%
Let us define 
\begin{eqnarray}
A_{N,i,j,l}^{(k)}\alpha _{N}^{(k)} &=&-\frac{\nabla _{x_{\ell }}G_{N,\ell
,i}\cdot \nabla _{x_{\ell }}G_{N,\ell ,j}}{G_{N,\ell ,i}\;G_{N,\ell ,j}}%
\alpha _{N}^{(k)},  \label{formula: a piece of three body} \\
E_{N,j,l}^{(k)}\alpha _{N}^{(k)} &=&2\frac{\nabla _{x_{\ell }}G_{N,j,\ell }}{%
G_{N,j,\ell }}\cdot \nabla _{x_{\ell }}\alpha _{N}^{(k)},
\label{formula: a piece of 1st order}
\end{eqnarray}%
then to prove estimate (\ref{estimate:PP_0 goes to zero}), it suffices to
prove the following estimates%
\begin{eqnarray*}
\left\Vert \int_{0}^{t_{k}}U^{(k)}(t_{k}-t_{k+1})A_{N,i,j,l}^{(k)}\alpha
_{N}^{(k)}\left( t_{k+1}\right) dt_{k+1}\right\Vert _{L_{T}^{\infty }L_{%
\mathbf{x,x}^{\prime }}^{2}} &\leqslant &C_{T}N^{-2+}\text{,} \\
\left\Vert \int_{0}^{t_{k}}U^{(k)}(t_{k}-t_{k+1})E_{N,j,l}^{(k)}\alpha
_{N}^{(k)}\left( t_{k+1}\right) dt_{k+1}\right\Vert _{L_{T}^{\infty }L_{%
\mathbf{x,x}^{\prime }}^{2}} &\leqslant &C_{T}N^{-1+}.
\end{eqnarray*}

In fact, assume the above estimates for the moment, we have%
\begin{eqnarray*}
\left\Vert PP^{k,0}\right\Vert _{L_{T}^{\infty }\mathcal{L}^{2}}
&=&\left\Vert PP^{k,0}\right\Vert _{L_{T}^{\infty }L_{\mathbf{x,x}^{\prime
}}^{2}}\leqslant C_{T}k^{3}N^{-2+}+C_{T}k^{2}N^{-1+} \\
&\rightarrow &0\text{ as }N\rightarrow \infty \text{,}
\end{eqnarray*}%
for $\beta \in \left( 0,1\right] $, where we used the facts that $\left\Vert
\cdot \right\Vert _{L_{T}^{\infty }\mathcal{L}^{2}}=\left\Vert \cdot
\right\Vert _{L_{T}^{\infty }L_{\mathbf{x,x}^{\prime }}^{2}}$ and there are $%
k^{3}$ summands in $A_{N}^{(k)}$ while there are $k^{2}$ summands in $%
E_{N}^{(k)}$. So we finish the estimate for the potential part in the proof
of Theorem \ref{THM:Convergence} with the following two lemmas.

\begin{lemma}
\label{Lem:key three-body in convergence}We have the estimate:%
\begin{eqnarray*}
&&\left\Vert \int_{0}^{t_{k}}U^{(k)}(t_{k}-t_{k+1})A_{N,i,j,l}^{(k)}\alpha
^{(k)}\left( t_{k+1}\right) dt_{k+1}\right\Vert _{L_{T}^{\infty }L_{\mathbf{%
x,x}^{\prime }}^{2}} \\
&\leqslant &C_{T}N^{-2+}\left\Vert \left\langle \nabla _{x_{i}}\right\rangle
\left\langle \nabla _{x_{j}}\right\rangle \left\langle \nabla
_{x_{l}}\right\rangle \alpha ^{(k)}\left( t_{k+1}\right) \right\Vert
_{L_{t}^{\infty }L_{\mathbf{x,x}^{\prime }}^{2}}\text{.}
\end{eqnarray*}%
In particular, if one assumes the energy bound (\ref{bound:energy estimate}%
), it reads 
\begin{equation*}
\left\Vert \int_{0}^{t_{k}}U^{(k)}(t_{k}-t_{k+1})A_{N,i,j,l}^{(k)}\alpha
_{N}^{(k)}\left( t_{k+1}\right) dt_{k+1}\right\Vert _{L_{T}^{\infty }L_{%
\mathbf{x,x}^{\prime }}^{2}}\leqslant C_{T}N^{-2+}.
\end{equation*}
\end{lemma}

\begin{lemma}
\label{Lem:key 1st order in convergence}%
\begin{eqnarray*}
&&\left\Vert \int_{0}^{t_{k}}U^{(k)}(t_{k}-t_{k+1})E_{N,j,l}^{(k)}\alpha
^{(k)}\left( t_{k+1}\right) dt_{k+1}\right\Vert _{L_{T}^{\infty }L_{\mathbf{%
x,x}^{\prime }}^{2}} \\
&\leqslant &C_{T}N^{-1+}\left\Vert \left\langle \nabla _{x_{j}}\right\rangle
\left\langle \nabla _{x_{\ell }}\right\rangle \alpha ^{(k)}\left(
t_{k+1}\right) \right\Vert _{L_{t}^{\infty }L_{\mathbf{x,x}^{\prime }}^{2}}
\end{eqnarray*}%
In particular, if one assumes the energy bound (\ref{bound:energy estimate}%
), it reads 
\begin{equation*}
\left\Vert \int_{0}^{t_{k}}U^{(k)}(t_{k}-t_{k+1})E_{N,j,l}^{(k)}\alpha
_{N}^{(k)}\left( t_{k+1}\right) dt_{k+1}\right\Vert _{L_{T}^{\infty }L_{%
\mathbf{x,x}^{\prime }}^{2}}\leqslant C_{T}N^{-1+}.
\end{equation*}
\end{lemma}

\begin{proof}[Proof of Lemma \protect\ref{Lem:key three-body in convergence}]

Define 
\begin{equation*}
v_{2,N}(x)=\frac{N^{2\beta -1}\left( \nabla \omega _{0}\right) \left(
N^{\beta }x\right) }{G_{N}(x)},
\end{equation*}%
then%
\begin{equation*}
v_{2,N}(x)=\frac{N^{2\beta -1}\left( \nabla \omega _{0}\right) \left(
N^{\beta }x\right) }{G_{N}(x)}\sim N^{2\beta -1}\left\langle N^{\beta
}x\right\rangle ^{-2}\equiv \tilde{v}_{2,N}(x).
\end{equation*}%
So%
\begin{eqnarray*}
&&\left\Vert \int_{0}^{t_{k}}U^{(k)}(t_{k}-t_{k+1})A_{N,i,j,l}^{(k)}\alpha
^{(k)}\left( t_{k+1}\right) dt_{k+1}\right\Vert _{L_{T}^{\infty }L_{\mathbf{%
x,x}^{\prime }}^{2}} \\
&\sim &\left\Vert \int_{0}^{t_{k}}U^{(k)}(t_{k}-t_{k+1})\left[ \tilde{v}%
_{2,N}(x_{l}-x_{i})\tilde{v}_{2,N}(x_{l}-x_{j})\alpha ^{(k)}\left(
t_{k+1}\right) \right] dt_{k+1}\right\Vert _{L_{T}^{\infty }L_{\mathbf{x,x}%
^{\prime }}^{2}}
\end{eqnarray*}%
Insert a smooth cut-off $\theta (t)$ with $\theta (t)=1$ for $t\in \left[
-T,T\right] $ and $\theta (t)=0$ for $t\in \left[ -2T,2T\right] ^{c}$ into
the above,%
\begin{equation*}
\leqslant \left\Vert \theta (t_{k})\int_{0}^{t_{k}}U^{(k)}(t_{k}-t_{k+1}) 
\left[ \tilde{v}_{2,N}(x_{l}-x_{i})\tilde{v}_{2,N}(x_{l}-x_{j})\theta
(t_{k+1})\alpha ^{(k)}\left( t_{k+1}\right) \right] dt_{k+1}\right\Vert
_{L_{t}^{\infty }L_{\mathbf{x,x}^{\prime }}^{2}}
\end{equation*}%
Since $\left\Vert \cdot \right\Vert _{L_{t}^{\infty }L_{\mathbf{x,x}^{\prime
}}^{2}}\leqslant C\left\Vert \cdot \right\Vert _{X_{\frac{1}{2}+}^{(k)}}$,
we have%
\begin{equation*}
\leqslant C\left\Vert \theta (t_{k})\int_{0}^{t_{k}}U^{(k)}(t_{k}-t_{k+1}) 
\left[ \tilde{v}_{2,N}(x_{l}-x_{i})\tilde{v}_{2,N}(x_{l}-x_{j})\theta
(t_{k+1})\alpha ^{(k)}\left( t_{k+1}\right) \right] dt_{k+1}\right\Vert _{X_{%
\frac{1}{2}+}^{(k)}}.
\end{equation*}%
By Lemma \ref{Lemma:b to b-1},

\begin{equation*}
\leqslant C\left\Vert \tilde{v}_{2,N}(x_{l}-x_{i})\tilde{v}%
_{2,N}(x_{l}-x_{j})\theta (t_{k+1})\alpha ^{(k)}\left( t_{k+1}\right)
\right\Vert _{X_{-\frac{1}{2}+}^{(k)}}.
\end{equation*}%
Use the first inequality of (\ref{estimate:basic corollary in three body})
in Corollary \ref{corollary:basic corollary in three body},%
\begin{eqnarray*}
&\leqslant &C\left\Vert \tilde{v}_{2,N}\right\Vert _{L^{\frac{3}{2}%
+}}^{2}\left\Vert \theta (t_{k+1})\left\langle \nabla _{x_{i}}\right\rangle
\left\langle \nabla _{x_{j}}\right\rangle \left\langle \nabla
_{x_{l}}\right\rangle \alpha ^{(k)}\left( t_{k+1}\right) \right\Vert
_{L_{t}^{2}L_{\mathbf{x,x}^{\prime }}^{2}} \\
&\leqslant &C_{T}\left\Vert \tilde{v}_{2,N}\right\Vert _{L^{\frac{3}{2}%
+}}^{2}\left\Vert \left\langle \nabla _{x_{i}}\right\rangle \left\langle
\nabla _{x_{j}}\right\rangle \left\langle \nabla _{x_{l}}\right\rangle
\alpha ^{(k)}\left( t_{k+1}\right) \right\Vert _{L_{t}^{\infty }L_{\mathbf{%
x,x}^{\prime }}^{2}}
\end{eqnarray*}%
where%
\begin{equation*}
\left\Vert \tilde{v}_{2,N}\right\Vert _{L^{\frac{3}{2}+}}=CN^{-1+}.
\end{equation*}%
That is%
\begin{eqnarray*}
&&\left\Vert \int_{0}^{t_{k}}U^{(k)}(t_{k}-t_{k+1})A_{N,i,j,l}^{(k)}\alpha
^{(k)}\left( t_{k+1}\right) dt_{k+1}\right\Vert _{L_{T}^{\infty }L_{\mathbf{%
x,x}^{\prime }}^{2}} \\
&\leqslant &C_{T}N^{-2+}\left\Vert \left\langle \nabla _{x_{i}}\right\rangle
\left\langle \nabla _{x_{j}}\right\rangle \left\langle \nabla
_{x_{l}}\right\rangle \alpha ^{(k)}\left( t_{k+1}\right) \right\Vert
_{L_{t}^{\infty }L_{\mathbf{x,x}^{\prime }}^{2}}
\end{eqnarray*}%
as claimed.
\end{proof}

\begin{proof}[Proof of Lemma \protect\ref{Lem:key 1st order in convergence}]

As in the proof of Lemma \ref{Lem:key three-body in convergence}, we replace 
\begin{equation*}
v_{2,N}(x)=\frac{N^{2\beta -1}\left( \nabla \omega _{0}\right) \left(
N^{\beta }x\right) }{G_{N}(x)}\sim N^{2\beta -1}\left\langle N^{\beta
}x\right\rangle ^{-2}=\tilde{v}_{2,N}(x)
\end{equation*}%
with $\tilde{v}_{2,N}(x)=N^{2\beta -1}\left\langle N^{\beta }x\right\rangle
^{-2}$. Then, we have 
\begin{eqnarray*}
&&\left\Vert \int_{0}^{t_{k}}U^{(k)}(t_{k}-t_{k+1})E_{N,j,l}^{(k)}\alpha
^{(k)}\left( t_{k+1}\right) dt_{k+1}\right\Vert _{L_{T}^{\infty }L_{\mathbf{%
x,x}^{\prime }}^{2}} \\
&\sim &\left\Vert \int_{0}^{t_{k}}U^{(k)}(t_{k}-t_{k+1})\left[ \tilde{v}%
_{2,N}(x)\left( \left\vert \nabla _{x_{\ell }}\right\vert \alpha
^{(k)}\left( t_{k+1}\right) \right) \right] dt_{k+1}\right\Vert
_{L_{T}^{\infty }L_{\mathbf{x,x}^{\prime }}^{2}} \\
&\leqslant &\left\Vert \tilde{v}_{2,N}(x)\left( \theta (t_{k+1})\left\vert
\nabla _{x_{\ell }}\right\vert \alpha ^{(k)}\left( t_{k+1}\right) \right)
\right\Vert _{X_{-\frac{1}{2}+}^{(k)}}
\end{eqnarray*}%
Use the second inequality of (\ref{estimate:basic corollary for 3D endpoint}%
) in Corollary \ref{corollary:basic corollary for 3D endpoint}, 
\begin{eqnarray*}
&\leqslant &C\left\Vert \tilde{v}_{2,N}\right\Vert _{L^{\frac{3}{2}%
+}}\left\Vert \theta (t_{k+1})\left\vert \nabla _{x_{j}}\right\vert
\left\vert \nabla _{x_{\ell }}\right\vert \alpha ^{(k)}\left( t_{k+1}\right)
\right\Vert _{L_{t}^{2}L_{\mathbf{x,x}^{\prime }}^{2}} \\
&\leqslant &C_{T}N^{-1+}\left\Vert \left\langle \nabla _{x_{j}}\right\rangle
\left\langle \nabla _{x_{\ell }}\right\rangle \alpha ^{(k)}\left(
t_{k+1}\right) \right\Vert _{L_{t}^{\infty }L_{\mathbf{x,x}^{\prime }}^{2}}.
\end{eqnarray*}%
So we have finished the proof of Lemma \ref{Lem:key 1st order in convergence}%
.
\end{proof}

\subsection{Estimate for the $k$-body Interaction Part}

Recall 
\begin{equation*}
\tilde{B}_{N,many}^{(k+1)}\alpha _{N}^{(k+1)}=\sum_{l=1}^{k}\tilde{B}%
_{N,many,l,k+1}\alpha _{N}^{(k+1)}.
\end{equation*}
To prove estimate (\ref{estimate:KIP_0 goes to zero}), we prove the estimate:%
\begin{equation}
\left\Vert \int_{0}^{t_{k}}U^{(k)}(t_{k}-t_{k+1})\tilde{B}%
_{N,many,l,k+1}^{+}\alpha _{N}^{(k+1)}\left( t_{k+1}\right)
dt_{k+1}\right\Vert _{L_{T}^{\infty }L_{\mathbf{x,x}^{\prime
}}^{2}}\leqslant C_{T}C^{k+1}N^{-\frac{1}{2}+}\text{,}
\label{estimate:key loc in convergence}
\end{equation}%
where $\tilde{B}_{N,many,l,k+1}^{+}$ is half of $\tilde{B}_{N,many,l,k+1}$.
Assume estimate (\ref{estimate:key loc in convergence}), then%
\begin{equation*}
\left\Vert \mathnormal{KIP}^{k,0}\right\Vert _{L_{T}^{\infty }\mathcal{L}%
^{2}}=\left\Vert \mathnormal{KIP}^{k,0}\right\Vert _{L_{T}^{\infty }L_{%
\mathbf{x,x}^{\prime }}^{2}}\leqslant 2kC_{T}C^{k+1}N^{-\frac{1}{2}%
+}\rightarrow 0\text{ as }N\rightarrow \infty \text{.}
\end{equation*}%
The rest of this section is the proof of estimate (\ref{estimate:key loc in
convergence}). We first give the following lemma.

\begin{lemma}
\label{lemma:term counting in L}One can decompose $\tilde{B}%
_{N,many,l,k+1}^{+}\alpha _{N}^{(k+1)}\left( t_{k+1}\right) $, defined in (%
\ref{def:B,loc}), as the sum of at most $8^{k}$ terms of the form 
\begin{eqnarray*}
&&\tilde{B}_{N,many,l,k+1,\sigma }^{+}\alpha _{N}^{(k+1)}\left(
t_{k+1}\right) \\
&=&\frac{N-k}{N}\int_{\mathbb{R}^{3}}\tilde{V}_{N}(x_{l}-x_{k+1})N^{\beta
-1}w_{0}(N^{\beta }(x_{\sigma }-x_{k+1}))A_{\sigma }\alpha _{N}^{(k+1)}(%
\mathbf{x}_{k},x_{k+1};\mathbf{x}_{k}^{\prime },x_{k+1})\,dx_{k+1}.
\end{eqnarray*}%
Here, $x_{\sigma }$ is some $x_{j}$ or $x_{j}^{\prime }$ but not $x_{l}$ and 
$A_{\sigma }$ is a product of $\left[ N^{\beta -1}w_{0}(N^{\beta
}(x_{j}-x_{k+1}))\right] $, $\left[ N^{\beta -1}w_{0}(N^{\beta
}(x_{j}^{\prime }-x_{k+1}))\right] $ or $1$ with $x_{j}$ not equal to $x_{l}$
or $x_{\sigma }$.
\end{lemma}

\begin{proof}
Recall, 
\begin{eqnarray*}
\tilde{B}_{N,many,l,k+1}^{+}\alpha _{N}^{(k+1)} &=&\frac{N-k}{N}\int_{%
\mathbb{R}^{3}}\tilde{V}_{N}(x_{l}-x_{k+1})L_{N,l,k+1} \\
&&\alpha _{N}^{(k+1)}(\mathbf{x}_{k},x_{k+1};\mathbf{x}_{k}^{\prime
},x_{k+1})\,dx_{k+1}
\end{eqnarray*}%
Notice that,%
\begin{eqnarray*}
L_{N,l,k+1}+1 &=&G_{N,l^{\prime },k+1}\prod_{_{\substack{ 1\leqslant
j\leqslant k  \\ j\neq l}}}G_{N,j,k+1}G_{N,j^{\prime },k+1} \\
G_{N,j,k+1} &=&1-N^{\beta -1}w_{0}(N^{\beta }(x_{j}-x_{k+1}))
\end{eqnarray*}%
Thus, taken as a binomial expansion, $L_{N,l,k+1}$ is a sum of $2k$ classes
where each class has $%
\begin{pmatrix}
2k \\ 
j%
\end{pmatrix}%
$\thinspace $,j=1,...,2k,$ terms inside, that is: 
\begin{eqnarray*}
&&L_{N,l,k+1} \\
&=&\left( \sum_{_{\substack{ 1\leqslant j\leqslant k  \\ j\neq l}}}N^{\beta
-1}w_{0}(N^{\beta }(x_{j}-x_{k+1}))+\sum_{_{1\leqslant j\leqslant
k}}N^{\beta -1}w_{0}(N^{\beta }(x_{j}^{\prime }-x_{k+1}))\right) \\
&&+... \\
&&+(\left[ N^{\beta -1}w_{0}(N^{\beta }(x_{l}^{\prime }-x_{k+1}))\right]
\prod_{_{\substack{ 1\leqslant j\leqslant k  \\ j\neq l}}}\left[ N^{\beta
-1}w_{0}(N^{\beta }(x_{j}-x_{k+1}))\right] \left[ N^{\beta -1}w_{0}(N^{\beta
}(x_{j}^{\prime }-x_{k+1}))\right] ).
\end{eqnarray*}%
Thus $L_{N,l,k+1}$ can be written as a sum of at most $8^{k}$ terms which
individually looks like%
\begin{equation*}
N^{\beta -1}w_{0}(N^{\beta }(x_{\sigma }-x_{k+1}))A_{\sigma }
\end{equation*}%
where $x_{\sigma }$ is some $x_{j}$ or $x_{j}^{\prime }$ but not $x_{l}$ and 
$A_{\sigma }$ is a product of $\left[ N^{\beta -1}w_{0}(N^{\beta
}(x_{j}-x_{k+1}))\right] $, $\left[ N^{\beta -1}w_{0}(N^{\beta
}(x_{j}^{\prime }-x_{k+1}))\right] $ or $1$ with $x_{j}$ not equal to $x_{l}$
or $x_{\sigma }$. Inserting this into (\ref{def:B,loc}), we have the claimed
decomposition.
\end{proof}

With Lemma \ref{lemma:term counting in L}, we have the following estimate.

\begin{lemma}
\label{Lem:KeyOfLocInConvergence}Let $\tilde{B}_{N,many,l,k+1,\sigma }^{+}$
be defined as in Lemma \ref{lemma:term counting in L}, we have%
\begin{eqnarray*}
&&\left\Vert \int_{0}^{t_{k}}U^{(k)}(t_{k}-t_{k+1})\tilde{B}%
_{N,many,l,k+1,\sigma }^{+}\alpha ^{(k+1)}\left( t_{k+1}\right)
dt_{k+1}\right\Vert _{L_{T}^{\infty }L_{\mathbf{x,x}^{\prime }}^{2}}\, \\
&\leqslant &C_{T}C^{k+1}N^{-\frac{1}{2}+}\frac{1}{\sqrt{N}}\left\Vert
\left\langle \nabla _{x_{k+1}}\right\rangle ^{2}\left\langle \nabla
_{x_{k+1}^{\prime }}\right\rangle \alpha ^{(k+1)}(t_{k+1})\right\Vert
_{L_{t}^{\infty }L_{\mathbf{x,x}^{\prime }}^{2}}.
\end{eqnarray*}%
Consequently,%
\begin{eqnarray*}
&&\left\Vert \int_{0}^{t_{k}}U^{(k)}(t_{k}-t_{k+1})\tilde{B}%
_{N,many,l,k+1}^{+}\alpha ^{(k+1)}\left( t_{k+1}\right) dt_{k+1}\right\Vert
_{L_{T}^{\infty }L_{\mathbf{x,x}^{\prime }}^{2}} \\
&\leqslant &C_{T}C^{k+1}N^{-\frac{1}{2}+}\frac{1}{\sqrt{N}}\left\Vert
\left\langle \nabla _{x_{k+1}}\right\rangle ^{2}\left\langle \nabla
_{x_{k+1}^{\prime }}\right\rangle \alpha ^{(k+1)}(t_{k+1})\right\Vert
_{L_{t}^{\infty }L_{\mathbf{x,x}^{\prime }}^{2}}
\end{eqnarray*}%
because, by Lemma \ref{lemma:term counting in L}, 
\begin{equation*}
\tilde{B}_{N,many,l,k+1}^{+}\alpha _{N}^{(k+1)}=\sum_{\sigma }\tilde{B}%
_{N,loc,l,k+1,\sigma }^{+},
\end{equation*}%
where the sum has at most $8^{k}$ terms inside. In particular, if one
assumes the energy bound (\ref{bound:energy estimate}), it reads%
\begin{eqnarray*}
&&\left\Vert \int_{0}^{t_{k}}U^{(k)}(t_{k}-t_{k+1})\tilde{B}%
_{N,many,l,k+1}^{+}\alpha ^{(k+1)}\left( t_{k+1}\right) dt_{k+1}\right\Vert
_{L_{T}^{\infty }L_{\mathbf{x,x}^{\prime }}^{2}} \\
&\leqslant &C_{T}C^{k+1}N^{-\frac{1}{2}+},
\end{eqnarray*}%
which is exactly estimate (\ref{estimate:key loc in convergence}).
\end{lemma}

\begin{proof}
Recall 
\begin{eqnarray*}
&&\tilde{B}_{N,many,l,k+1,\sigma }^{+}\alpha ^{(k+1)}\left( t_{k+1}\right) \\
&=&\frac{N-k}{N}\int_{\mathbb{R}^{3}}\tilde{V}_{N}(x_{l}-x_{k+1})N^{\beta
-1}w_{0}(N^{\beta }(x_{\sigma }-x_{k+1}))A_{\sigma }\alpha ^{(k+1)}(\mathbf{x%
}_{k},x_{k+1};\mathbf{x}_{k}^{\prime },x_{k+1})dx_{k+1}.
\end{eqnarray*}%
There is no need to write out the variables in $A_{\sigma }$. In fact, $%
A_{\sigma }$ is a harmless factor because $N^{\beta -1}w_{0}(N^{\beta
}(x_{j}-x_{k+1}))$ is in $L^{\infty }$ uniformly in $N$ if $\beta \leqslant
1 $.

As in the proof of Lemmas \ref{Lem:key three-body in convergence} and \ref%
{Lem:key 1st order in convergence}, we insert a smooth cut-off $\theta (t)$,%
\begin{eqnarray*}
&&\left\Vert \int_{0}^{t_{k}}U^{(k)}(t_{k}-t_{k+1})\tilde{B}%
_{N,many,l,k+1,\sigma }^{+}\alpha ^{(k+1)}\left( t_{k+1}\right)
dt_{k+1}\right\Vert _{L_{T}^{\infty }L_{\mathbf{x,x}^{\prime }}^{2}} \\
&\leqslant &\left\Vert \theta (t_{k})\int_{0}^{t_{k}}U^{(k)}(t_{k}-t_{k+1})%
\tilde{B}_{N,many,l,k+1,\sigma }^{+}\left( \theta (t_{k+1})\alpha
^{(k+1)}\left( t_{k+1}\right) \right) dt_{k+1}\right\Vert _{L_{t}^{\infty
}L_{\mathbf{x,x}^{\prime }}^{2}},
\end{eqnarray*}%
and proceed to%
\begin{eqnarray*}
&\leqslant &C\left\Vert \tilde{B}_{N,many,l,k+1,\sigma }^{+}\left( \theta
(t_{k+1})\alpha ^{(k+1)}\left( t_{k+1}\right) \right) \right\Vert _{X_{-%
\frac{1}{2}+}^{(k)}} \\
&=&C\Bigg\|\int_{\mathbb{R}^{3}}\tilde{V}_{N}(x_{l}-x_{k+1})N^{\beta
-1}w_{0}(N^{\beta }(x_{\sigma }-x_{k+1})) \\
&&A_{\sigma }\theta (t_{k+1})\alpha ^{(k+1)}(\mathbf{x}_{k},x_{k+1};\mathbf{x%
}_{k}^{\prime },x_{k+1})dx_{k+1}\Bigg\|_{X_{-\frac{1}{2}+}^{(k)}}.
\end{eqnarray*}%
The third inequality of (\ref{E:Str20}) of Lemma \ref{Lem:basic lem in KIP}
gives%
\begin{eqnarray*}
&\leqslant &C\left\Vert \tilde{V}_{N}\right\Vert _{L^{1+}}\left\Vert
N^{\beta -1}w_{0}(N^{\beta }(\cdot ))\right\Vert _{L^{3+}}\left\Vert
A_{\sigma }\right\Vert _{L^{\infty }} \\
&&\times \left\Vert \left\langle \nabla _{x_{k+1}}\right\rangle
^{2}\left\langle \nabla _{x_{k+1}^{\prime }}\right\rangle \theta
(t_{k+1})\alpha ^{(k+1)}(\mathbf{x}_{k},x_{k+1};\mathbf{x}_{k}^{\prime
},x_{k+1})\right\Vert _{L_{t}^{2}L_{\mathbf{x,x}^{\prime }}^{2}}
\end{eqnarray*}%
where%
\begin{equation*}
\left\Vert \tilde{V}_{N}\right\Vert _{L^{1+}}\left\Vert N^{\beta
-1}w_{0}(N^{\beta }(\cdot ))\right\Vert _{L^{3+}}\left\Vert A_{\sigma
}\right\Vert _{L^{\infty }}\leqslant C^{k+1}N^{-1+}.
\end{equation*}%
Thus%
\begin{eqnarray*}
&&\left\Vert \int_{0}^{t_{k}}U^{(k)}(t_{k}-t_{k+1})\tilde{B}%
_{N,many,l,k+1,\sigma }^{+}\alpha ^{(k+1)}\left( t_{k+1}\right)
dt_{k+1}\right\Vert _{L_{T}^{\infty }L_{\mathbf{x,x}^{\prime }}^{2}} \\
&\leqslant &C_{T}C^{k+1}N^{-\frac{1}{2}+}\frac{1}{\sqrt{N}}\left\Vert
\left\langle \nabla _{x_{k+1}}\right\rangle ^{2}\left\langle \nabla
_{x_{k+1}^{\prime }}\right\rangle \alpha ^{(k+1)}(t_{k+1})\right\Vert
_{L_{t}^{\infty }L_{\mathbf{x,x}^{\prime }}^{2}}.
\end{eqnarray*}%
which is good enough to conclude the proof of Lemma \ref%
{Lem:KeyOfLocInConvergence}.
\end{proof}

\section{Proof of Theorem \protect\ref{THM:KM Bound}\label{sec:pf of thm2}}

We will use Littlewood-Paley theory to prove Theorem \ref{THM:KM Bound}. Let 
$P_{\leq M}^{i}$ be the projection onto frequencies $\leq M$ and $P_{M}^{i}$
the analogous projections onto frequencies $\sim M$, acting on functions of $%
x_{i}\in \mathbb{R}^{3}$ (the $i$th coordinate). We take $M$ to be a dyadic
frequency range $2^{\ell }\geq 1$. Similarly, we define $P_{\leq
M}^{i^{\prime }}$ and $P_{M}^{i^{\prime }}$, which act on the variable $%
x_{i}^{\prime }$. Let 
\begin{equation}
P_{\leq M}^{(k)}=\prod_{i=1}^{k}P_{\leq M}^{i}P_{\leq M}^{i^{\prime }}.
\label{E:LP-def}
\end{equation}%
As observed in earlier work \cite%
{TChenAndNPSpace-Time,Chen3DDerivation,C-H2/3}, to establish Theorem \ref%
{THM:KM Bound}, it suffices to prove the following theorem.\footnote{%
To be precise, this formulation with frequency localization is from \cite%
{C-H2/3}. The formulations in \cite{TChenAndNPSpace-Time,Chen3DDerivation}
do not have the Littlewood-Paley projector inside.}

\begin{theorem}
\label{THM:KMwithL-P}Under the assumptions of Theorem \ref{THM:KM Bound},
there exists a $C$ (independent of $k,M,N$) such that for each $M\geqslant 1$
there exists $N_{0}$ (depending on $M$) such that for $N\geqslant N_{0}$,
there holds 
\begin{equation}
\Vert P_{\leq M}^{(k)}R^{(k)}\tilde{B}_{N,j,k+1}\gamma _{N}^{(k+1)}(t)\Vert
_{L_{T}^{1}L_{\mathbf{x},\mathbf{x}^{\prime }}^{2}}\leqslant C^{k}.
\label{estimate:main target}
\end{equation}
\end{theorem}

In fact, passing to the weak limit $\gamma _{N}^{(k)}\rightarrow \gamma
^{(k)}$ as $N\rightarrow \infty $, we obtain 
\begin{equation*}
\Vert P_{\leq M}^{(k)}R^{(k)}B_{j,k+1}\gamma ^{(k+1)}\Vert _{L_{T}^{1}L_{%
\mathbf{x},\mathbf{x}^{\prime }}^{2}}\leqslant C^{k}
\end{equation*}%
Since it holds uniformly in $M$, we can send $M\rightarrow \infty $ and, by
the monotone convergence theorem, we obtain 
\begin{equation*}
\Vert R^{(k)}B_{j,k+1}\gamma ^{(k+1)}\Vert _{L_{T}^{1}L_{\mathbf{x},\mathbf{x%
}^{\prime }}^{2}}\leqslant C^{k}
\end{equation*}%
which is exactly the Klainerman-Machedon space-time bound (\ref{bound:target
KM bound}). This completes the proof Theorem \ref{THM:KM Bound}, assuming
Theorem \ref{THM:KMwithL-P}.

The rest of this section is devoted to proving Theorem \ref{THM:KMwithL-P}.
We will first establish estimate (\ref{estimate:main target}) for a
sufficiently small $T$ which depends on the controlling constant in
condition (\ref{bound:k-energy estimate}) and is independent of $k$, $N$ and 
$M,$ then a bootstrap argument together with condition (\ref{bound:k-energy
estimate}) give estimate (\ref{estimate:main target}) for every finite time
at the price of a larger constant $C$. The first step of the proof of
Theorem \ref{THM:KMwithL-P} is to iterate (\ref{Hierarchy: Real BBGKY at k
level}) $p$ times and get to the formula%
\begin{equation*}
\alpha _{N}^{(k)}(t_{k})=\mathnormal{FP}^{k,p}(t_{k})+\mathnormal{PP}%
^{k,p}(t_{k})+\mathnormal{KIP}^{k,p}(t_{k})+\mathnormal{IP}^{k,p}(t_{k}),
\end{equation*}%
then estimate each term, that is, prove the following estimates:%
\begin{eqnarray}
\left\Vert P_{\leq M}^{(k-1)}R^{(k-1)}\tilde{B}_{N,1,k}\mathnormal{FP}%
^{k,p}\right\Vert _{L_{T}^{1}L_{\mathbf{x,x}^{\prime }}^{2}} &\leqslant
&C^{k-1},  \label{estimate:fp} \\
\left\Vert P_{\leq M}^{(k-1)}R^{(k-1)}\tilde{B}_{N,1,k}\mathnormal{PP}%
^{k,p}\right\Vert _{L_{T}^{1}L_{\mathbf{x,x}^{\prime }}^{2}} &\leqslant
&C^{k-1},  \label{estimate:pp} \\
\left\Vert P_{\leq M}^{(k-1)}R^{(k-1)}\tilde{B}_{N,1,k}K\mathnormal{IP}%
^{k,p}\right\Vert _{L_{T}^{1}L_{\mathbf{x,x}^{\prime }}^{2}} &\leqslant
&C^{k-1},  \label{estimate:kip} \\
\left\Vert P_{\leq M}^{(k-1)}R^{(k-1)}\tilde{B}_{N,1,k}\mathnormal{IP}%
^{k,p}\right\Vert _{L_{T}^{1}L_{\mathbf{x,x}^{\prime }}^{2}} &\leqslant
&C^{k-1}.  \label{estimate:ip}
\end{eqnarray}%
for all $k\geqslant 2$ and for some $C$ and a sufficiently small $T$
determined by the controlling constant in condition (\ref{bound:k-energy
estimate}) and independent of $k$, $N$ and $M$. Here, we iterate (\ref%
{Hierarchy: Real BBGKY at k level}) because it is difficult to show (\ref%
{estimate:ip}) unless $p=\ln N$, a fact first observed by Chen and Pavlovic 
\cite{TChenAndNPSpace-Time}, who proved (\ref{bound:target KM bound}) for $%
\beta \in \left( 0,1/4\right) $, and then used in the $\beta \in \left( 0,2/7%
\right] $ work \cite{Chen3DDerivation}\ by X.C and in the $\beta \in \left(
0,2/3\right) $ work \cite{C-H2/3}\ by X.C and J.H. As proven in \cite%
{Chen3DDerivation,C-H2/3}, once $p$ is set to be $\ln N$, one can prove
estimates (\ref{estimate:fp}) and (\ref{estimate:ip}) for all $\beta \in
\left( 0,\infty \right) $. The obstacle in achieving higher $\beta $ lies
solely in proving (\ref{estimate:pp}) and (\ref{estimate:kip}). Hence, in
the rest of this section, we prove estimates (\ref{estimate:pp}) and (\ref%
{estimate:kip}) only and refer the readers to \cite{Chen3DDerivation,C-H2/3}
for the proof of estimates (\ref{estimate:fp}) and (\ref{estimate:ip}).

To make formulas shorter, for $q\geqslant 1$, we introduce the following
notation:%
\begin{equation*}
J_{N}^{(k,q)}(\underline{t}_{k,q})(f^{(k+q)})=\left( U^{(k)}(t_{k}-t_{k+1})%
\tilde{B}_{N}^{(k+1)}\right) \cdots \left( U^{(k+q-1)}(t_{k+q-1}-t_{k+q})%
\tilde{B}_{N}^{(k+q)}\right) f^{(k+q)}\text{,}
\end{equation*}%
where $\underline{t}_{k,q}$ means $\left( t_{k+1},\ldots ,t_{k+q}\right) .$
When $q=0$, the above product is degenerate and we let 
\begin{equation*}
J_{N}^{(k,0)}(\underline{t}_{k,0})(f^{(k)})=f^{(k)}.
\end{equation*}

Now plug the $(k+1)$ version of (\ref{Hierarchy: Real BBGKY at k level})
into the \emph{last term only} of (\ref{Hierarchy: Real BBGKY at k level})
to obtain 
\begin{equation*}
\alpha _{N}^{(k)}(t_{k})=\mathnormal{FP}^{k,1}(t_{k})+\mathnormal{PP}%
^{k,1}(t_{k})+\mathnormal{KIP}^{k,1}(t_{k})+\mathnormal{IP}^{k,1}(t_{k})
\end{equation*}%
where the \emph{free part} is 
\begin{eqnarray*}
&&\mathnormal{FP}^{k,1}(t_{k}) \\
&=&U^{(k)}(t_{k})\alpha _{N,0}^{(k)}+\left( -i\right)
\int_{0}^{t_{k}}U^{(k)}(t_{k}-t_{k+1})\tilde{B}%
_{N}^{(k+1)}U^{(k+1)}(t_{k+1})\alpha _{N,0}^{(k+1)}\,dt_{k+1} \\
&=&\sum_{q=0}^{1}\left( -i\right) ^{q}\int_{0\leq t_{k+q-1}\leq \cdots \leq
t_{k}}J_{N}^{(k,q)}(\underline{t}_{k,q})( f_{FP}^{(k,q)})d\underline{t}_{k,q}
\end{eqnarray*}%
with%
\begin{equation}
f_{FP}^{(k,q)}=U^{(k+q)}(t_{k+q})\alpha _{N,0}^{(k+q)}  \label{def:f_FP}
\end{equation}%
the \emph{potential part} is%
\begin{eqnarray*}
&&\mathnormal{PP}^{k,1}(t_{k}) \\
&=&-i\int_{0}^{t_{k}}U^{(k)}(t_{k}-t_{k+1})\tilde{V}_{N}^{(k)}\alpha
_{N}^{(k)}\left( t_{k+1}\right) dt_{k+1} \\
&&+\left( -i\right)
^{2}\int_{0}^{t_{k}}\int_{0}^{t_{k+1}}U^{(k)}(t_{k}-t_{k+1})\tilde{B}%
_{N}^{(k+1)}U^{(k+1)}(t_{k+1}-t_{k+2})\tilde{V}_{N}^{(k+1)}\alpha
_{N}^{(k+1)}\left( t_{k+2}\right) d\underline{t}_{k,2} \\
&=&\sum_{q=0}^{1}\left( -i\right) ^{q+1}\int_{0\leq t_{k+q-1}\leq \cdots
\leq t_{k}}J_{N}^{(k,q)}(\underline{t}_{k,q})(f_{PP}^{(k,q)})d\underline{t}%
_{k,q}
\end{eqnarray*}%
with%
\begin{equation}
f_{PP}^{(k,q)}=\int_{0}^{t_{k+q}}U^{(k+q)}(t_{k+q}-t_{k+q+1})\tilde{V}%
_{N}^{(k+q)}\alpha _{N}^{(k+q)}\left( t_{k+q+1}\right) dt_{k+q+1}
\label{def:f_PP}
\end{equation}%
the \emph{k-body} \emph{interaction part} is%
\begin{eqnarray*}
&&\mathnormal{KIP}^{k,1}(t_{k}) \\
&=&-i\int_{0}^{t_{k}}U^{(k)}(t_{k}-t_{k+1})\tilde{B}_{N,many}^{(k+1)}\alpha
_{N}^{(k+1)}\left( t_{k+1}\right) dt_{k+1} \\
&&+(-i)^{2}\int_{0}^{t_{k}}\int_{0}^{t_{k+1}}U^{(k)}(t_{k}-t_{k+1})\tilde{B}%
_{N}^{(k+1)}U^{(k+1)}(t_{k+1}-t_{k+2})\tilde{B}_{N,many}^{(k+2)}\alpha
_{N}^{(k+2)}\left( t_{k+2}\right) d\underline{t}_{k,2} \\
&=&\sum_{q=0}^{1}\left( -i\right) ^{q+1}\int_{0\leq t_{k+q-1}\leq \cdots
\leq t_{k}}J_{N}^{(k,q)}(\underline{t}_{k,q})(f_{KIP}^{(k,q)})d\underline{t}%
_{k,q}
\end{eqnarray*}%
with%
\begin{equation}
f_{KIP}^{(k,q)}=\int_{0}^{t_{k+q}}U^{(k+q)}(t_{k+q}-t_{k+q+1})\tilde{B}%
_{N,many}^{(k+q+1)}\alpha _{N}^{(k+q+1)}\left( t_{k+q+1}\right) dt_{k+q+1}
\label{def:f_LIP}
\end{equation}%
and the \emph{interaction part} is 
\begin{eqnarray*}
&&\mathnormal{IP}^{k,1}(t_{k}) \\
&=&(-i)^{2}\int_{0}^{t_{k}}\int_{0}^{t_{k+1}}U^{(k)}(t_{k}-t_{k+1})\tilde{B}%
_{N}^{(k+1)}U^{(k+1)}(t_{k+1}-t_{k+2})\tilde{B}_{N}^{(k+2)}\alpha
_{N}^{(k+2)}\left( t_{k+2}\right) \,d\underline{t}_{k,2} \\
&=&(-i)^{1+1}\int_{0}^{t_{k}}\int_{0}^{t_{k+1}}J_{N}^{(k,2)}(\underline{t}%
_{k,2})(\alpha _{N}^{(k+2)}\left( t_{k+2}\right) )d\underline{t}_{k,2}
\end{eqnarray*}

Now we iterate this process $(p-1)$ more times to obtain 
\begin{equation*}
\alpha _{N}^{(k)}(t_{k})=\mathnormal{FP}^{k,p}(t_{k})+\mathnormal{PP}%
^{k,p}(t_{k})+\mathnormal{KIP}^{k,p}(t_{k})+\mathnormal{IP}^{k,p}(t_{k})
\end{equation*}%
where the \emph{free part} is%
\begin{equation}
\mathnormal{FP}^{k,p}(t_{k})=\sum_{q=0}^{p}\left( -i\right) ^{q}\int_{0\leq
t_{k+q-1}\leq \cdots \leq t_{k}}J_{N}^{(k,q)}(\underline{t}%
_{k,q})(f_{FP}^{(k,q)})d\underline{t}_{k,q}.  \label{def:fp}
\end{equation}%
The \emph{potential part} is%
\begin{equation}
\mathnormal{PP}^{k,p}(t_{k})=\sum_{q=0}^{p}\left( -i\right)
^{q+1}\int_{0\leq t_{k+q-1}\leq \cdots \leq t_{k}}J_{N}^{(k,q)}(\underline{t}%
_{k,q})(f_{PP}^{(k,q)})d\underline{t}_{k,q}.  \label{def:pp}
\end{equation}%
The \emph{k-body interaction part} is%
\begin{equation}
\mathnormal{KIP}^{k,p}(t_{k})=\sum_{q=0}^{p}\left( -i\right)
^{q+1}\int_{0\leq t_{k+q-1}\leq \cdots \leq t_{k}}J_{N}^{(k,q)}(\underline{t}%
_{k,q})(f_{KIP}^{(k,q)})d\underline{t}_{k,q}.  \label{def:lip}
\end{equation}%
The \emph{interaction part} is 
\begin{equation}
\mathnormal{IP}^{k,p}(t_{k})=(-i)^{p+1}\int_{0\leq t_{k+p}\leq \cdots \leq
t_{k}}J_{N}^{(k+p+1)}(\underline{t}_{k,p+1})(\alpha _{N}^{(k+p+1)}\left(
t_{k+p+1}\right) )d\underline{t}_{k,p+1}.  \label{def:ip}
\end{equation}%
We then apply the Klainerman-Machedon board game to the free part, potential
part, $k$-body interaction part, and interaction part.

\begin{lemma}[Klainerman-Machedon board game]
\label{lemma:Klainerman-MachedonBoardGameForBBGKY}\cite%
{KlainermanAndMachedon}One can express 
\begin{equation*}
\int_{0\leq t_{k+q-1}\leq \cdots \leq t_{k}}J_{N}^{(k,q)}(\underline{t}%
_{k,q})(f^{(k+q)})d\underline{t}_{k,q},
\end{equation*}%
as a sum of at most $4^{q}$ terms of the form 
\begin{equation*}
\int_{D}J_{N}^{(k,q)}(\underline{t}_{k,q},\mu _{m})(f^{(k+q)})d\underline{t}%
_{k,q},
\end{equation*}%
or in other words, 
\begin{equation*}
\int_{0\leq t_{k+q-1}\leq \cdots \leq t_{k}}J_{N}^{(k,q)}(\underline{t}%
_{k,q})(f^{(k+q)})d\underline{t}_{k,q}=\sum_{m}\int_{D}J_{N}^{(k,q)}(%
\underline{t}_{k,q},\mu _{m})(f^{(k+q)})d\underline{t}_{k,q}.
\end{equation*}%
Here $D\subset \lbrack 0,t_{k}]^{q}$, $\mu _{m}$ are a set of maps from $%
\{k+1,\ldots ,k+q\}$ to $\{k,\ldots ,k+q-1\}$ satisfying $\mu _{m}(k+1)=k$
and $\mu _{m}(l)<l$ for all $l,$ and 
\begin{eqnarray*}
J_{N}^{(k,q)}(\underline{t}_{k,q},\mu _{m})(f^{(k+q)})
&=&U^{(k)}(t_{k}-t_{k+1})\tilde{B}_{N,k,k+1}U^{(k+1)}(t_{k+1}-t_{k+2})\tilde{%
B}_{N,\mu _{m}(k+2),k+2}\cdots \\
&&\cdots U^{(k+q-1)}(t_{k+q-1}-t_{k+q})\tilde{B}_{N,\mu
_{m}(k+q),k+q}(f^{(k+q)}).
\end{eqnarray*}
\end{lemma}

\subsection{Estimate for the k-body Interaction Part\label{sec:estimate for
LIP}}

%As shown in \S \ref{sec:pf of thm1}, the estimate for the $k$-body
%interaction part is the most complicated one. So we deal with it first in
%the proof of Theorem \ref{THM:KM Bound}. 
To make formulas shorter, let us write 
\begin{equation*}
R_{\leqslant M_{k}}^{(k)}=P_{\leqslant M_{k}}^{(k)}R^{(k)},
\end{equation*}%
since $P_{\leqslant M_{k}}^{(k)}$ and $R^{(k)}$ are usually bundled together.

\subsubsection{Step I}

Applying Lemma \ref{lemma:Klainerman-MachedonBoardGameForBBGKY} to (\ref%
{def:lip}), we get%
\begin{eqnarray*}
&&\left\Vert R_{\leqslant M_{k-1}}^{(k-1)}B_{N,1,k}\mathnormal{KIP}%
^{k,p}\right\Vert _{L_{T}^{1}L_{\mathbf{x,x}^{\prime }}^{2}} \\
&\leqslant &\sum_{q=0}^{p}\sum_{m}\left\Vert R_{\leqslant
M_{k-1}}^{(k-1)}B_{N,1,k}\int_{D}J_{N}^{(k,q)}(\underline{t}_{k,q},\mu
_{m})(f_{KIP}^{(k,q)})d\underline{t}_{k,q}\right\Vert _{L_{T}^{1}L_{\mathbf{%
x,x}^{\prime }}^{2}}
\end{eqnarray*}%
where $f_{KIP}^{(k,q)}$ is given by (\ref{def:f_LIP}) and the sum $\sum_{m}$
has at most $4^{q}$ terms inside. By Minkowski's integral inequality,%
\begin{eqnarray*}
&&\left\Vert R_{\leqslant M_{k-1}}^{(k-1)}B_{N,1,k}\int_{D}J_{N}^{(k,q)}(%
\underline{t}_{k,q},\mu _{m})(f_{KIP}^{(k,q)})d\underline{t}%
_{k,q}\right\Vert _{L_{T}^{1}L_{\mathbf{x,x}^{\prime }}^{2}} \\
&=&\int_{0}^{T}dt_{k}\left\Vert \int_{D}R_{\leqslant
M_{k-1}}^{(k-1)}B_{N,1,k}J_{N}^{(k,q)}(\underline{t}_{k,q},\mu
_{m})(f_{KIP}^{(k,q)})d\underline{t}_{k,q}\right\Vert _{L_{\mathbf{x,x}%
^{\prime }}^{2}}dt_{k} \\
&\leqslant &\int_{[0,T]^{q+1}}\left\Vert R_{\leqslant
M_{k-1}}^{(k-1)}B_{N,1,k}U^{(k)}(t_{k}-t_{k+1})\tilde{B}_{N,k,k+1}...\right%
\Vert _{L_{\mathbf{x,x}^{\prime }}^{2}}dt_{k}d\underline{t}_{k,q}.
\end{eqnarray*}%
Cauchy-Schwarz in the $t_{k}$ integration, 
\begin{equation*}
\leqslant T^{\frac{1}{2}}\int_{\left[ 0,T\right] ^{q}}\left( \int \left\Vert
R_{\leqslant M_{k-1}}^{(k-1)}B_{N,1,k}U^{(k)}(t_{k}-t_{k+1})\tilde{B}%
_{N,k,k+1}...\right\Vert _{L_{\mathbf{x,x}^{\prime }}^{2}}^2 dt_{k}\right) ^{%
\frac{1}{2}}d\underline{t}_{k,q}
\end{equation*}%
By Lemma \ref{Lemma:LocalizedKM}, 
\begin{equation*}
\leqslant C_{\varepsilon }T^{\frac{1}{2}}\sum_{M_{k}\geqslant M_{k-1}}\left( 
\frac{M_{k-1}}{M_{k}}\right) ^{1-\varepsilon }\int_{\left[ 0,T\right]
^{q}}\left\Vert R_{\leqslant M_{k}}^{(k)}\tilde
B_{N,k,k+1}U^{(k+1)}(t_{k+1}-t_{k+2})\cdots \right\Vert _{L_{\mathbf{x},%
\mathbf{x}^{\prime }}^{2}}d\underline{t}_{k,q}
\end{equation*}%
Iterate the previous steps $(q-1)$ times,%
\begin{eqnarray*}
&\leqslant &(C_{\varepsilon }T^{\frac{1}{2}})^{q}\sum_{M_{k+q-1}\geqslant
\cdots \geqslant M_{k}\geqslant M_{k-1}}\Big[\left( \frac{M_{k-1}}{M_{k}}%
\frac{M_{k}}{M_{k+1}}\cdots \frac{M_{k+q-2}}{M_{k+q-1}}\right)
^{1-\varepsilon } \\
&&\times \left\Vert R_{\leqslant M_{k+q-1}}^{(k+q-1)}B_{N,\mu
_{m}(k+q),k+q}\left( f_{KIP}^{(k,q)}\right) \right\Vert _{L_{T}^{1}L_{%
\mathbf{x,x}^{\prime }}^{2}}\Big] \\
&=&(C_{\varepsilon }T^{\frac{1}{2}})^{q}\sum_{M_{k+q-1}\geqslant \cdots
\geqslant M_{k}\geqslant M_{k-1}}\Big[\left( \frac{M_{k-1}}{M_{k+q-1}}%
\right) ^{1-\varepsilon } \\
&&\times \left\Vert R_{\leqslant M_{k+q-1}}^{(k+q-1)}\tilde{B}_{N,\mu
_{m}(k+q),k+q}\left( f_{KIP}^{(k,q)}\right) \right\Vert _{L_{T}^{1}L_{%
\mathbf{x,x}^{\prime }}^{2}}\Big]
\end{eqnarray*}%
where the sum is over all $M_{k},\ldots ,M_{k+q-1}$ dyadic such that $%
M_{k+q-1}\geqslant \cdots \geqslant M_{k}\geqslant M_{k-1}$.

Hence%
\begin{eqnarray*}
&&\left\Vert R_{\leqslant M_{k-1}}^{(k-1)}B_{N,1,k}\mathnormal{KIP}%
^{k,p}\right\Vert _{L_{T}^{1}L_{\mathbf{x,x}^{\prime }}^{2}} \\
&\leqslant &\sum_{q=0}^{p}(C_{\varepsilon }T^{\frac{1}{2}})^{q}%
\sum_{M_{k+q-1}\geqslant \cdots \geqslant M_{k}\geqslant M_{k-1}}\Big[\left( 
\frac{M_{k-1}}{M_{k+q-1}}\right) ^{1-\varepsilon } \\
&&\times \left\Vert R_{\leqslant M_{k+q-1}}^{(k+q-1)}\tilde{B}_{N,\mu
_{m}(k+q),k+q}\left( f_{KIP}^{(k,q)}\right) \right\Vert _{L_{T}^{1}L_{%
\mathbf{x,x}^{\prime }}^{2}}\Big]
\end{eqnarray*}%
We then insert a smooth cut-off $\theta (t)$ with $\theta (t)=1$ for $t\in %
\left[ -T,T\right] $ and $\theta (t)=0$ for $t\in \left[ -2T,2T\right] ^{c}$
into the above estimate to get%
\begin{eqnarray*}
&&\left\Vert R_{\leqslant M_{k-1}}^{(k-1)}B_{N,1,k}\mathnormal{KIP}%
^{k,p}\right\Vert _{L_{T}^{1}L_{\mathbf{x,x}^{\prime }}^{2}} \\
&\leqslant &\sum_{q=0}^{p}(C_{\varepsilon }T^{\frac{1}{2}})^{q}%
\sum_{M_{k+q-1}\geqslant \cdots \geqslant M_{k}\geqslant M_{k-1}}\Big[\left( 
\frac{M_{k-1}}{M_{k+q-1}}\right) ^{1-\varepsilon } \\
&&\times \left\Vert R_{\leqslant M_{k+q-1}}^{(k+q-1)}\tilde{B}_{N,\mu
_{m}(k+q),k+q}\left( \theta (t_{k+q})\tilde{f}_{KIP}^{(k,q)}\right)
\right\Vert _{L_{T}^{1}L_{\mathbf{x,x}^{\prime }}^{2}}\Big]
\end{eqnarray*}%
with%
\begin{equation}
\tilde{f}_{KIP}^{(k,q)}=\int_{0}^{t_{k+q}}U^{(k+q)}(t_{k+q}-t_{k+q+1})\theta
(t_{k+q+1})\tilde{B}_{N,many}^{(k+q+1)}\alpha _{N}^{(k+q+1)}\left(
t_{k+q+1}\right) dt_{k+q+1},  \label{def:f tutle_lip}
\end{equation}%
where the sum is over all $M_{k},\ldots ,M_{k+q-1}$ dyadic such that $%
M_{k+q-1}\geqslant \cdots \geqslant M_{k}\geqslant M_{k-1}$.

\subsubsection{Step II}

With Lemma \ref{Lemma:LocalizedKMWithX_b}, the $X_{b}$ space version of
Lemma \ref{Lemma:LocalizedKM}, we turn Step I into%
\begin{eqnarray*}
&&\left\Vert R_{\leqslant M_{k-1}}^{(k-1)}B_{N,1,k}\mathnormal{KIP}%
^{k,p}\right\Vert _{L_{T}^{1}L_{\mathbf{x,x}^{\prime }}^{2}} \\
&\leqslant &\sum_{q=0}^{p}(C_{\varepsilon }T^{\frac{1}{2}})^{q+1}%
\sum_{M_{k+q}\geqslant M_{k+q-1}\geqslant \cdots \geqslant M_{k}\geqslant
M_{k-1}}\Big[\left( \frac{M_{k-1}}{M_{k+q}}\right) ^{1-\varepsilon } \\
&&\times \left\Vert R_{\leqslant M_{k+q}}^{(k+q)}\left( \theta (t_{k+q})%
\tilde{f}_{KIP}^{(k,q)}\right) \right\Vert _{X_{\frac{1}{2}+}^{(k+q)}}\Big].
\end{eqnarray*}%
Use Lemma \ref{Lemma:b to b-1} gives us 
\begin{eqnarray*}
&\leqslant &\sum_{q=0}^{p}(C_{\varepsilon }T^{\frac{1}{2}})^{q+1}%
\sum_{M_{k+q}\geqslant M_{k+q-1}\geqslant \cdots \geqslant M_{k}\geqslant
M_{k-1}}\Big[\left( \frac{M_{k-1}}{M_{k+q}}\right) ^{1-\varepsilon } \\
&&\times \left\Vert R_{\leqslant M_{k+q}}^{(k+q)}\left( \theta (t_{k+q+1})%
\tilde{B}_{N,many}^{(k+q+1)}\alpha _{N}^{(k+q+1)}\left( t_{k+q+1}\right)
\right) \right\Vert _{X_{-\frac{1}{2}+}^{(k+q)}}\Big].
\end{eqnarray*}%
Carry out the sum in $M_{k}\leqslant \cdots \leqslant M_{k+q-1}$ with the
help of Lemma \ref{L:iterates3}:%
\begin{eqnarray*}
&\leqslant &\sum_{q=0}^{p}(C_{\varepsilon }T^{\frac{1}{2}})^{q+1}%
\sum_{M_{k+q}\geqslant M_{k-1}}\Big[\left( \frac{M_{k-1}}{M_{k+q}}\right)
^{1-\varepsilon }\left( \frac{(\log _{2}\frac{M_{k+q}}{M_{k-1}}+q)^{q}}{q!}%
\right) \\
&&\times \left\Vert R_{\leqslant M_{k+q}}^{(k+q)}\left( \theta (t_{k+q+1})%
\tilde{B}_{N,many}^{(k+q+1)}\alpha _{N}^{(k+q+1)}\left( t_{k+q+1}\right)
\right) \right\Vert _{X_{-\frac{1}{2}+}^{(k+q)}}\Big].
\end{eqnarray*}%
Take a $T^{j/4}$ from the front to apply Lemma \ref{L:iterates4}:%
\begin{eqnarray*}
&\leqslant &(C_{\varepsilon }T^{\frac{1}{2}})\sum_{q=0}^{p}\Bigg\{%
(C_{\varepsilon }T^{\frac{1}{4}})^{q}\sum_{M_{k+q}\geqslant M_{k-1}}\Big[%
\left( \frac{M_{k-1}^{1-2\varepsilon }}{M_{k+q}^{1-2\varepsilon }}\right) \\
&&\times \left\Vert R_{\leqslant M_{k+q}}^{(k+q)}\left( \theta (t_{k+q+1})%
\tilde{B}_{N,many}^{(k+q+1)}\alpha _{N}^{(k+q+1)}\left( t_{k+q+1}\right)
\right) \right\Vert _{X_{-\frac{1}{2}+}^{(k+q)}}\Big]\Bigg\}.
\end{eqnarray*}%
where the sum is over dyadic $M_{k+q}$ such that $M_{k+q}\geqslant M_{k-1}$.

\begin{lemma}[{\protect\cite[Lemma 3.1]{C-H2/3}}]
\label{L:iterates3} 
\begin{equation*}
\left( \sum_{M_{k-1}\leq M_{k}\leq \cdots \leq M_{k+j-1}\leq
M_{k+j}}1\right) \leq \frac{(\log _{2}\frac{M_{k+j}}{M_{k-1}}+j)^{j}}{j!},
\end{equation*}%
where the sum is in $M_{k}\leq \cdots \leq M_{k+j-1}$ over dyads, such that $%
M_{k-1}\leq M_{k}\leq \cdots \leq M_{k+j-1}\leq M_{k+j}$.
\end{lemma}

\begin{lemma}[{\protect\cite[Lemma 3.2]{C-H2/3}}]
\label{L:iterates4} For each $\alpha >0$ (possibly large) and each $\epsilon
>0$ (arbitrarily small), there exists $t>0$ (independent of $M$)
sufficiently small such that 
\begin{equation*}
\forall \;j\geq 1,\;\forall \;M\,,\quad \text{we have}\quad \frac{%
t^{j}(\alpha \log M+j)^{j}}{j!}\leq M^{\epsilon }
\end{equation*}
\end{lemma}

\subsubsection{Step III}

Recall the ending result of Step II,

\begin{eqnarray*}
&&\left\Vert R_{\leqslant M_{k-1}}^{(k-1)}B_{N,1,k}\mathnormal{KIP}%
^{k,p}\right\Vert _{L_{T}^{1}L_{\mathbf{x,x}^{\prime }}^{2}} \\
&\leqslant &(C_{\varepsilon }T^{\frac{1}{2}})\sum_{q=0}^{p}\Bigg\{%
(C_{\varepsilon }T^{\frac{1}{4}})^{q}\sum_{M_{k+q}\geqslant M_{k-1}}\Big[%
\left( \frac{M_{k-1}^{1-2\varepsilon }}{M_{k+q}^{1-2\varepsilon }}\right) \\
&&\times \left\Vert R_{\leqslant M_{k+q}}^{(k+q)}\left( \theta (t_{k+q+1})%
\tilde{B}_{N,many}^{(k+q+1)}\alpha _{N}^{(k+q+1)}\left( t_{k+q+1}\right)
\right) \right\Vert _{X_{-\frac{1}{2}+}^{(k+q)}}\Big]\Bigg\}.
\end{eqnarray*}%
Use Corollary \ref{Corollary:Key of loc},%
\begin{eqnarray*}
&\leqslant &(C_{\varepsilon }T^{\frac{1}{2}})\sum_{q=0}^{p}\Bigg\{%
(C_{\varepsilon }T^{\frac{1}{4}})^{q}\sum_{M_{k+q}\geqslant M_{k-1}}\Big[%
\left( \frac{M_{k-1}^{1-2\varepsilon }}{M_{k+q}^{1-2\varepsilon }}\right) \\
&&\times C^{k+q+1}N^{-\frac{1}{2}+}\min (N^{\beta },M_{k+q})\frac{1}{\sqrt{N}%
}\left\Vert \theta (t_{k+q+1})S_{1}S^{\left( k+q+1\right) }\alpha
_{N}^{(k+q+1)}\right\Vert _{L_{t_{k+q+1}}^{2}L_{\mathbf{x}}^{2}L_{\mathbf{x}%
^{\prime }}^{2}}\Big]\Bigg\},
\end{eqnarray*}%
because there are $(k+q)$ terms inside $\tilde{B}_{N,many}^{(k+q+1)}$.
Rearranging terms%
\begin{eqnarray*}
&\leqslant &C^{k}(C_{\varepsilon }T^{\frac{1}{2}})\sum_{q=0}^{p}\Bigg\{(CT^{%
\frac{1}{4}})^{q}\frac{1}{\sqrt{N}}\left\Vert \theta
(t_{k+q+1})S_{1}S^{\left( k+q+1\right) }\alpha _{N}^{(k+q+1)}\right\Vert
_{L_{t_{k+q+1}}^{2}L_{\mathbf{x}}^{2}L_{\mathbf{x}^{\prime }}^{2}} \\
&&M_{k-1}^{1-2\varepsilon }N^{-\frac{1}{2}+}\sum_{M_{k+q}\geqslant
M_{k-1}}\min (M_{k+q}^{-1+2\varepsilon }N^{\beta },M_{k+q}^{2\varepsilon })%
\Bigg\},
\end{eqnarray*}

We carry out the sum in $M_{k+q}$ by dividing into $M_{k+q}\leqslant
N^{\beta }$ (for which $\min (M_{k+q}^{-1+2\varepsilon }N^{\beta
},M_{k+q}^{2\varepsilon })=M_{j}^{2\epsilon }$) and $M_{k+q}\geqslant
N^{\beta }$ (for which $\min (M_{k+q}^{-1+2\varepsilon }N^{\beta
},M_{k+q}^{2\varepsilon })=M_{k+q}^{-1+2\varepsilon }N^{\beta }$). This
yields%
\begin{eqnarray}
\sum_{M_{k+q}\geqslant M_{k-1}}\min (M_{k+q}^{-1+2\varepsilon }N^{\beta
},M_{k+q}^{2\varepsilon }) &\lesssim &\sum_{N^{\beta }\geqslant
M_{k+q}\geqslant M_{k-1}}\left( ...\right) +\sum_{M_{k+q}\geqslant
M_{k-1},M_{k+q}\geqslant N^{\beta }}\left( ...\right)  \label{sum:key in L-P}
\\
&\lesssim &\sum_{N^{\beta }\geqslant M_{k+q}\geqslant
M_{k-1}}M_{k+q}^{2\epsilon }+\sum_{M_{k+q}\geqslant N^{\beta
}}M_{k+q}^{-1+2\varepsilon }N^{\beta }  \notag \\
&\lesssim &N^{2\epsilon }.  \notag
\end{eqnarray}

\begin{remark}
The above is exactly what we meant by writing "gains one derivative via
Littlewood-Paley" in \S \ref{Sec:org of paper}.
\end{remark}

So we have reached%
\begin{eqnarray*}
&&\left\Vert R_{\leqslant M_{k-1}}^{(k-1)}B_{N,1,k}\mathnormal{KIP}%
^{k,p}\right\Vert _{L_{T}^{1}L_{\mathbf{x,x}^{\prime }}^{2}} \\
&\leqslant &C^{k}(C_{\varepsilon }T^{\frac{1}{2}})\sum_{q=0}^{p}\Bigg\{(CT^{%
\frac{1}{4}})^{q}\frac{1}{\sqrt{N}}\left\Vert \theta
(t_{k+q+1})S_{1}S^{\left( k+q+1\right) }\alpha _{N}^{(k+q+1)}\right\Vert
_{L_{t_{k+q+1}}^{2}L_{\mathbf{x}}^{2}L_{\mathbf{x}^{\prime }}^{2}} \\
&&M_{k-1}^{1-2\varepsilon }N^{-\frac{1}{2}+2\varepsilon }\Bigg\}.
\end{eqnarray*}%
Via Condition $(\ref{bound:k-energy estimate})$ (the energy estimate), it
becomes%
\begin{eqnarray*}
&\leqslant &C^{k}(C_{\varepsilon }T^{\frac{1}{2}})\sum_{q=0}^{p}(CT^{\frac{1%
}{4}})^{q}C_{0}^{k+q+1}M_{k-1}^{1-2\varepsilon }N^{-\frac{1}{2}+2\varepsilon
} \\
&\leq &C^{k}(C_{\varepsilon }T^{\frac{1}{2}})\sum_{q=0}^{\infty }(CT^{\frac{1%
}{4}})^{q}C_{0}^{q+1}M_{k-1}^{1-2\varepsilon }N^{-\frac{1}{2}+2\varepsilon }.
\end{eqnarray*}%
We can then choose a $T$ independent of $M_{k-1}$, $k,$ $p$ and $N$ such
that the infinite series converges. We then have%
\begin{equation*}
\left\Vert R_{\leqslant M_{k-1}}^{(k-1)}B_{N,1,k}\mathnormal{KIP}%
^{k,p}\right\Vert _{L_{T}^{1}L_{\mathbf{x,x}^{\prime }}^{2}}\leqslant
C^{k-1}M_{k-1}^{1-2\varepsilon }N^{-\frac{1}{2}+2\varepsilon }
\end{equation*}%
for some $C$ larger than $C_{0}$. Therefore, on the one hand, there is a $C$
independent of $M_{k-1}$, $k,$ $p,$ and $N$ s.t. given a $M_{k-1}$, there
is\ $N_{0}(M_{k-1})$ which makes 
\begin{equation*}
\left\Vert R_{\leqslant M_{k-1}}^{(k-1)}B_{N,1,k}\mathnormal{KIP}%
^{k,p}\right\Vert _{L_{T}^{1}L_{\mathbf{x,x}^{\prime }}^{2}}\leqslant C^{k-1}%
\text{, for all }N\geqslant N_{0},
\end{equation*}%
on the other hand, 
\begin{equation*}
\left\Vert R_{\leqslant M_{k-1}}^{(k-1)}B_{N,1,k}\mathnormal{KIP}%
^{k,p}\right\Vert _{L_{T}^{1}L_{\mathbf{x,x}^{\prime }}^{2}}\rightarrow 0%
\text{ as }N\rightarrow \infty
\end{equation*}%
which matches Theorem \ref{THM:Convergence} as well. Whence we have finished
the proof of estimate (\ref{estimate:kip}).

\begin{corollary}
\label{Corollary:Key of loc} 
\begin{equation}  \label{E:Str24}
\begin{aligned} \hspace{0.3in}&\hspace{-0.3in} \left\| R_{\leq
M_{k+q}}^{(k+q)} \tilde B_{N,\mathrm{many}}^{(k+q+1)} \alpha_N^{(k+q+1)}
\right\|_{X_{-\frac12+}^{(k+q)}} \\ & \lesssim C^{k+q} \left( N^{-\frac12}
\|S_1 S^{(k+q+1)} \alpha_N^{(k+q+1)} \|_{L_t^2L_{\mathbf{x}\mathbf{x}'}^2}
\right) \left\{ \begin{aligned} & M_{k+q} N^{-\frac12+} \\ &
N^{\beta-\frac12+} \end{aligned} \right. \end{aligned}
\end{equation}
\end{corollary}

\begin{proof}
Recall \eqref{E:B-many-decomp}, which gives the expansion 
\begin{equation}  \label{E:Str26}
\tilde B_{N,\mathrm{many}}^{(k+q)} = \sum_{\ell=1}^{k+q} \tilde B_{N,\mathrm{%
many}, \ell,k+q+1}
\end{equation}
where $\tilde B_{N,\mathrm{many},\ell,k+q+1}$ is defined by \eqref{def:B,loc}
and itself decomposed in Lemma \ref{lemma:term counting in L} into a sum of
at most $8^{k+q}$ terms of the form 
\begin{equation}  \label{E:Str27}
\beta_N^{(k+q)} = \begin{aligned}[t]
&\frac{N-k}{N}\int_{\mathbb{R}^{3}}\tilde{V}_{N}(x_{l}-x_{k+q+1})N^{\beta
-1} w_{0}(N^{\beta }(x_{\sigma }-x_{k+q+1})) \\ &\qquad \qquad \qquad
A_{\sigma }\alpha _{N}^{(k+q+1)}(
\mathbf{x}_{k+q},x_{k+q+1};\mathbf{x}_{k+q}^{\prime
},x_{k+q+1})\,dx_{k+q+1}. \end{aligned}
\end{equation}
Here, $x_\sigma \in \{ x_1, \ldots, x_{k+q+1}, x_1^{\prime },\ldots,
x_{k+q+1}^{\prime }\}\backslash \{x_\ell\}$ and 
\begin{equation*}
A_\sigma = \prod_{\substack{ 1\leq j \leq k+q  \\ j\neq \ell, \; j\neq
\sigma }} Z_jZ_j^{\prime }
\end{equation*}
where $Z_j$ is either $1$ or $N^{\beta-1}w_0(N^\beta(x_j-x_{k+q+1}))$, and
likewise $Z_j^{\prime }$ is either $1$ or $N^{\beta-1}w_0(N^\beta(x_j^{%
\prime }-x_{k+q+1}))$. Since there are $(k+q)$ terms in \eqref{E:Str26} and $%
\leq 8^{k+q}$ terms of the type $\beta_N^{(k+q)}$ in \eqref{E:Str27}, we
multiply by a factor $C^{k+q}$. For each individual term $\beta_N^{(k+q)}$,
the derivatives $\nabla_{x_j}$ for $1\leq j \leq k+q$, $j\neq \ell$, $j\neq
\sigma$ can either land on $Z_jZ_j^{\prime }$ or $\alpha_N^{(k+q+1)}$,
giving $2^{k+q-1}$ terms. Each possibility is accommodated by a suitable
variant of Proposition \ref{P:Str22}. Of course, we actually need to modify %
\eqref{E:Str22} so that it has a $(k+q+1)$-component density (as opposed to $%
4$) and multiple factors of the type $f_N(x_1-x_4)$ in \eqref{E:Str22}, but
these modifications are straightforward and amount to bookkeeping. The
remaining coordinates act as ``passive variables'' and are placed in $L^2$
on the inside of the estimates, and otherwise do not play any role.
\end{proof}

\subsection{Estimate for the Potential Part}

Repeating Steps I and II in the treatment of the k-body interaction part, we
have%
\begin{eqnarray*}
&&\left\Vert R_{\leqslant M_{k-1}}^{(k-1)}B_{N,1,k}\mathnormal{PP}%
^{k,p}\right\Vert _{L_{T}^{1}L_{\mathbf{x,x}^{\prime }}^{2}} \\
&\leqslant &(C_{\varepsilon }T^{\frac{1}{2}})\sum_{q=0}^{p}\Bigg\{%
(C_{\varepsilon }T^{\frac{1}{4}})^{q}\sum_{M_{k+q}\geqslant M_{k-1}}\Big[%
\left( \frac{M_{k-1}^{1-2\varepsilon }}{M_{k+q}^{1-2\varepsilon }}\right) \\
&&\times \left\Vert R_{\leqslant M_{k+q}}^{(k+q)}\left( \theta (t_{k+q+1})%
\tilde{V}_{N}^{(k+q)}\alpha _{N}^{(k+q)}\left( t_{k+q+1}\right) \right)
\right\Vert _{X_{-\frac{1}{2}+}^{(k+q)}}\Big]\Bigg\}.
\end{eqnarray*}%
Recall%
\begin{equation*}
\tilde{V}_{N}^{(k)}\alpha _{N}^{(k)}=(A_{N}^{(k)}-A_{N}^{(k)^{\prime
}})\alpha _{N}^{(k)}+(E_{N}^{(k)}-E_{N}^{(k)^{\prime }})\alpha _{N}^{(k)}.
\end{equation*}%
From here on out, we will call $(A_{N}^{(k)}-A_{N}^{(k)^{\prime }})\alpha
_{N}^{(k)}$ the three-body potential term and $(E_{N}^{(k)}-E_{N}^{(k)^{%
\prime }})\alpha _{N}^{(k)}$ the two-body error term.

By Step III in the estimate of the $k$-body interaction term, it suffices to
prove the following two corollaries.

\begin{corollary}
\label{Corollary:Key of 3-body}Recall 
\begin{equation*}
A_{N,i,j,\ell }^{(k+q)}\alpha ^{(k+q)}=\frac{\nabla _{x_{\ell }}G_{N,\ell
,i}\cdot \nabla _{x_{\ell }}G_{N,\ell ,j}}{G_{N,\ell ,i}\;G_{N,\ell ,j}}%
\alpha ^{(k+q)},
\end{equation*}%
as defined in \eqref{formula: a piece of three body}. Then 
\begin{equation*}
\Vert R_{\leq M_{k+q}}^{(k+q)}A_{N,i,j,\ell }\alpha _{N}^{(k+q)}\Vert _{X_{-%
\frac{1}{2}+}^{(k+q)}}\lesssim \Vert S^{(k+q)}\alpha _{N}^{(k+q)}\Vert
_{L_{t}^{2}L_{\mathbf{x}\mathbf{x}^{\prime }}^{2}}\left\{ \begin{aligned} &
N^{3\beta-2} \\ & M_{k+q}^3 N^{-2+} \end{aligned}\right.
\end{equation*}
\end{corollary}

\begin{corollary}
\label{Corollary:key for 1st order term}Recall 
\begin{equation*}
E_{N,j,l}^{(k+q)}\alpha ^{(k+q)}=\frac{\nabla _{x_{\ell }}G_{N,j,\ell }}{%
G_{N,j,\ell }}\cdot \nabla _{x_{\ell }}\alpha ^{(k+q)}
\end{equation*}%
as defined in (\ref{formula: a piece of 1st order}), we have 
\begin{equation}
\begin{aligned} \hspace{0.3in}&\hspace{-0.3in} \left\Vert R^{(k+q)}_{\leq
M_{k+q}} \left( E_{N,j,l}^{(k)}\alpha ^{(k)}\left( t_{k+1}\right) \right)
\right\Vert _{X_{-\frac{1}{2}+}^{(k+q)}} \\ &\lesssim \Big( N^{-\frac12}
\|S_1S^{(k+q)} \alpha\|_{L_t^2L_{\mathbf{x}_{k+q}\mathbf{x}_{k+q}'}^2} +
\|S^{(k)} \alpha\|_{L_t^2L_{\mathbf{x}_{k+q}\mathbf{x}_{k+q}'}^2} \Big)
\left\{ \begin{aligned} &M_{k+q} N^{\frac{\beta}{2}-\frac34} \\
&N^{\frac{3\beta}{2}-\frac34} \end{aligned} \right. \end{aligned}
\end{equation}%
where, for convenience, we have assumed that $\beta >\frac{1}{2}$.
\end{corollary}

Then one merely needs to estimate the following two sums:%
\begin{equation*}
N^{-2+}\sum_{M_{k+q}\geqslant M_{k-1}}\min \left( M_{k+q}^{-1+2\varepsilon
}N^{3\beta },M_{k+q}^{2+2\varepsilon }\right) ,
\end{equation*}%
and%
\begin{equation*}
N^{\frac{1}{2}\beta -\frac{3}{4}+}\sum_{M_{k+q}\geqslant M_{k-1}}\min \left(
M_{k+q}^{-1+2\varepsilon }N^{\beta },M_{k+q}^{2\varepsilon }\right) .
\end{equation*}%
In fact, separate the above sums at $M_{k+q}\geqslant N^{\beta }$ and $%
M_{k+q}\leqslant N^{\beta }$, then use the same method as in estimate (\ref%
{sum:key in L-P}), we get to%
\begin{eqnarray*}
N^{-2+}\sum_{M_{k+q}\geqslant M_{k-1}}\min \left( M_{k+q}^{-1+2\varepsilon
}N^{3\beta },M_{k+q}^{2+2\varepsilon }\right) &\lesssim &N^{-2+}N^{2\beta
+2\varepsilon } \\
N^{\frac{1}{2}\beta -\frac{3}{4}+}\sum_{M_{k+q}\geqslant M_{k-1}}\min \left(
M_{k+q}^{-1+2\varepsilon }N^{\beta },M_{k+q}^{2\varepsilon }\right)
&\lesssim &N^{\frac{1}{2}\beta -\frac{3}{4}+}N^{2\varepsilon }
\end{eqnarray*}%
which is enough to conclude the estimates of the potential part for $\beta
\in \left( 0,1\right) $.

\begin{remark}
We remark that the estimate for the three-body interaction term is the only
place in this paper which requires $\beta <1$.
\end{remark}

\begin{proof}[Proof of Corollary \protect\ref{Corollary:Key of 3-body}]
Since $\nabla _{x_{\mu }}$ and $\nabla _{x_{\mu }^{\prime }}$ move directly
onto $\alpha _{N}^{(k+q)}$, for $\mu \in \{1,\ldots ,k+q\}\backslash
\{i,j,\ell \}$, it suffices to use the obvious extension of Proposition \ref%
{P:Str29} where $\{1,2,3\}$ is replaced by $\{\ell ,i,j\}$, $\alpha ^{(3)}$
is replaced by $\alpha ^{(k+q)}$, and $X_{-\frac{1}{2}+}^{(3)}$ is replaced
by $X_{-\frac{1}{2}+}^{(k+q)}$. Note that 
\begin{equation*}
A_{Nij\ell }=N^{-2\beta -2}U_{N}(x_{\ell }-x_{i})U_{N}(x_{\ell }-x_{j})
\end{equation*}%
where 
\begin{equation*}
U(x)=\frac{(\nabla w_{0})(x)}{1-N^{\beta -1}w_{0}(x)}
\end{equation*}%
Note that 
\begin{equation*}
\nabla U(x)=\frac{\nabla ^{2}w_{0}}{(1-w_{0}(x))^{2}}+N^{\beta -1}\left( 
\frac{1}{1-w_{0}(x)}\right) ^{2}
\end{equation*}%
and 
\begin{equation*}
|U(x)|\lesssim \langle x\rangle ^{-2}\,,\qquad |\nabla U(x)|\lesssim \langle
x\rangle ^{-3}\,,\qquad |\nabla ^{2}U(x)|\lesssim \langle x\rangle ^{-4}
\end{equation*}%
uniformly in $N$. Hence $U$, $\nabla U$, and $\nabla ^{2}U$ all belong to $%
L^{p}$ for $p>\frac{3}{2}$ (uniformly in $N$).
\end{proof}

\begin{proof}[Proof of Corollary \protect\ref{Corollary:key for 1st order
term}]
Note that 
\begin{equation*}
\frac{\nabla_{x_\ell} G_{N,j,\ell}}{G_{N,j,\ell}} = N^{-\beta-1}
U_N(x_j-x_\ell)
\end{equation*}
where 
\begin{equation*}
U(x) = \frac{\nabla w_0(x)}{1-N^{\beta-1}w_0(x)}
\end{equation*}
We then appeal to the straightforward generalization of Proposition \ref%
{P:2body-potential-prep} to $(k+q)$-level density, noting that $%
|U(x)|\lesssim \langle x \rangle^{-2}$, $|\nabla U(x)| \lesssim \langle x
\rangle^{-3}$, and $|\nabla^2 U(x)|\lesssim \langle x \rangle^{-4}$,
uniformly in $N$, so $C_U<\infty$ and independent of $N$.
\end{proof}

\section{Collapsing and Strichartz Estimates}

\label{S:Strichartz-X}

Define the norm\footnote{%
To be precise, this $X_{b}$ should be written as $X_{0,b}$ in the usual
notation for the $X_{s,b}$ spaces. Since we are not using the $s$ in $%
X_{s,b} $, we write it as $X_{b}$.} 
\begin{equation*}
\Vert \alpha ^{(k)}\Vert _{X_{b}^{(k)}}=\left( \int \langle \tau +\left\vert 
\mathbf{\xi }_{k}\right\vert ^{2}-\left\vert \mathbf{\xi }_{k}^{\prime
}\right\vert ^{2}\rangle ^{2b}\left\vert \hat{\alpha}^{(k)}(\tau ,\mathbf{%
\xi }_{k},\mathbf{\xi }_{k}^{\prime })\right\vert ^{2}\,d\tau \,d\mathbf{\xi 
}_{k}\,d\mathbf{\xi }_{k}^{\prime }\right) ^{1/2}
\end{equation*}

We will use the case $b=\frac{1}{2}+$ of the following lemma.

\begin{lemma}[{\protect\cite[Lemma 4.1]{C-H2/3}}]
\label{Lemma:b to b-1}Let $\frac{1}{2}<b<1$ and $\theta (t)$ be a smooth
cutoff. Then 
\begin{equation}
\left\Vert \theta (t)\int_{0}^{t}U^{(k)}(t-s)\beta ^{(k)}(s)\,ds\right\Vert
_{X_{b}^{(k)}}\lesssim \Vert \beta ^{(k)}\Vert _{X_{b-1}^{(k)}}
\label{E:X-1}
\end{equation}
\end{lemma}

\begin{lemma}[{\protect\cite[Lemma 4.4]{C-H2/3}}]
\label{Lemma:LocalizedKM} For each $\varepsilon >0$, there is a $%
C_{\varepsilon }$ independent of $M_{k},j,k$, and $N$ such that 
\begin{equation*}
\Vert R_{\leqslant M_{k}}^{(k)}\tilde{B}_{N,j,k+1}U^{(k+1)}(t)f^{(k+1)}\Vert
_{L_{t}^{2}L_{\mathbf{x},\mathbf{x}^{\prime }}^{2}}\leqslant C_{\varepsilon
}\left\Vert \tilde{V}\right\Vert _{L^{1}}\sum_{M_{k+1}\geqslant M_{k}}\left( 
\frac{M_{k}}{M_{k+1}}\right) ^{1-\varepsilon }\left\Vert R_{\leqslant
M_{k+1}}^{(k+1)}f^{(k+1)}\right\Vert _{L_{\mathbf{x},\mathbf{x}^{\prime
}}^{2}}
\end{equation*}%
where the sum on the right is in $M_{k+1}$, over dyads such that $%
M_{k+1}\geqslant M_{k}$.
\end{lemma}

\begin{lemma}[{\protect\cite[Lemma 4.5]{C-H2/3}}]
\label{Lemma:LocalizedKMWithX_b}For each $\varepsilon >0$, there is a $%
C_{\varepsilon }$ independent of $M_{k},j,k$, and $N$ such that 
\begin{equation*}
\Vert R_{\leqslant M_{k}}^{(k)}\tilde{B}_{N,j,k+1}\alpha ^{(k+1)}\Vert
_{L_{t}^{2}L_{\mathbf{x},\mathbf{x}^{\prime }}^{2}}\leqslant C_{\varepsilon
}\sum_{M_{k+1}\geqslant M_{k}}\left( \frac{M_{k}}{M_{k+1}}\right)
^{1-\varepsilon }\left\Vert R_{\leqslant M_{k+1}}^{(k+1)}\alpha
^{(k+1)}\right\Vert _{X_{\frac{1}{2}+}^{(k)}}.
\end{equation*}%
where the sum on the right is in $M_{k+1}$, over dyads such that $%
M_{k+1}\geqslant M_{k}$.
\end{lemma}

The 3D endpoint Strichartz estimate directly yields the following
multiparticle estimate: 
\begin{equation*}
\| \beta^{(k)}\|_{X^{(k)}_{-\frac12+}} \lesssim \|\beta^{(k)}
\|_{L_t^2L_{x_1}^{\frac65+}L^2_c}
\end{equation*}
where $c$ stands for the remaining spatial coordinates $(x_2,\ldots,x_k,
x_1^{\prime }, \ldots, x_k^{\prime })$. However, when $\beta^{(k)} =
V(x_1-x_2) \gamma^{(k)}$, this estimate does not allow us to put $V$ in $%
L^{\frac65+}$ since the $L^2_{x_2}$ norm comes before the $%
L^{\frac65+}_{x_1} $ norm. In order to effectively put the $L^2_{x_2}$ norm 
\emph{after} the $L^{\frac65+}_{x_1}$ norm, we need to translate coordinates
before applying the Strichartz estimate. This maneuver was introduced in our
earlier paper \cite[Lemma 4.6]{C-H3Dto2D}. We restate the relevant estimate
in the following lemma.

Since we will need to deal with Fourier transforms in only selected
coordinates, we introduce the following notation: $\mathcal{F}_{0}$ denotes
the Fourier transform in $t$, $\mathcal{F}_{j}$ denotes the Fourier
transform in $x_{j}$, and $\mathcal{F}_{j^{\prime }}$ denotes Fourier
transform in $x_{j}^{\prime }$. Fourier transforms in multiple coordinates
will be denoted as combined subscripts -- for example, $\mathcal{F}%
_{01^{\prime }}=\mathcal{F}_{0}\mathcal{F}_{1^{\prime }}$ denotes the
Fourier transform in $t$ and $x_{1}^{\prime }$.

\begin{lemma}[3D endpoint Strichartz in transformed coordinates]
\label{Lemma:TheStrichartzEstimate} Let 
\begin{equation*}
T_1f(x_1,x_2) = f(x_1+x_2,x_2)
\end{equation*}
\begin{equation*}
T_2f(x_1,x_2) = f(x_1,x_2+x_1)
\end{equation*}
Then 
\begin{equation}  \label{E:Str1}
\| \beta^{(k)}\|_{X^{(k)}_{-\frac12+}} \lesssim \left\{ \begin{aligned} &\|
(\mathcal{F}_2 T_1 \beta^{(k)})(t,x_1,\xi_2) \|_{L_t^2 L_{\xi_2}^2
L_{x_1}^{\frac65+} L^2_c}\\ &\| (\mathcal{F}_1 T_2 \beta^{(k)})(t,\xi_1,x_2)
\|_{L_t^2 L_{\xi_1}^2 L_{x_2}^{\frac65+} L^2_c} \end{aligned} \right.
\end{equation}
where in each case $c$ stands for ``complementary coordinates'',
specifically coordinates $(x_3,\ldots, x_k, x_1^{\prime }, \ldots,
x_k^{\prime })$.
\end{lemma}

\begin{lemma}[H\"older and Sobolev]
\label{L:3DHolderSobolev} If 
\begin{equation}  \label{E:Str12}
\beta^{(k)}(t,x_1,x_2) = V(x_1-x_2) \gamma^{(k)}(t,x_1,x_2)
\end{equation}
then 
\begin{equation}  \label{E:Str3}
\| (\mathcal{F}_2 T_1 \beta^{(k)})(t,x_1,\xi_2) \|_{L_t^2L_{\xi_2}^2
L_{x_1}^{\frac65+} L^2_c} \lesssim \left\{ \begin{aligned}
&\|V\|_{L^{\frac65+}} \| \langle \nabla_{x_1}\rangle^{\frac32} \gamma^{(k)}
\|_{L_t^2L_{\mathbf{x}\mathbf{x}'}^2} \\ &\|V\|_{L^{\frac32+}} \|
\nabla_{x_1}\gamma^{(k)} \|_{L_t^2L_{\mathbf{x}\mathbf{x}'}^2} \\
&\|V\|_{L^{3+}} \| \gamma^{(k)} \|_{L_t^2L_{\mathbf{x}\mathbf{x}'}^2}
\end{aligned} \right.
\end{equation}

\begin{equation}  \label{E:Str4}
\| (\mathcal{F}_1 T_2 \beta^{(k)})(t,\xi_1,x_2) \|_{L_t^2L_{\xi_1}^2
L_{x_2}^{\frac65+} L^2_c} \lesssim \left\{ \begin{aligned}
&\|V\|_{L^{\frac65+}} \| \langle \nabla_{x_2}\rangle^{\frac32} \gamma^{(k)}
\|_{L_t^2L_{\mathbf{x}\mathbf{x}'}^2} \\ &\|V\|_{L^{\frac32+}} \|
\nabla_{x_2} \gamma^{(k)} \|_{L_t^2L_{\mathbf{x}\mathbf{x}'}^2} \\
&\|V\|_{L^{3+}} \| \gamma^{(k)} \|_{L_t^2L_{\mathbf{x}\mathbf{x}'}^2}
\end{aligned} \right.
\end{equation}
\end{lemma}

\begin{proof}
Consider \eqref{E:Str3}. By \eqref{E:Str12}, 
\begin{equation*}
(T_{1}\beta ^{(k)})(t,x_{1},x_{2})=V(x_{1})(T_{1}\gamma
^{(k)})(t,x_{1},x_{2})
\end{equation*}%
and hence, applying $\mathcal{F}_{2}$, we obtain 
\begin{equation*}
(\mathcal{F}_{2}T_{1}\beta ^{(k)})(t,x_{1},\xi _{2})=V(x_{1})(\mathcal{F}%
_{2}T_{1}\gamma ^{(k)})(t,x_{1},\xi _{2})
\end{equation*}%
Applying H\"{o}lder, 
\begin{equation*}
\Vert (\mathcal{F}_{2}T_{1}\beta ^{(k)})(t,x_{1},\xi _{2})\Vert _{L_{x_{1}}^{%
\frac{6}{5}+}L_{c}^{2}}\leq \Vert V\Vert _{L^{\frac{6}{5}+}}\Vert (\mathcal{F%
}_{2}T_{1}\gamma ^{(k)})(t,x_{1},\xi _{2})\Vert _{L_{x_{1}}^{\infty
-}L_{c}^{2}}
\end{equation*}%
By Sobolev, 
\begin{align*}
\Vert (\mathcal{F}_{2}T_{1}\beta ^{(k)})(t,x_{1},\xi _{2})\Vert _{L_{x_{1}}^{%
\frac{6}{5}+}L_{c}^{2}}& \leq \Vert V\Vert _{L^{\frac{6}{5}+}}\Vert \langle
\nabla _{x_{1}}\rangle ^{\frac{3}{2}}(\mathcal{F}_{2}T_{1}\gamma
^{(k)})(t,x_{1},\xi _{2})\Vert _{L_{x_{1}}^{2}L_{c}^{2}} \\
& =\Vert V\Vert _{L^{\frac{6}{5}+}}\Vert \mathcal{F}_{2}\langle \nabla
_{x_{1}}\rangle ^{\frac{3}{2}}(T_{1}\gamma ^{(k)})(t,x_{1},\xi _{2})\Vert
_{L_{x_{1}}^{2}L_{c}^{2}}
\end{align*}%
Applying the $L_{\xi _{2}}^{2}$ norm and Plancherel in $x_{2}$, 
\begin{equation*}
\Vert (\mathcal{F}_{2}T_{1}\beta ^{(k)})(t,x_{1},\xi _{2})\Vert _{L_{\xi
_{2}}^{2}L_{x_{1}}^{\frac{6}{5}+}L_{c}^{2}}=\Vert V\Vert _{L^{\frac{6}{5}%
+}}\Vert \langle \nabla _{x_{1}}\rangle ^{\frac{3}{2}}(T_{1}\gamma
^{(k)})(t,x_{1},x_{2})\Vert _{L_{\mathbf{x}\mathbf{x}^{\prime }}^{2}}
\end{equation*}%
Reviewing the definition of $T_{1}$, we see that 
\begin{equation*}
=\Vert V\Vert _{L^{\frac{6}{5}+}}\Vert \langle \nabla _{x_{1}}\rangle ^{%
\frac{3}{2}}\gamma ^{(k)}(t,x_{1},x_{2})\Vert _{L_{\mathbf{x}\mathbf{x}%
^{\prime }}^{2}}.
\end{equation*}%
The other cases are similar.
\end{proof}

Using frequency localization, we can share derivatives between two
coordinates, as in the following corollary.

\begin{corollary}
\label{corollary:basic corollary for 3D endpoint}If $\gamma ^{(k)}$ is
symmetric and 
\begin{equation*}
\beta ^{(k)}(t,x_{1},x_{2})=V(x_{1}-x_{2})\gamma ^{(k)}(t,x_{1},x_{2})
\end{equation*}%
then 
\begin{equation}  \label{estimate:basic corollary for 3D endpoint}
\Vert \beta ^{(k)}\Vert _{X_{-\frac{1}{2}+}^{(k)}} \lesssim \left\{ %
\begin{aligned} &\Vert V\Vert _{L^{\frac{6}{5}+}}\Vert \langle \nabla
_{x_{1}}\rangle ^{\frac{3}{4}}\langle \nabla _{x_{2}}\rangle
^{\frac{3}{4}}\gamma ^{(k)}\Vert _{L_{t}^{2}L_{\mathbf{x}\mathbf{x}^{\prime
}}^{2}} \\ &\Vert V\Vert _{L^{\frac{3}{2}+}}\Vert \langle \nabla
_{x_{i}}\rangle \gamma ^{(k)}\Vert
_{L_{t}^{2}L_{\mathbf{x}\mathbf{x}^{\prime }}^{2}}\text{ with }i=1,2 \\
&\Vert V\Vert _{L^{3+}}\Vert \gamma ^{(k)}\Vert
_{L_{t}^{2}L_{\mathbf{x}\mathbf{x}^{\prime }}^{2}} \end{aligned} \right.
\end{equation}
\end{corollary}

\begin{proof}
We need only to prove the first inequality of (\ref{estimate:basic corollary
for 3D endpoint}). The other two are directly from Lemma \ref%
{L:3DHolderSobolev} and the fact that $\gamma ^{(k)}$ is symmetric i.e. $%
\left\Vert \langle \nabla _{x_{1}}\rangle \gamma ^{(k)}\right\Vert
=\left\Vert \langle \nabla _{x_{2}}\rangle \gamma ^{(k)}\right\Vert $.

Split $\gamma ^{(k)}$ according to the relative magnitude of the $\xi _{1}$
and $\xi _{2}$ frequencies: 
\begin{equation*}
\gamma ^{(k)}=P_{|\xi _{1}|\leq |\xi _{2}|}\gamma ^{(k)}+P_{|\xi _{2}|\leq
|\xi _{1}|}\gamma ^{(k)}
\end{equation*}%
and define 
\begin{equation*}
\beta _{1\leq 2}^{(k)}\overset{\mathrm{def}}{=}V(x_{1}-x_{2})P_{|\xi
_{1}|\leq |\xi _{2}|}\gamma ^{(k)}
\end{equation*}%
\begin{equation*}
\beta _{2\leq 1}^{(k)}\overset{\mathrm{def}}{=}V(x_{1}-x_{2})P_{|\xi
_{2}|\leq |\xi _{1}|}\gamma ^{(k)}
\end{equation*}%
so that 
\begin{equation*}
\beta ^{(k)}=\beta _{1\leq 2}^{(k)}+\beta _{2\leq 1}^{(k)}
\end{equation*}%
For the $\beta _{1\leq 2}^{(k)}$ piece, use the first estimate in %
\eqref{E:Str1} together with the first estimate of \eqref{E:Str3} to obtain 
\begin{align*}
\Vert \beta _{1\leq 2}^{(k)}\Vert _{X_{-\frac{1}{2}+}^{(k)}}& \lesssim \Vert
V\Vert _{L^{\frac{6}{5}+}}\Vert \langle \nabla _{x_{1}}\rangle ^{\frac{3}{2}%
}P_{|\xi _{1}|\leq |\xi _{2}|}\gamma ^{(k)}\Vert _{L_{t}^{2}L_{\mathbf{x}%
\mathbf{x}^{\prime }}^{2}} \\
& \lesssim \Vert V\Vert _{L^{\frac{6}{5}+}}\Vert \langle \nabla
_{x_{1}}\rangle ^{\frac{3}{4}}\langle \nabla _{x_{2}}\rangle ^{\frac{3}{4}%
}P_{|\xi _{1}|\leq |\xi _{2}|}\gamma ^{(k)}\Vert _{L_{t}^{2}L_{\mathbf{x}%
\mathbf{x}^{\prime }}^{2}} \\
& \lesssim \Vert V\Vert _{L^{\frac{6}{5}+}}\Vert \langle \nabla
_{x_{1}}\rangle ^{\frac{3}{4}}\langle \nabla _{x_{2}}\rangle ^{\frac{3}{4}%
}\gamma ^{(k)}\Vert _{L_{t}^{2}L_{\mathbf{x}\mathbf{x}^{\prime }}^{2}}
\end{align*}%
where, in the middle line, we used the frequency restriction to $|\xi
_{1}|\leq |\xi _{2}|$.

For the $\beta _{2\leq 1}^{(k)}$ piece, use the second estimate in %
\eqref{E:Str1} together with the first estimate of \eqref{E:Str4}, and
proceed in an analogous fashion to obtain 
\begin{equation*}
\Vert \beta _{2\leq 1}^{(k)}\Vert _{X_{-\frac{1}{2}+}^{(k)}}\lesssim \Vert
V\Vert _{L^{\frac{6}{5}+}}\Vert \langle \nabla _{x_{1}}\rangle ^{\frac{3}{4}%
}\langle \nabla _{x_{2}}\rangle ^{\frac{3}{4}}\gamma ^{(k)}\Vert
_{L_{t}^{2}L_{\mathbf{x}\mathbf{x}^{\prime }}^{2}}
\end{equation*}
\end{proof}

Using Corollary \ref{corollary:basic corollary for 3D endpoint}, we can
prove the following proposition which will be for the first order term in
the $PP$ estimate.

\begin{proposition}
\label{P:2body-potential-prep} For any $U(x)$, let $U_N(x)=N^{3\beta}
U(N^\beta x)$. Then 
\begin{equation}  \label{E:Str32}
\begin{aligned} \hspace{0.3in}&\hspace{-0.3in} \| R_{\leq M_2}^{(2)} \Big(
U_N(x_1-x_2) \nabla_{x_2} \alpha^{(2)}\Big) \|_{X_{-\frac12+}^{(2)}} \\
&\lesssim C_U \Big( N^{-\frac12} \|S_1S^{(2)}
\alpha\|_{L_t^2L_{\mathbf{x}_2\mathbf{x}_2'}^2} + \|S^{(2)}
\alpha\|_{L_t^2L_{\mathbf{x}_2\mathbf{x}_2'}^2} \Big) \left\{
\begin{aligned} &M_2 (N^{\frac{3\beta}{2}+\frac14+}+N^{\beta+\frac12+}) \\
&(N^{\frac{5\beta}{2}+\frac14+}+N^{2\beta+\frac12+}) \end{aligned} \right.
\end{aligned}
\end{equation}
where 
\begin{equation*}
C_U = \| \nabla^2 U\|_{L^{\frac65+}}+ \| \nabla U\|_{L^{\frac65+}\cap
L^{\frac32+}} + \|U\|_{L^{\frac32+}\cap L^{3+}}
\end{equation*}
\end{proposition}

\begin{proof}
We begin by proving the first estimate of \eqref{E:Str32}. Let 
\begin{equation*}
\beta^{(2)} = R^{(2)}_{\leq M_2} \Big( U_N(x_1-x_2) \nabla_{x_2}
\alpha^{(2)} \Big) = P^{(2)}_{\leq M_2} \nabla_{x_1}\nabla_{x_2} \Big( %
U_N(x_1-x_2) \nabla_{x_2} \alpha^{(2)} \Big) = \mathrm{A}+\mathrm{B}
\end{equation*}
where $\mathrm{A}$ and $\mathrm{B}$ are produced by distributing $%
\nabla_{x_1}$ into the product: 
\begin{equation*}
\mathrm{A} = N^\beta P^{(2)}_{\leq M_2} \nabla_{x_2} \Big( (\nabla
U)_N(x_1-x_2) \nabla_{x_2} \alpha^{(2)} \Big)
\end{equation*}
\begin{equation*}
\mathrm{B} = P^{(2)}_{\leq M_2} \nabla_{x_2} \Big( U_N(x_1-x_2) \nabla_{x_1}
\nabla_{x_2} \alpha^{(2)} \Big)
\end{equation*}
Using $P^{(2)}_{\leq M_2} \nabla_{x_2} \leq M_2$, we have 
\begin{equation*}
\| \mathrm{A} \|_{X_{-\frac12+}} \lesssim M_2 N^\beta \| (\nabla
U)_N(x_1-x_2) \nabla_{x_2} \alpha^{(2)} \|_{X_{-\frac12+}}
\end{equation*}
By the first estimate of \eqref{E:Str1} combined with the first estimate of %
\eqref{E:Str3}, we obtain 
\begin{align*}
&\lesssim M_2 N^\beta \| (\nabla U)_N \|_{L^{\frac65+}} \| \langle
\nabla_{x_1} \rangle^{\frac32} \langle \nabla_{x_2} \rangle \alpha^{(2)}
\|_{L_t^2L_{\mathbf{x}_2\mathbf{x}_2^{\prime }}^2} \\
&= M_2 N^{\frac{3\beta}{2}+\frac14+} \|\nabla U \|_{L^{\frac65+}} \Big( %
N^{-\frac14} \|\langle \nabla_{x_1}\rangle^{\frac32} \langle
\nabla_{x_2}\rangle \alpha^{(2)} \|_{L_t^2L_{\mathbf{x}_2\mathbf{x}%
_2^{\prime }}^2} \Big)
\end{align*}
Using $P^{(2)}_{\leq M_2} \nabla_{x_2} \leq M_2$, we have 
\begin{equation*}
\| \mathrm{B} \|_{X_{-\frac12+}} \lesssim M_2 \| U_N(x_1-x_2)
\nabla_{x_1}\nabla_{x_2} \alpha^{(2)} \|_{X_{-\frac12+}}
\end{equation*}
By the second estimate of \eqref{estimate:basic corollary for 3D endpoint},
we obtain 
\begin{align*}
&\lesssim M_2 \|U_N\|_{L^{\frac32+}} \| \langle \nabla_{x_1}\rangle^2\langle
\nabla_{x_2}\rangle \alpha^{(2)} \|_{L_t^2L_{\mathbf{x}_2\mathbf{x}%
_2^{\prime }}^2} \\
& = M_2 N^{\beta+\frac12+} \|U\|_{L^{\frac32+}}\Big( N^{-\frac12} \|
|\nabla_{x_1}|^{\frac32} |\nabla_{x_2}| \alpha^{(2)} \|_{L_t^2L_{\mathbf{x}_2%
\mathbf{x}_2^{\prime }}^2} \Big)
\end{align*}
Now we turn to the second estimate of \eqref{E:Str32}. In this case, we
distribute both $\nabla_{x_1}$ and $\nabla_{x_2}$ into the product to obtain 
$4$ terms 
\begin{align*}
\beta^{(2)} &= R^{(2)}_{\leq M_2} \Big( U_N(x_1-x_2) \nabla_{x_2}
\alpha^{(2)} \Big) = P^{(2)}_{\leq M_2} \nabla_{x_1}\nabla_{x_2} \Big( %
U_N(x_1-x_2) \nabla_{x_2} \alpha^{(2)} \Big) \\
&= \mathrm{C}+\mathrm{D}+\mathrm{E}+\mathrm{F}
\end{align*}
where 
\begin{equation*}
\mathrm{C} = -N^{2\beta} P^{(2)}_{\leq M_2} (\nabla^2 U)_N(x_1-x_2)
\nabla_{x_2}\alpha^{(2)}
\end{equation*}
\begin{equation*}
\mathrm{D} = N^{\beta} P^{(2)}_{\leq M_2} (\nabla U)_N(x_1-x_2) \nabla_{x_1}
\nabla_{x_2}\alpha^{(2)}
\end{equation*}
\begin{equation*}
\mathrm{E} = -N^{\beta} P^{(2)}_{\leq M_2} (\nabla U)_N(x_1-x_2)
\nabla_{x_2}^2\alpha^{(2)}
\end{equation*}
\begin{equation*}
\mathrm{F} = P^{(2)}_{\leq M_2} U_N(x_1-x_2) \nabla_{x_1} \nabla_{x_2}^2
\alpha^{(2)}
\end{equation*}
By the first estimate in \eqref{E:Str1} followed by the first estimate in %
\eqref{E:Str3}, 
\begin{align*}
\| \mathrm{C} \|_{X_{-\frac12+}} &\lesssim N^{2\beta} \|(\nabla^2
U)_N(x_1-x_2) \nabla_{x_2} \alpha^{(2)} \|_{X_{-\frac12+}} \\
&\lesssim N^{2\beta} \| (\nabla^2 U)_N \|_{L^{\frac65+}} \| \langle
\nabla_{x_1} \rangle^{\frac32} \nabla_{x_2} \alpha^{(2)} \|_{L_t^2L_{\mathbf{%
x}_2\mathbf{x}_2^{\prime }}^2} \\
&= N^{\frac{5\beta}{2}+\frac14+} \|\nabla^2 U\|_{L^{\frac65+}} \Big( %
N^{-\frac14} \| \langle \nabla_{x_1} \rangle^{\frac32} \nabla_{x_2}
\alpha^{(2)} \|_{L_t^2L_{\mathbf{x}_2\mathbf{x}_2^{\prime }}^2} \Big)
\end{align*}
By the second estimate of \eqref{estimate:basic corollary for 3D endpoint} 
\begin{align*}
\| \mathrm{D} \|_{X_{-\frac12+}} &\lesssim N^\beta \|(\nabla U)_N
\|_{L^{\frac32+}} \| \langle \nabla_{x_1} \rangle^2 \nabla_{x_2}
\alpha^{(2)} \|_{L_t^2L_{\mathbf{x}_2\mathbf{x}_2^{\prime }}^2} \\
&\lesssim N^{2\beta+\frac12+} \| \nabla U\|_{L^{\frac32+}} \Big( %
N^{-\frac12} \| \langle \nabla_{x_1} \rangle^2 \nabla_{x_2} \alpha^{(2)}
\|_{L_t^2L_{\mathbf{x}_2\mathbf{x}_2^{\prime }}^2} \Big)
\end{align*}
By the second estimate of \eqref{estimate:basic corollary for 3D endpoint}.
The treatment of $\mathrm{E}$ is nearly identical. By the third estimate of %
\eqref{estimate:basic corollary for 3D endpoint} 
\begin{align*}
\| \mathrm{F} \|_{X_{-\frac12+}} &\lesssim \|U_N \|_{L^{3+}} \| \langle
\nabla_{x_1} \rangle^2 \nabla_{x_2} \alpha^{(2)} \|_{L_t^2L_{\mathbf{x}_2%
\mathbf{x}_2^{\prime }}^2} \\
&\lesssim N^{2\beta+\frac12+} \| U\|_{L^{3+}} \Big( N^{-\frac12} \| \langle
\nabla_{x_1} \rangle^2 \nabla_{x_2} \alpha^{(2)} \|_{L_t^2L_{\mathbf{x}_2%
\mathbf{x}_2^{\prime }}^2} \Big)
\end{align*}
\end{proof}

We now provide 6D analogues to the above coordinate translated 3D Strichartz
estimate in Lemma \ref{Lemma:TheStrichartzEstimate} and the associated
H\"older and Sobolev estimates in Lemma \ref{L:3DHolderSobolev}. These 6D
estimates are essential to optimally distribute the derivatives in
three-body estimates.

\begin{lemma}[6D endpoint Strichartz in transformed coordinates]
\footnote{%
It was first observed by X.C. \cite{Chen2ndOrder} in the Hartree setting
that the 6D retarded endpoint Strichartz estimate helps to deal with
three-body interactions and shows that three-body interactions are "better"
than two-body interactions. However, the problem we are discussing here
provides a much deeper and much more substantial explanation to this
phenomenon.} Let 
\begin{equation*}
T_{12}f(x_1,x_2,x_3) = f(x_1+x_3,x_2+x_3,x_3)
\end{equation*}
\begin{equation*}
T_{23}f(x_1,x_2,x_3) = f(x_1,x_2+x_1,x_3+x_1)
\end{equation*}
\begin{equation*}
T_{13}f(x_1,x_2,x_3) = f(x_1+x_2,x_2,x_3+x_2)
\end{equation*}
Then 
\begin{equation}  \label{E:Str5}
\| \beta^{(k)}\|_{X^{(k)}_{-\frac12+}} \lesssim \left\{ \begin{aligned} &\|
(\mathcal{F}_3 T_{12} \beta^{(k)})(t,x_1,x_2,\xi_3) \|_{L_t^2 L_{\xi_3}^2
L_{x_1x_2}^{\frac32+} L^2_c} \\ &\| (\mathcal{F}_2 T_{13}
\beta^{(k)})(t,x_1,\xi_2,x_3) \|_{L_t^2 L_{\xi_2}^2 L_{x_1x_3}^{\frac32+}
L^2_c}\\ &\| (\mathcal{F}_1 T_{23} \beta^{(k)})(t,\xi_1,x_2,x_3) \|_{L_t^2
L_{\xi_1}^2 L_{x_2x_3}^{\frac32+} L^2_c} \end{aligned} \right.
\end{equation}
where in each case $c$ stands for ``complementary coordinates'',
specifically coordinates $(x_4,\ldots, x_k, x_1^{\prime }, \ldots,
x_k^{\prime })$.
\end{lemma}

\begin{proof}
We will only prove the first estimate in \eqref{E:Str5}. The other two
estimates follow in analogy or can be deduced from the first estimate by
permuting coordinates (this does not require symmetry of $\beta^{(k)}$).

\begin{equation}  \label{E:Str14}
(\mathcal{F}_{123}T_{12}\beta^{(k)})(t,\xi_1,\xi_2,\xi_3) = (\mathcal{F}%
_{123}\beta^{(k)})(t,\xi_1,\xi_2,\xi_3-\xi_1-\xi_2)
\end{equation}
Also 
\begin{equation}  \label{E:Str15}
\begin{aligned} \hspace{0.3in}&\hspace{-0.3in}
e^{-2it\xi_1\cdot\xi_3}e^{-2it\xi_2\cdot\xi_3} (\mathcal{F}_{123} T_{12}
\beta^{(k)})(t,\xi_1,\xi_2,\xi_3) \\ &= \mathcal{F}_{12} [ (\mathcal{F}_3
T_{12} \beta^{(k)})(t,x_1-2t\xi_3,x_2-2t\xi_3,\xi_3)](\xi_1,\xi_2)
\end{aligned}
\end{equation}
Now 
\begin{align*}
\hspace{0.3in}&\hspace{-0.3in} (\mathcal{F}_{0123}\beta^{(k)})(\tau-|%
\xi_3|^2+2\xi_1\cdot\xi_3+2\xi_2\cdot\xi_3, \xi_1, \xi_2, \xi_3-\xi_1-\xi_2)
&  \\
&= (\mathcal{F}_{0123}T_{12}\beta^{(k)})(\tau-|\xi_3|^2+2\xi_1\cdot\xi_3+2%
\xi_2\cdot\xi_3,\xi_1,\xi_2,\xi_3) & & \text{by \eqref{E:Str14}} \\
&= \mathcal{F}_0[ e^{it|\xi_3|^2} e^{-2it\xi_1\cdot\xi_3} e^{-2it\xi_2\cdot
\xi_3}(\mathcal{F}_{123}T_{12}\beta^{(k)})(t,\xi_1,\xi_2,\xi_3)](\tau) &  \\
&= \mathcal{F}_0 [ e^{it|\xi_3|^2} \mathcal{F}_{12}[ (\mathcal{F}%
_3T_{12}\beta^{(k)})(t,x_1-2t\xi_3,x_2-2t\xi_3,\xi_3)](\xi_1,\xi_2)](\tau) & 
& \text{by \eqref{E:Str15}} \\
&= \mathcal{F}_{012} [ e^{it|\xi_3|^2} (\mathcal{F}_3T_{12}%
\beta^{(k)})(t,x_1-2t\xi_3,x_2-2t\xi_3,\xi_3)](\tau,\xi_1,\xi_2) & 
\end{align*}
Let 
\begin{equation}  \label{E:Str17}
\sigma_{\xi_3}(t,x_1,x_2) \overset{\mathrm{def}}{=} e^{it|\xi_3|^2} (%
\mathcal{F}_3T_{12}\beta^{(k)})(t,x_1-2t\xi_3,x_2-2t\xi_3,\xi_3)
\end{equation}
where $\xi_3$ is regarded as a fixed parameter. Then we have shown that 
\begin{equation}  \label{E:Str16}
(\mathcal{F}_{012}\sigma_{\xi_3})(\tau,\xi_1,\xi_2) = (\mathcal{F}%
_{0123}\beta^{(k)})(\tau-|\xi_3|^2+2\xi_1\cdot\xi_3+2\xi_2\cdot\xi_3, \xi_1,
\xi_2, \xi_3-\xi_1-\xi_2)
\end{equation}

Now consider 
\begin{equation*}
\| \beta^{(k)} \|_{X_{-\frac12+}} = \| \langle \tau + |\xi_1|^2+|\xi_2|^2
+|\xi_3|^2\rangle^{-\frac12+} (\mathcal{F}_{0123}\beta)(\tau,\xi_1,\xi_2,%
\xi_3) \|_{L^2_{\tau\xi_1\xi_2\xi_3}}
\end{equation*}

Change variable $\xi_3\mapsto \xi_3-\xi_1-\xi_2$ and then $\tau\mapsto \tau
- |\xi_3|^2 + 2\xi_1\cdot\xi_3 + 2\xi_1\cdot \xi_3$ and substitute %
\eqref{E:Str16} to obtain 
\begin{equation*}
\| \beta^{(k)} \|_{X_{-\frac12+}} = \| \langle \tau + 2|\xi_1|^2+2|\xi_2|^2
+2\xi_1\cdot\xi_2 \rangle^{-\frac12+} (\mathcal{F}_{012}\sigma_{\xi_3})(%
\tau,\xi_1,\xi_2) \|_{L^2_{\xi_3} L^2_{\tau\xi_1\xi_2}}
\end{equation*}
By the dual 6D endpoint Strichartz estimate \cite{Keel-Tao} 
\begin{equation*}
\| \beta^{(k)} \|_{X_{-\frac12+}} \lesssim \| \sigma_{\xi_3} \|_{L_{\xi_3}^2
L_t^2L_{x_1x_2}^{\frac32+}}
\end{equation*}
Returning to the definition of $\sigma_{\xi_3}$ given above in %
\eqref{E:Str17}, and changing variable $x_1 \mapsto x_1+2t\xi_3$, $x_2
\mapsto x_2+2t\xi_3$, we obtain 
\begin{equation*}
\| \beta^{(k)} \|_{X_{-\frac12+}} \lesssim \| \mathcal{F}_3
T_{12}\beta^{(k)} \|_{L_{\xi_3}^2 L_t^2L_{x_1x_2}^{\frac32+}}
\end{equation*}
\end{proof}

\begin{lemma}[H\"{o}lder and Sobolev]
\label{lem:three-body Holder soblev}If $\beta ^{(k)}$ has any one of the
following three forms 
\begin{equation}
\beta ^{(k)}(t,x_{1},x_{2},x_{3})=\gamma ^{(k)}(t,x_{1},x_{2},x_{3})\times
\left\{ \begin{aligned} &V(x_1-x_2)W(x_1-x_3)\\ &V(x_1-x_2)W(x_2-x_3)\\
&V(x_1-x_3)W(x_2-x_3) \end{aligned}\right.  \label{E:Str11}
\end{equation}%
then all three of the following estimates hold 
\begin{equation}
\Vert (\mathcal{F}_{3}T_{12}\beta ^{(k)})(t,x_{1},x_{2},\xi _{3})\Vert
_{L_{t}^{2}L_{\xi _{3}}^{2}L_{x_{1}x_{2}}^{\frac{3}{2}+}L_{c}^{2}}\lesssim
\left\{ \begin{aligned} &\|V\|_{L^{\frac32+}}\|W\|_{L^{\frac32+}} \| \langle
\nabla_{x_1}\rangle^{\frac32} \langle \nabla_{x_2}\rangle^{\frac32}
\gamma^{(k)} \|_{L_t^2L_{\mathbf{x}\mathbf{x}'}^2} \\ &\|V\|_{L^{2+}}
\|W\|_{L^{2+}} \| \nabla_{x_1}\nabla_{x_2} \gamma^{(k)}
\|_{L_t^2L_{\mathbf{x}\mathbf{x}'}^2} \\ &\|V\|_{L^{2+}} \|W\|_{L^{6+}} \|
\nabla_{x_1} \gamma^{(k)} \|_{L_t^2L_{\mathbf{x}\mathbf{x}'}^2} \\
&\|V\|_{L^{6+}} \|W\|_{L^{2+}} \| \nabla_{x_2} \gamma^{(k)}
\|_{L_t^2L_{\mathbf{x}\mathbf{x}'}^2} \\ &\|V\|_{L^{6+}}\|W\|_{L^{6+}} \|
\gamma^{(k)} \|_{L_t^2L_{\mathbf{x}\mathbf{x}'}^2} \end{aligned}\right.
\label{E:Str8}
\end{equation}%
\begin{equation}
\Vert (\mathcal{F}_{2}T_{13}\beta ^{(k)})(t,x_{1},\xi _{2},x_{3})\Vert
_{L_{t}^{2}L_{\xi _{2}}^{2}L_{x_{1}x_{3}}^{\frac{3}{2}+}L_{c}^{2}}\lesssim
\left\{ \begin{aligned} &\|V\|_{L^{\frac32+}}\|W\|_{L^{\frac32+}} \| \langle
\nabla_{x_1}\rangle^{\frac32} \langle \nabla_{x_3}\rangle^{\frac32}
\gamma^{(k)} \|_{L_t^2L_{\mathbf{x}\mathbf{x}'}^2} \\ &\|V\|_{L^{2+}}
\|W\|_{L^{2+}} \| \nabla_{x_1}\nabla_{x_3} \gamma^{(k)}
\|_{L_t^2L_{\mathbf{x}\mathbf{x}'}^2} \\ &\|V\|_{L^{2+}} \|W\|_{L^{6+}} \|
\nabla_{x_1} \gamma^{(k)} \|_{L_t^2L_{\mathbf{x}\mathbf{x}'}^2} \\
&\|V\|_{L^{6+}} \|W\|_{L^{2+}} \| \nabla_{x_3} \gamma^{(k)}
\|_{L_t^2L_{\mathbf{x}\mathbf{x}'}^2} \\ &\|V\|_{L^{6+}}\|W\|_{L^{6+}} \|
\gamma^{(k)} \|_{L_t^2L_{\mathbf{x}\mathbf{x}'}^2} \end{aligned}\right.
\label{E:Str9}
\end{equation}%
\begin{equation}
\Vert (\mathcal{F}_{1}T_{23}\beta ^{(k)})(t,\xi _{1},x_{2},x_{3})\Vert
_{L_{t}^{2}L_{\xi _{1}}^{2}L_{x_{2}x_{3}}^{\frac{3}{2}+}L_{c}^{2}}\lesssim
\left\{ \begin{aligned} &\|V\|_{L^{\frac32+}}\|W\|_{L^{\frac32+}} \| \langle
\nabla_{x_2}\rangle^{\frac32} \langle \nabla_{x_3}\rangle^{\frac32}
\gamma^{(k)} \|_{L_t^2L_{\mathbf{x}\mathbf{x}'}^2} \\ &\|V\|_{L^{2+}}
\|W\|_{L^{2+}} \| \nabla_{x_2}\nabla_{x_3} \gamma^{(k)}
\|_{L_t^2L_{\mathbf{x}\mathbf{x}'}^2} \\ &\|V\|_{L^{6+}} \|W\|_{L^{2+}} \|
\nabla_{x_2} \gamma^{(k)} \|_{L_t^2L_{\mathbf{x}\mathbf{x}'}^2} \\
&\|V\|_{L^{2+}} \|W\|_{L^{6+}} \| \nabla_{x_3} \gamma^{(k)}
\|_{L_t^2L_{\mathbf{x}\mathbf{x}'}^2} \\ &\|V\|_{L^{6+}}\|W\|_{L^{6+}} \|
\gamma^{(k)} \|_{L_t^2L_{\mathbf{x}\mathbf{x}'}^2} \end{aligned}\right.
\label{E:Str10}
\end{equation}
\end{lemma}

\begin{proof}
All of the estimates have a similar proof. As an illustrative example,
consider the first estimate of \eqref{E:Str8}. By \eqref{E:Str11}, 
\begin{equation*}
(T_{12}\beta^{(k)})(t,x_1,x_2,x_3) = (T_{12}\gamma^{(k)})(t,x_1,x_2,x_3)
\times \left\{ \begin{aligned} &V(x_1-x_2)W(x_1) \\ &V(x_1-x_2)W(x_2) \\
&V(x_1)W(x_2) \end{aligned} \right.
\end{equation*}
Hence 
\begin{equation*}
(\mathcal{F}_3 T_{12}\beta^{(k)})(t,x_1,x_2,\xi_3) = (\mathcal{F}%
_3T_{12}\gamma^{(k)})(t,x_1,x_2,\xi_3) \times \left\{ \begin{aligned}
&V(x_1-x_2)W(x_1) \\ &V(x_1-x_2)W(x_2) \\ &V(x_1)W(x_2) \end{aligned} \right.
\end{equation*}
By H\"older, 
\begin{equation*}
\| (\mathcal{F}_3
T_{12}\beta^{(k)})(t,x_1,x_2,\xi_3)\|_{L_{x_1x_2}^{\frac32+}L_c^2} \leq
\|V\|_{L^{\frac32+}} \|W\|_{L^{\frac32+}} \| (\mathcal{F}_3T_{12}%
\gamma^{(k)})(t,x_1,x_2,\xi_3)\|_{L_{x_1x_2}^{\infty-}L_c^2}
\end{equation*}
By Sobolev, 
\begin{equation*}
\begin{aligned} \hspace{0.3in}&\hspace{-0.3in} \| (\mathcal{F}_3
T_{12}\beta^{(k)})(t,x_1,x_2,\xi_3)\|_{L_{x_1x_2}^{\frac32+}L_c^2} \\ &\leq
\|V\|_{L^{\frac32+}} \|W\|_{L^{\frac32+}} \| \langle \nabla_{x_1}
\rangle^{\frac32} \langle \nabla_{x_2} \rangle^{\frac32}
(\mathcal{F}_3T_{12}\gamma^{(k)})(t,x_1,x_2,\xi_3)\|_{L_{x_1x_2}^2L_c^2}\\
&= \|V\|_{L^{\frac32+}} \|W\|_{L^{\frac32+}} \| \mathcal{F}_3 \langle
\nabla_{x_1} \rangle^{\frac32} \langle \nabla_{x_2} \rangle^{\frac32}
(T_{12}\gamma^{(k)})(t,x_1,x_2,\xi_3)\|_{L_{x_1x_2}^2L_c^2} \end{aligned}
\end{equation*}
Apply $L^2_{\xi_3}$ and use Plancherel, 
\begin{equation*}
\begin{aligned} \hspace{0.3in}&\hspace{-0.3in} \| (\mathcal{F}_3
T_{12}\beta^{(k)})(t,x_1,x_2,\xi_3)\|_{L_{\xi_3}^2
L_{x_1x_2}^{\frac32+}L_c^2} \\ &= \|V\|_{L^{\frac32+}} \|W\|_{L^{\frac32+}}
\| \langle \nabla_{x_1} \rangle^{\frac32} \langle \nabla_{x_2}
\rangle^{\frac32}
(T_{12}\gamma^{(k)})(t,x_1,x_2,x_3)\|_{L_{\mathbf{x}\mathbf{x}'}^2} \\ &=
\|V\|_{L^{\frac32+}} \|W\|_{L^{\frac32+}} \| \langle \nabla_{x_1}
\rangle^{\frac32} \langle \nabla_{x_2} \rangle^{\frac32}
\gamma^{(k)}(t,x_1,x_2,x_3)\|_{L_{\mathbf{x}\mathbf{x}'}^2} \end{aligned}
\end{equation*}
The other estimates in \eqref{E:Str8} follow by using H\"older differently.
\end{proof}

By splitting up $\gamma^{(k)}$ according to the relative magnitude of
frequencies, we can share derivatives among three coordinates, as in the
following corollary.

\begin{corollary}
\label{corollary:basic corollary in three body} If $\gamma ^{(k)}$ is
symmetric and $\beta ^{(k)}$ has any one of the following three forms 
\begin{equation}
\beta ^{(k)}(t,x_{1},x_{2},x_{3})=\gamma ^{(k)}(t,x_{1},x_{2},x_{3})\times
\left\{ \begin{aligned} &V(x_1-x_2)W(x_1-x_3)\\ &V(x_1-x_2)W(x_2-x_3)\\
&V(x_1-x_3)W(x_2-x_3) \end{aligned}\right.  \label{E:Str13}
\end{equation}%
then 
\begin{equation}  \label{estimate:basic corollary in three body}
\Vert \beta ^{(k)}\Vert _{X_{-\frac{1}{2}+}^{(k)}} \lesssim \left\{ %
\begin{aligned} &\Vert V\Vert _{L^{\frac{3}{2}+}}\Vert W\Vert
_{L^{\frac{3}{2}+}}\Vert \langle \nabla _{x_{1}}\rangle \langle \nabla
_{x_{2}}\rangle \langle \nabla _{x_{3}}\rangle \gamma ^{(k)}\Vert
_{L_{t}^{2}L_{\mathbf{x}\mathbf{x}^{\prime }}^{2}} \\ &\Vert V\Vert
_{L^{2+}}\Vert W\Vert _{L^{2+}}\Vert \langle \nabla _{x_{i}}\rangle \langle
\nabla _{x_{j}}\rangle \gamma ^{(k)}\Vert
_{L_{t}^{2}L_{\mathbf{x}\mathbf{x}^{\prime }}^{2}}\text{ with
}i,j=1,2,3\text{ but }i\neq j \\ &\Vert V\Vert _{L^{2+}}\Vert W\Vert
_{L^{6+}}\Vert \langle \nabla _{x_{i}}\rangle \gamma ^{(k)}\Vert
_{L_{t}^{2}L_{\mathbf{x}\mathbf{x}^{\prime }}^{2}}\text{ with }i=1,2,3 \\
&\Vert V\Vert _{L^{6+}}\Vert W\Vert _{L^{6+}}\Vert \gamma ^{(k)}\Vert
_{L_{t}^{2}L_{\mathbf{x}\mathbf{x}^{\prime }}^{2}} \end{aligned} \right.
\end{equation}
\end{corollary}

\begin{proof}
As in the proof of Corollary \ref{corollary:basic corollary for 3D endpoint}%
, it suffices to prove the first inequality of (\ref{estimate:basic
corollary in three body}) since the other two follows directly from Lemma %
\ref{lem:three-body Holder soblev} and the fact that $\gamma ^{(k)}$ is
symmetric.

Split $\gamma ^{(k)}$ according to whether $\max (|\xi _{1}|,|\xi _{2}|,|\xi
_{3}|)=|\xi _{3}|$, $\max (|\xi _{1}|,|\xi _{2}|,|\xi _{3}|)=|\xi _{2}|$, or 
$\max (|\xi _{1}|,|\xi _{2}|,|\xi _{3}|)=|\xi _{1}|$ 
\begin{equation*}
\gamma ^{(k)}=P_{\substack{ |\xi _{1}|\leq |\xi _{3}|  \\ |\xi _{2}|\leq
|\xi _{3}|}}\gamma ^{(k)}+P_{\substack{ |\xi _{1}|\leq |\xi _{2}|  \\ |\xi
_{3}|\leq |\xi _{2}|}}\gamma ^{(k)}+P_{\substack{ |\xi _{2}|\leq |\xi _{1}| 
\\ |\xi _{3}|\leq |\xi _{1}|}}\gamma ^{(k)}
\end{equation*}%
and define 
\begin{equation*}
\beta _{1,2\leq 3}^{(k)}\overset{\mathrm{def}}{=}P_{\substack{ |\xi
_{1}|\leq |\xi _{3}|  \\ |\xi _{2}|\leq |\xi _{3}|}}\gamma ^{(k)}\times
\left\{ \begin{aligned} &V(x_1-x_2)W(x_1-x_3)\\ &V(x_1-x_2)W(x_2-x_3)\\
&V(x_1-x_3)W(x_2-x_3) \end{aligned}\right.
\end{equation*}%
\begin{equation*}
\beta _{1,3\leq 2}^{(k)}\overset{\mathrm{def}}{=}P_{\substack{ |\xi
_{1}|\leq |\xi _{2}|  \\ |\xi _{3}|\leq |\xi _{2}|}}\gamma ^{(k)}\times
\left\{ \begin{aligned} &V(x_1-x_2)W(x_1-x_3)\\ &V(x_1-x_2)W(x_2-x_3)\\
&V(x_1-x_3)W(x_2-x_3) \end{aligned}\right.
\end{equation*}%
\begin{equation*}
\beta _{2,3\leq 1}^{(k)}\overset{\mathrm{def}}{=}P_{\substack{ |\xi
_{2}|\leq |\xi _{1}|  \\ |\xi _{3}|\leq |\xi _{1}|}}\gamma ^{(k)}\times
\left\{ \begin{aligned} &V(x_1-x_2)W(x_1-x_3)\\ &V(x_1-x_2)W(x_2-x_3)\\
&V(x_1-x_3)W(x_2-x_3) \end{aligned}\right.
\end{equation*}%
so that 
\begin{equation*}
\beta ^{(k)}=\beta _{1,2\leq 3}^{(k)}+\beta _{1,3\leq 2}^{(k)}+\beta
_{2,3\leq 1}^{(k)}
\end{equation*}%
For the $\beta _{1,2\leq 3}^{(k)}$ piece, we use the first estimate of %
\eqref{E:Str5} combined with the first estimate of \eqref{E:Str8} 
\begin{equation*}
\Vert \beta _{1,2\leq 3}^{(k)}\Vert _{X_{-\frac{1}{2}+}^{(k)}}\lesssim \Vert
V\Vert _{L^{\frac{3}{2}+}}\Vert W\Vert _{L^{\frac{3}{2}+}}\Vert \langle
\nabla _{x_{1}}\rangle ^{\frac{3}{2}}\langle \nabla _{x_{2}}\rangle ^{\frac{3%
}{2}}P_{\substack{ |\xi _{1}|\leq |\xi _{3}|  \\ |\xi _{2}|\leq |\xi _{3}|}}%
\gamma ^{(k)}\Vert _{L_{t}^{2}L_{\mathbf{x}\mathbf{x}^{\prime }}^{2}}
\end{equation*}%
By the frequency restriction, we can move $\frac{1}{2}$ derivative in $x_{1}$
to $x_{3}$ and $\frac{1}{2}$ derivative in $x_{2}$ to $x_{3}$ to obtain: 
\begin{equation*}
\Vert \beta _{1,2\leq 3}^{(k)}\Vert _{X_{-\frac{1}{2}+}^{(k)}}\lesssim \Vert
V\Vert _{L^{\frac{3}{2}+}}\Vert W\Vert _{L^{\frac{3}{2}+}}\Vert \langle
\nabla _{x_{1}}\rangle \langle \nabla _{x_{2}}\rangle \langle \nabla
_{x_{3}}\rangle \gamma ^{(k)}\Vert _{L_{t}^{2}L_{\mathbf{x}\mathbf{x}%
^{\prime }}^{2}}
\end{equation*}%
The term $\beta _{1,3\leq 2}^{(k)}$ is handled analogously, using the second
estimate of \eqref{E:Str5} together with the first estimate of \eqref{E:Str9}%
. The term $\beta _{2,3\leq 1}^{(k)}$ is handled using the third estimate of %
\eqref{E:Str5} together with the first estimate of \eqref{E:Str10}.
\end{proof}

\begin{proposition}
\label{P:Str29} For any $U(x)$, let $U_N(x) = N^{3\beta}U(N^\beta x)$. Then 
\begin{equation}  \label{E:Str29}
\|R^{(3)}_{\leq M_3} ( U_N(x_1-x_2)U_N(x_1-x_3) \,
\alpha^{(3)})\|_{X_{-\frac12+}^{(3)}} \lesssim C_U \|\alpha^{(3)}\|_{L_t^2L_{%
\mathbf{x}_3\mathbf{x}_3}^2} \left\{ \begin{aligned} & N^{5\beta+} \\ &
M_3^3 N^{2\beta+} \end{aligned} \right.
\end{equation}
where 
\begin{equation*}
C_U = \|\nabla U\|_{L^{\frac32+}}\|\nabla^2 U \|_{L^{\frac32+}}+\| \nabla
U\|_{L^{\frac32+}}^2 + \|U\|_{L^{\frac32+}}^2
\end{equation*}
\end{proposition}

\begin{proof}
To prove the top estimate of \eqref{E:Str29}, we use do not use the
frequency restriction and distribute all derivatives $\nabla_{x_1}%
\nabla_{x_2}\nabla_{x_3} \nabla_{x_1^{\prime }}\nabla_{x_2^{\prime
}}\nabla_{x_3^{\prime }}$ into the expression. The $\nabla_{x_1}^{\prime
}\nabla_{x_2^{\prime }}\nabla_{x_3^{\prime }}$ move directly onto $%
\alpha^{(3)}$. The expansion of 
\begin{equation}  \label{E:Str28}
\nabla_{x_1}\nabla_{x_2}\nabla_{x_3} \Big( U_N(x_1-x_2)U_N(x_1-x_3)
(\nabla_{x_1^{\prime }}\nabla_{x_2^{\prime }}\nabla_{x_3^{\prime }}
\alpha_N^{(3)}) \Big)
\end{equation}
has $3\times 2 \times 2=12 $ terms total. Each is estimated using different
estimates in \eqref{estimate:basic corollary in three body}. We will not
write out each term, but take some representative examples. Let us consider
the case 
\begin{equation*}
\begin{aligned} \hspace{0.3in}&\hspace{-0.3in} [\nabla_{x_1}\nabla_{x_2} \;
U_N(x_1-x_2)] [\nabla_{x_3} \; U_N(x_1-x_3)]
(\nabla_{x_1'}\nabla_{x_2'}\nabla_{x_3'} \alpha_N^{(3)})\\ & = -N^{3\beta}
(\nabla^2 U)_N(x_1-x_2) (\nabla U)_N(x_1-x_3) \alpha_N^{(3)} \end{aligned}
\end{equation*}
Apply the first estimate of \eqref{estimate:basic corollary in three body}
to obtain 
\begin{equation*}
\begin{aligned} \hspace{0.3in}&\hspace{-0.3in} \| [\nabla_{x_1}\nabla_{x_2}
\; U_N(x_1-x_2)] [\nabla_{x_3} \; U_N(x_1-x_3)]
(\nabla_{x_1'}\nabla_{x_2'}\nabla_{x_3'} \alpha_N^{(3)}) \|_{X_{-\frac12+}}
\\ &\lesssim N^{3\beta} \|(\nabla^2 U)_N \|_{L^{\frac32+}} \| (\nabla U)_N
\|_{L^{\frac32+}}
\|S^{(3)}\alpha^{(3)}\|_{L_t^2L_{\mathbf{x}_3\mathbf{x}_3'}^2} \\ &\lesssim
N^{5\beta+} \| \nabla^2 U \|_{L^{\frac32+}} \| \nabla U\|_{L^{\frac32+}}
\|S^{(3)}\alpha^{(3)}\|_{L_t^2L_{\mathbf{x}_3\mathbf{x}_3'}^2} \end{aligned}
\end{equation*}
Another term resulting from the expansion of \eqref{E:Str28} is 
\begin{equation*}
U_N(x_1-x_2) \, U_N(x_1-x_3) (\nabla_{x_1}\nabla_{x_2}\nabla_{x_3}
\nabla_{x_1^{\prime }}\nabla_{x_2^{\prime }}\nabla_{x_3^{\prime }}
\alpha_N^{(3)}))
\end{equation*}
In this case, we apply the fourth estimate of 
\eqref{estimate:basic corollary in
three body} to obtain 
\begin{equation*}
\begin{aligned} \hspace{0.3in}&\hspace{-0.3in} \| U_N(x_1-x_2) \,
U_N(x_1-x_3) (\nabla_{x_1}\nabla_{x_2}\nabla_{x_3}
\nabla_{x_1'}\nabla_{x_2'}\nabla_{x_3'} \alpha_N^{(3)}))\|_{X_{-\frac12+}}
\\ &\lesssim \| U_N\|_{L^{6+}}^2\|S^{(3)} \alpha^{(3)}
\|_{L_t^2L_{\mathbf{x}_3\mathbf{x}_3'}^2} \\ &\lesssim N^{5\beta+} \|
U\|_{L^{6+}}^2 \|S^{(3)} \alpha^{(3)}
\|_{L_t^2L_{\mathbf{x}_3\mathbf{x}_3'}^2} \end{aligned}
\end{equation*}
To prove the bottom estimate of \eqref{E:Str29}, we use the frequency
restriction $R^{(3)}_{\leq M_3} \leq M_3^3 \nabla_{x_1^{\prime
}}\nabla_{x_2^{\prime }}\nabla_{x_3^{\prime }}$, and all of the $%
\nabla_{x_1^{\prime }}\nabla_{x_2^{\prime }}\nabla_{x_3^{\prime }}$
derivatives move directly onto $\alpha^{(3)}$. One then estimates using the
first estimate of \eqref{estimate:basic corollary in three body} to obtain 
\begin{equation*}
\begin{aligned} \|R^{(3)}_{\leq M_3} ( U_N(x_1-x_2)U_N(x_1-x_3) \,
\alpha^{(3)})\|_{X_{-\frac12+}^{(3)}} &\lesssim M_3^3
\|U_N\|_{L_{\frac32+}}^2 \|S^{(3)} \alpha_N^{(3)}
\|_{L_t^2L_{\mathbf{x}_3\mathbf{x}_3'}^2} \\ &\lesssim
M_3^3N^{2\beta}\|U\|_{L_{\frac32+}}^2 \|S^{(3)} \alpha_N^{(3)}
\|_{L_t^2L_{\mathbf{x}_3\mathbf{x}_3'}^2} \end{aligned}
\end{equation*}
\end{proof}

For the KIP estimates, we provide the following lemma and Proposition \ref%
{P:Str22}.

\begin{lemma}
\label{Lem:basic lem in KIP}%
\begin{equation}
\begin{aligned} \hspace{0.3in}&\hspace{-0.3in} \left\| \int_{x_4}
V(x_2-x_4)W(x_3-x_4)f(x_1-x_4)\alpha^{(4)}(x_1,x_2,x_3,x_4;
x_1',x_2',x_3',x_4) \, dx_4 \right\|_{X^{(3)}_{-\frac12+}} \\ &\lesssim
\left\{ \begin{aligned} &\|V\|_{L^{1+}} \|W\|_{L^{\frac32+}}
\|f\|_{L^{\infty}} \| \langle \nabla_{x_3} \rangle \langle
\nabla_{x_4}\rangle \langle \nabla_{x_4'}) (\langle \nabla_{x_4} \rangle +
\langle \nabla_{x_4'} \rangle )\alpha^{(4)}
\|_{L_t^2L_{\mathbf{x}\mathbf{x}'}^2} \\ &\|V\|_{L^{1+}}
\|W\|_{L^{\frac32+}} \|f\|_{L^3} \| \langle \nabla_{x_1} \rangle \langle
\nabla_{x_3} \rangle \langle \nabla_{x_4}\rangle \langle \nabla_{x_4'})
(\langle \nabla_{x_4} \rangle + \langle \nabla_{x_4'} \rangle
)\alpha^{(4)}\|_{L_t^2L_{\mathbf{x}\mathbf{x}'}^2} \\ &\|V\|_{L^{1+}}
\|W\|_{L^{3+}} \|f\|_{L^{\infty}} \|\langle \nabla_{x_4}\rangle \langle
\nabla_{x_4'}) (\langle \nabla_{x_4}\rangle + \langle \nabla_{x_4'} \rangle
)\alpha^{(4)} \|_{L_t^2L_{\mathbf{x}\mathbf{x}'}^2} \\ &\|V\|_{L^{1+}}
\|W\|_{L^{3+}} \|f\|_{L^3} \| \langle \nabla_{x_1}\rangle \langle
\nabla_{x_4}\rangle \langle \nabla_{x_4'}) (\langle \nabla_{x_4}\rangle +
\langle \nabla_{x_4'} \rangle )\alpha^{(4)}
\|_{L_t^2L_{\mathbf{x}\mathbf{x}'}^2} \end{aligned} \right. \end{aligned}
\label{E:Str20}
\end{equation}
\end{lemma}

\begin{proof}
Change variables $x_4\mapsto x_4+x_2$ to get 
\begin{equation*}
\begin{aligned} \hspace{0.3in}&\hspace{-0.3in}
\beta^{(3)}(t,x_1,x_2,x_3;x_1',x_2',x_3') \\ &\overset{\mathrm{def}}{=}
\int_{x_4} V(x_2-x_4)W(x_3 - x_4)f(x_4-x_1)\alpha^{(4)}(x_1,x_2,x_3,x_4;
x_1', x_2^{\prime },x_3^{\prime},x_4) \, dx_4 \\ &= \int_{x_4}
V(-x_4)W(x_3-x_2-x_4)f(x_4+x_2-x_1)\alpha^{(4)}(x_1,x_2,x_3,x_4+x_2; x_1',
x_2^{\prime },x_3^{\prime },x_4+x_2) \, dx_4 \end{aligned}
\end{equation*}
Let us, for notational convenience, write 
\begin{equation}  \label{E:Str18a}
\tilde \sigma \overset{\mathrm{def}}{=}
\alpha^{(4)}(x_1,x_2,x_3+x_2,x_4+x_2;x_1^{\prime }, x_2^{\prime
},x_3^{\prime },x_4+x_2)
\end{equation}
\begin{equation}  \label{E:Str18}
\begin{aligned} \sigma &\stackrel{\rm{def}}{=} f(x_4+x_2-x_1) \tilde \sigma
\\ &= f(x_4+x_2-x_1)\alpha^{(4)}(x_1,x_2,x_3+x_2,x_4+x_2;x_1', x_2^{\prime
},x_3^{\prime},x_4+x_2) \end{aligned}
\end{equation}
Recalling that $T_3g(x_2,x_3) = g(x_2,x_3+x_2)$ (as in the proof of Lemma %
\ref{Lemma:TheStrichartzEstimate}), we have 
\begin{equation*}
(T_3\beta^{(3)})(t,x_1,x_2,x_3;x_1^{\prime },x_2^{\prime },x_3^{\prime }) =
\int_{x_4}V(-x_4)W(x_3-x_4) \sigma \, dx_4
\end{equation*}
By \eqref{E:Str1}, 
\begin{equation*}
\| \beta^{(3)} \|_{X_{-\frac12+}^{(3)}} \lesssim \|\mathcal{F}_2 T_3
\beta^{(2)}\|_{L^2_{\xi_2} L_t^2 L_{x_3}^{\frac65+} L_{x_1x_1^{\prime
}x_2^{\prime }x_3^{\prime }}^2}
\end{equation*}
Moving the $L^2_{\xi_2}$ norm to the inside (by Minkowski's integral
inequality) and applying Plancherel to convert $L^2_{\xi_2}$ to $L^2_{x_2}$ 
\begin{equation*}
\| \beta^{(2)} \|_{X_{-\frac12+}^{(3)}} \lesssim \| T_3 \beta^{(2)}
\|_{L_t^2 L_{x_3}^{\frac65+} L_{x_1x_2x_1^{\prime }x_2^{\prime }x_3^{\prime
}}^2}
\end{equation*}
By Minkowski's integral inequality, 
\begin{equation*}
\lesssim \int_{x_4} |V(-x_4)| \| W(x_3-x_4) \, \sigma \|_{L_t^2
L_{x_3}^{\frac65+}L_{x_1x_2x_1^{\prime }x_2^{\prime }x_3^{\prime }}^2} \,dx_4
\end{equation*}
\begin{equation*}
=\int_{x_4} |V(-x_4)| \| W(x_3-x_4) \, \| \sigma \|_{L^2_{x_1x_2x_1^{\prime
}x_2^{\prime }x_3^{\prime }}} \, \|_{L_t^2 L_{x_3}^{\frac65+}} \, dx_4
\end{equation*}
At this point, recalling \eqref{E:Str18a}, \eqref{E:Str18}, we either
estimate the inside term as 
\begin{equation}  \label{E:Str31}
\|\sigma\|_{L_{x_1}^2} \leq \|f\|_{L^\infty} \|\tilde \sigma\|_{L_{x_1}^2}
\end{equation}
which leads to the first and third estimates of \eqref{E:Str20} or we
estimate using H\"older and Sobolev 
\begin{equation*}
\| \sigma \|_{L_{x_1}^2} \leq \|f\|_{L^3} \| \tilde \sigma \|_{L_{x_1}^6}
\lesssim \|f\|_{L^3} \|\nabla_{x_1} \tilde \sigma \|_{L_{x_1}^2}
\end{equation*}
which leads to the second and fourth estimates of \eqref{E:Str20}. Since the
remaining steps are similar in either case, we will content ourselves to use %
\eqref{E:Str31} and prove the first and third estimates of \eqref{E:Str20}
below.

We next apply H\"older in $x_3$. For the first estimate of \eqref{E:Str20},
we use $\frac56=\frac23+\frac16$, and for the third estimate of %
\eqref{E:Str20} we use $\frac56=\frac13+\frac12$. Let us proceed with the
proof of the first estimate in \eqref{E:Str20} 
\begin{equation*}
\lesssim \| W\|_{L^{\frac32+}} \|f\|_{L^\infty} \int_{x_4} |V(-x_4)| \|
\tilde \sigma \|_{L_t^2L_{x_3}^6 L^2_{x_1x_2x_1^{\prime }x_2^{\prime
}x_3^{\prime }}} \,\, dx_4
\end{equation*}
\begin{equation*}
\lesssim \| W\|_{L^{\frac32+}} \|f\|_{L^\infty} \int_{x_4} |V(-x_4)| \, \|
\tilde \sigma \|_{L_t^2L^2_{x_1x_2 x_1^{\prime }x_2^{\prime }x_3^{\prime
}}L_{x_3}^6} \, \, dx_4
\end{equation*}
By Sobolev in $x_3$, 
\begin{equation*}
\lesssim \| W\|_{L^{\frac32+}}\|f\|_{L^\infty} \int_{x_4} |W(-x_4)| \, \|
\langle \nabla_{x_3}\rangle \, \tilde \sigma \|_{L_t^2
L^2_{x_1x_2x_1^{\prime }x_2^{\prime }x_3^{\prime }}L_{x_3}^2} \, \, dx_4
\end{equation*}
Now apply H\"older in $x_4$ to obtain 
\begin{equation*}
\lesssim \|W\|_{L^{\frac32+}} \|V\|_{L^{1+}} \|f\|_{L^\infty}\| \langle
\nabla_{x_3}\rangle \tilde \sigma \|_{L_{x_4}^{\infty-}L_{t
x_2x_3x_2^{\prime }x_3^{\prime }}^2}
\end{equation*}
\begin{equation*}
\lesssim \|W\|_{L^{\frac32+}} \|V\|_{L^{1+}} \|f\|_{L^\infty} \| \langle
\nabla_{x_3}\rangle \tilde \sigma \|_{L_{t x_2x_3x_2^{\prime }x_3^{\prime
}}^2L_{x_4}^{\infty-}}
\end{equation*}
Apply Sobolev in $x_4$ to obtain 
\begin{equation*}
\lesssim \|W\|_{L^{\frac32+}} \|V\|_{L^{1+}} \|f\|_{L^{\infty}} \|\langle
\nabla_{x_4}\rangle^{\frac32-} \langle \nabla_{x_3}\rangle \, \tilde \sigma
\|_{L_{tx_2x_3x_2^{\prime }x_3^{\prime }}^2L_{x_4}^2}
\end{equation*}
Changing variable $x_3\mapsto x_3-x_2$ and $x_4\mapsto x_4-x_2$, 
\begin{equation*}
= \|W\|_{L^{\frac32+}} \|V\|_{L^{1+}} \|f\|_{L^{\infty}}\|\langle
\nabla_{x_4}\rangle^{\frac32-} \langle
\nabla_{x_3}\rangle\alpha^{(4)}(x_1,x_2,x_3,x_4;x_1^{\prime }, x_2^{\prime
},x_3^{\prime },x_4) \|_{L_{t x_2x_3x_2^{\prime }x_3^{\prime }}^2L_{x_4}^2}
\end{equation*}
By standard trace estimates, we complete the proof of the first estimate in %
\eqref{E:Str20}.
\end{proof}

\begin{proposition}
\label{P:Str22} 
\begin{equation}  \label{E:Str22}
\begin{aligned} \hspace{0.3in}&\hspace{-0.3in} \left\| R_{\leq M_3}^{(3)}
\int_{x_4} V_N(x_2-x_4) w_N(x_3-x_4) f_N(x_1-x_4) \alpha^{(4)} \, dx_4
\right\|_{X_{-\frac12+}^{(3)}} \\ & \lesssim C_{V,w_0,f} \left( N^{-\frac12}
\|S_4 S^{(4)} \alpha^{(4)} \|_{L_t^2L_{\mathbf{x}_4\mathbf{x_4}'}^2} \right)
\left\{ \begin{aligned} & M_3 N^{-\frac12+} \\ & N^{\beta-\frac12+}
\end{aligned} \right. \end{aligned}
\end{equation}
where $V_N(x) = N^{3\beta}V(N^\beta x)$, $w_N(x) = N^{\beta-1} w_0(N^\beta
x) $, $f_N(x_1-x_4) = w_0(N^\beta(x_1-x_4))$ or $f_N(x_1-x_4)=1$, and 
\begin{equation*}
C_{V,w_0} = (\|V\|_{L^{1+}}+\|\nabla V\|_{L^{1+}} +
\|w_0\|_{L^{3+}}+\|\nabla w_0\|_{L^{\frac32+}})(1+ \|w_0\|_{L^{\infty}} +
\|\nabla w_0\|_{L^3})
\end{equation*}
(which is finite and independent of $N$).
\end{proposition}

\begin{proof}
Let\footnote{%
We write the operator $R^{(k)}$ using true derivatives $\nabla$ rather than $%
|\nabla|$. Once the $X_{-\frac12+}$ norm is applied, one can be converted to
the other.} 
\begin{equation}  \label{E:Str23}
\begin{aligned} \hspace{0.3in}&\hspace{-0.3in}
\beta^{(3)}(x_1,x_2,x_3;x_1',x_2',x_3') \stackrel{\rm{def}}{=} R_{\leq
M_3}^{(3)} \int_{x_4} V_N(x_2-x_4) w_N(x_3-x_4) f_N(x_1-x_4)\alpha^{(4)} \,
dx_4 \\ &= P_{\leq M_3}^{(3)} \nabla_{x_2}\nabla_{x_3}\int_{x_4}
V_N(x_2-x_4) w_N(x_3-x_4) \nabla_{x_1} [f_N(x_1-x_4) \nabla_{x_1'}
\nabla_{x_2'}\nabla_{x_3'} \alpha^{(4)}] \, dx_4 \end{aligned}
\end{equation}
If $\nabla_{x_1}$ lands on $f_N(x_1-x_4)$, then we ultimately use the second
or fourth estimate in \eqref{E:Str20}. If, on the other hand, $\nabla_{x_1}$
lands on $\alpha^{(4)}$, then we ultimately use either the first or third
estimate in \eqref{E:Str20}. Since the two cases are similar, we will just
proceed assuming that $\nabla_{x_1}$ lands on $\alpha^{(4)}$. Then %
\eqref{E:Str22} is the two estimates: 
\begin{equation}  \label{E:Str21}
\| \beta^{(3)} \|_{X_{-\frac12+}} \lesssim C_{V,w_0} \left( N^{-\frac12}
\|S_4 S^{(4)} \alpha^{(4)} \|_{L_t^2L_{\mathbf{x}_4\mathbf{x_4}^{\prime
}}^2} \right) \left\{ \begin{aligned} & M_2 N^{-\frac12+} \\ &
N^{\beta-\frac12+} \end{aligned} \right.
\end{equation}

We begin by proving the first estimate in \eqref{E:Str21}. Distributing the $%
\nabla_{x_3}$ derivative into the integral, we obtain two terms: 
\begin{equation*}
\beta^{(3)} = \mathrm{A}+\mathrm{B}
\end{equation*}
where 
\begin{equation*}
\mathrm{A} = N^\beta P_{\leq M_3}^{(3)} \nabla_{x_2}\int_{x_4} V_N(x_2-x_4)
(\nabla_{x_3} w)_N(x_3-x_4) f_N(x_1-x_4) \nabla_{x_1} \nabla_{x_1^{\prime }}
\nabla_{x_2^{\prime }}\nabla_{x_3^{\prime }} \alpha^{(4)} \, dx_4
\end{equation*}
\begin{equation*}
\mathrm{B} = P_{\leq M_3}^{(3)} \nabla_{x_2} \int_{x_4} V_N(x_2-x_4)
w_N(x_3-x_4) f_N(x_1-x_4) \nabla_{x_1}\nabla_{x_3}\nabla_{x_1^{\prime }}
\nabla_{x_2^{\prime }}\nabla_{x_3^{\prime }}\alpha^{(4)} \, dx_4
\end{equation*}
Now use that $P_{\leq M_3}^{(3)} \nabla_{x_2} \leq M_3$ to obtain 
\begin{equation*}
\| \mathrm{A}\|_{X_{-\frac12+}} \lesssim M_3 N^\beta \left\| \int_{x_4}
V_N(x_2-x_4) (\nabla_{x_3} w)_N(x_3-x_4) f_N(x_1-x_4)
\nabla_{x_1}\nabla_{x_1^{\prime }}\nabla_{x_2^{\prime }}\nabla_{x_3^{\prime
}} \alpha^{(4)} \, dx_4 \right\|_{X_{-\frac12+}^{(2)}}
\end{equation*}
By the first estimate in \eqref{E:Str20}, 
\begin{equation*}
\|\mathrm{A} \|_{X_{-\frac12+}} \lesssim M_3 N^{-1+} \|V\|_{L^{1+}}\|\nabla
w_0\|_{L^{\frac32+}}\|w_0\|_{L^\infty} \left\| (S^4+S^{4^{\prime }}) \frac{%
S^{(4)}}{\langle \nabla_{x_2}\rangle} \alpha^{(4)} \right\|_{L_t^2L_{\mathbf{%
x}_3\mathbf{x}_3^{\prime }}^2}
\end{equation*}
Again, using that $P_{\leq M_3}^{(2)} \nabla_{x_2} \leq M_3$, we obtain 
\begin{equation*}
\| \mathrm{B} \|_{X_{-\frac12+}} \lesssim M_3 \left\| \int_{x_4}
V_N(x_2-x_4) w_N(x_3-x_4) f_N(x_1-x_4) \nabla_{x_1}\nabla_{x_3}
\nabla_{x_1^{\prime }}\nabla_{x_2^{\prime }} \nabla_{x_3^{\prime }}
\alpha^{(4)} \, dx_4 \right\|_{X_{-\frac12+}^{(2)}}
\end{equation*}
By the third estimate in \eqref{E:Str20}, 
\begin{equation*}
\| \mathrm{B} \|_{X_{-\frac12+}} \lesssim M_3 N^{-1+} \|V\|_{L^{1+}}
\|w_0\|_{L^{3+}}\|w_0\|_{L^\infty} \left\| (S^4+S^{4^{\prime }}) \frac{%
S^{(4)}}{\langle \nabla_{x_2}\rangle} \alpha^{(4)} \right\|_{L_t^2L_{\mathbf{%
x}_3\mathbf{x}_3^{\prime }}^2}
\end{equation*}
Combining the above estimates for terms A and B, we obtain the first
estimate of \eqref{E:Str21}.

For the second estimate in \eqref{E:Str21}, starting from \eqref{E:Str23},
we distribute both $\nabla _{x_{2}}$ and $\nabla _{x_{3}}$ into the
integral. The result is four terms 
\begin{equation*}
\beta ^{(2)}=\mathrm{C}+\mathrm{D}+\mathrm{E}+\mathrm{F}
\end{equation*}%
where 
\begin{equation*}
\mathrm{C}=N^{2\beta }P_{\leq M_{3}}^{(2)}\int_{x_{4}}(\nabla
V)_{N}(x_{2}-x_{4})(\nabla w)_{N}(x_{3}-x_{4})f_{N}(x_{1}-x_{4})\nabla
_{x_{1}} \nabla_{x_1^{\prime }} \nabla _{x_{2}^{\prime }}\nabla
_{x_{3}^{\prime }}\alpha ^{(4)}\,dx_{4}
\end{equation*}%
\begin{equation*}
\mathrm{D}=N^{\beta }P_{\leq
M_{3}}^{(2)}\int_{x_{4}}V_{N}(x_{2}-x_{4})(\nabla
w)_{N}(x_{3}-x_{4})f_{N}(x_{1}-x_{4})\nabla _{x_{1}}\nabla
_{x_{2}}\nabla_{x_1^{\prime }} \nabla_{x_{2}^{\prime }}\nabla
_{x_{3}^{\prime }}\alpha ^{(4)}\,dx_{4}
\end{equation*}%
\begin{equation*}
\mathrm{E}=N^{\beta }P_{\leq M_{3}}^{(2)}\int_{x_{4}}(\nabla
V)_{N}(x_{2}-x_{4})w_{N}(x_{3}-x_{4})f_{N}(x_{1}-x_{4})\nabla _{x_{1}}\nabla
_{x_{3}} \nabla_{x_1^{\prime }} \nabla _{x_{2}^{\prime }}\nabla
_{x_{3}^{\prime }}\alpha ^{(4)}\,dx_{4}
\end{equation*}%
\begin{equation*}
\mathrm{F}=P_{\leq
M_{3}}^{(2)}%
\int_{x_{4}}V_{N}(x_{2}-x_{4})w_{N}(x_{3}-x_{4})f_{N}(x_{1}-x_{4})\nabla
_{x_{1}}\nabla _{x_{2}}\nabla _{x_{3}} \nabla_{x_1^{\prime }} \nabla
_{x_{2}^{\prime }}\nabla _{x_{3}^{\prime }}\alpha ^{(4)}\,dx_{4}
\end{equation*}%
For C and D, we use the first estimate of \eqref{E:Str20}, and for E and F,
we use the second estimate of \eqref{E:Str20}. This gives 
\begin{equation*}
\Vert \mathrm{C}\Vert _{X_{-\frac{1}{2}+}^{(k)}}\lesssim N^{\beta -\frac{1}{2%
}+}\Vert w_{0}\Vert _{L^{\infty }}\Vert \nabla V\Vert _{L^{1+}}\Vert \nabla
w_{0}\Vert _{L^{\frac{3}{2}+}}\Vert w_{0}\Vert _{L^{\infty }}\left( N^{-%
\frac{1}{2}}\Vert (S_{4}+S_{4^{\prime }})\frac{S^{(4)}}{\langle \nabla
_{x_{2}}\rangle }\alpha ^{(4)}\Vert _{L_{t}^{2}L_{\mathbf{x}_{4}\mathbf{x}%
_{4}^{\prime }}^{2}}\right)
\end{equation*}%
\begin{equation*}
\Vert \mathrm{D}\Vert _{X_{-\frac{1}{2}+}^{(k)}}\lesssim N^{-\frac{1}{2}%
+}\Vert V\Vert _{L^{1+}}\Vert \nabla w_{0}\Vert _{L^{\frac{3}{2}+}}\Vert
w_{0}\Vert _{L^{\infty }}\left( N^{-\frac{1}{2}}\Vert (S_{4}+S_{4^{\prime
}})S^{(4)}\alpha ^{(4)}\Vert _{L_{t}^{2}L_{\mathbf{x}_{4}\mathbf{x}%
_{4}^{\prime }}^{2}}\right)
\end{equation*}%
\begin{equation*}
\Vert \mathrm{E}\Vert _{X_{-\frac{1}{2}+}^{(k)}}\lesssim N^{\beta -\frac{1}{2%
}+}\Vert \nabla V\Vert _{L^{1+}}\Vert w_{0}\Vert _{L^{3+}}\Vert w_{0}\Vert
_{L^{\infty }}\left( N^{-\frac{1}{2}}\Vert (S_{4}+S_{4^{\prime }})\frac{%
S^{(4)}}{\langle \nabla _{x_{2}}\rangle }\alpha ^{(4)}\Vert _{L_{t}^{2}L_{%
\mathbf{x}_{4}\mathbf{x}_{4}^{\prime }}^{2}}\right)
\end{equation*}%
\begin{equation*}
\Vert \mathrm{F}\Vert _{X_{-\frac{1}{2}+}^{(k)}}\lesssim N^{-\frac{1}{2}%
+}\Vert V\Vert _{L^{1+}}\Vert w_{0}\Vert _{L^{3+}}\Vert w_{0}\Vert
_{L^{\infty }}\left( N^{-\frac{1}{2}}\Vert (S_{4}+S_{4^{\prime
}})S^{(4)}\alpha ^{(4)}\Vert _{L_{t}^{2}L_{\mathbf{x}_{4}\mathbf{x}%
_{4}^{\prime }}^{2}}\right)
\end{equation*}%
Pulling these together gives the second estimate in \eqref{E:Str21}.
\end{proof}

\appendix

\section{The Topology on the Density Matrices\label{appendix:ESYTopology}}

In this appendix, we define a topology $\tau _{prod}$ on the density
matrices as was previously done in \cite{E-E-S-Y1, E-Y1,
E-S-Y1,E-S-Y2,E-S-Y5,
E-S-Y3,Kirpatrick,TChenAndNP,ChenAnisotropic,Chen3DDerivation,C-H3Dto2D,C-H2/3,C-HFocusing,C-HFocusingII,C-PUniqueness}%
.

Denote the space of Hilbert-Schmidt operators on $L^{2}\left( \mathbb{R}%
^{3k}\right) $ as $\mathcal{L}_{k}^{2}$. Then $\left( \mathcal{L}%
_{k}^{2}\right) ^{\prime }=\mathcal{L}_{k}^{2}$. By the fact that $\mathcal{L%
}_{k}^{2}$ is separable, we select a dense countable subset $%
\{J_{i}^{(k)}\}_{i\geqslant 1}\subset \mathcal{L}_{k}^{2}$ in the unit ball
of $\mathcal{L}_{k}^{2}$ (so $\Vert J_{i}^{(k)}\Vert _{\func{op}}\leqslant 1$
where $\left\Vert \cdot \right\Vert _{\func{op}}$ is the operator norm). For 
$\gamma ^{(k)},\tilde{\gamma}^{(k)}\in \mathcal{L}_{k}^{2}$, we then define
a metric $d_{k}$ on $\mathcal{L}_{k}^{2}$ by 
\begin{equation*}
d_{k}(\gamma ^{(k)},\tilde{\gamma}^{(k)})=\sum_{i=1}^{\infty
}2^{-i}\left\vert \left\langle J_{i}^{(k)},\left( \gamma ^{(k)}-\tilde{\gamma%
}^{(k)}\right) \right\rangle \right\vert .
\end{equation*}%
A uniformly bounded sequence $\gamma _{N}^{(k)}\in \mathcal{L}_{k}^{2}$
converges to $\gamma ^{(k)}\in \mathcal{L}_{k}^{2}$ with respect to the weak
topology if and only if 
\begin{equation*}
\lim_{N}d_{k}(\gamma _{N}^{(k)},\gamma ^{(k)})=0.
\end{equation*}%
For fixed $T>0$, let $C\left( \left[ 0,T\right] ,\mathcal{L}_{k}^{2}\right) $
be the space of functions of $t\in \left[ 0,T\right] $ with values in $%
\mathcal{L}_{k}^{2}$ which are continuous with respect to the metric $d_{k}.$
On $C\left( \left[ 0,T\right] ,\mathcal{L}_{k}^{2}\right) ,$ we define the
metric 
\begin{equation*}
\hat{d}_{k}(\gamma ^{(k)}\left( \cdot \right) ,\tilde{\gamma}^{(k)}\left(
\cdot \right) )=\sup_{t\in \left[ 0,T\right] }d_{k}(\gamma ^{(k)}\left(
t\right) ,\tilde{\gamma}^{(k)}\left( t\right) ).
\end{equation*}%
We can then define a topology $\tau _{prod}$ on the space $\oplus
_{k\geqslant 1}C\left( \left[ 0,T\right] ,\mathcal{L}_{k}^{2}\right) $ by
the product of topologies generated by the metrics $\hat{d}_{k}$ on $C\left( %
\left[ 0,T\right] ,\mathcal{L}_{k}^{2}\right) .$

\end{document}